\documentclass[12pt, reqno]{amsart}
\usepackage{amsmath, amsthm, amscd, amsfonts, amssymb, graphicx, color,enumerate,eucal,multirow,makecell,threeparttable}
\usepackage{enumitem}
\usepackage{xcolor}
\usepackage{extarrows}
\usepackage{multirow}
\usepackage{booktabs}
\newcommand{\brackbinom}[2]{\genfrac{[}{]}{0pt}{}{#1}{#2}}
\usepackage[bookmarksnumbered, colorlinks, plainpages]{hyperref}
\hypersetup{colorlinks=true,linkcolor=red, anchorcolor=green, citecolor=cyan, urlcolor=red, filecolor=magenta, pdftoolbar=true}

\newtheorem{theorem}{Theorem}[section]
\newtheorem{lemma}[theorem]{Lemma}
\newtheorem{proposition}[theorem]{Proposition}
\newtheorem{corollary}[theorem]{Corollary}
\theoremstyle{definition}
\newtheorem{definition}[theorem]{Definition}
\newtheorem{example}[theorem]{Example}

\theoremstyle{remark}
\newtheorem{remark}[theorem]{Remark}
\numberwithin{equation}{section}

\begin{document}

\setcounter{page}{1}

\title[]{Brown-Halmos type theorems for generalized Cauchy singular integral operators and applications}

\author[Yuanqi. Sang \MakeLowercase{and} Liankuo Zhao]{Yuanqi. Sang,$^1$ \MakeLowercase{and} Liankuo Zhao $^2$}

\address{$^{1}$School of Mathematics, Southwestern University of Finance and Economics, Chengdu, 611130, People’s Republic of China}
\email{\textcolor[rgb]{0.00,0.00,0.84}{sangyq@swufe.edu.cn}}

\address{$^{2}$School of Mathematics and Computer Science, Shanxi Normal University, Taiyuan, 030031, People’s Republic of China}
\email{\textcolor[rgb]{0.00,0.00,0.84}{lkzhao@sxnu.edu.cn}}


\let\thefootnote\relax\footnote{}
\keywords{generalized Cauchy singular integral operators, Toeplitz operators, Hankel operators, commutativity}
\subjclass[2020]{Primary 47B35; Secondary 47A08.}

\begin{abstract}
We investigate the commutativity and semi-commutativity of generalized singular integral operators of the form
\[
P_{+} f P_{+} + P_{-} g P_{+} + P_{+} u P_{-} + P_{-} v P_{-}
\]
on $L^{2}$,  where $P_{+}$ denotes the Riesz projection and  $P_{-}=I-P_{+}.$
Building on this analysis, we develop a unified approach to studying the algebraic properties 
of operator classes on $L^{2}$ that are generated by multiplication operators together with the Riesz projection.
 These classes include, but are not limited to, 
 Toeplitz+Hankel operators, singular integral operators, 
 Foguel--Hankel operators, and asymmetric dual truncated Toeplitz operators. 
 We provide complete characterizations of (i) the quasinormality of singular integral operators; 
 and (ii) the necessary and sufficient conditions under which the product of two asymmetric dual truncated Toeplitz operators 
 is again an asymmetric dual truncated Toeplitz operator. 
 In addition, our methods provide new proofs of several known results, including the classical Brown–Halmos theorems 
 and the commutativity of Hankel operators, singular integral operators,
and dual truncated Toeplitz operators. We also improved the conditions for the normality of singular integral operators.
\end{abstract}

\maketitle

\section{Introduction}
The Brown-Halmos theorems \cite{Brown1963} explore when the product of two Toeplitz operators on the Hardy space
is another Toeplitz operator, as well as when two Toeplitz operators commute.
These results play a crucial role in the spectral theory of Toeplitz operators \cite{douglas1998banach,bottcher2013analysis}.
Many studies involve extending the Brown-Halmos theorem to other function spaces, including spaces of analytic or harmonic functions.

As a dilation of Toeplitz operators, singular integral operators have deep theoretical roots
in  one-dimensional singular integral
equations \cite{siomikhlin,Gohberg1992a}.
In the research of operator theory, some operators closely related to singular integral
operators have been discovered,
such as Foguel-Hankel operators \cite{Foguel1964,Aleksandrov1996,Pisier1997,Davidson1997}, 
dual truncated Toeplitz operator \cite{Sang2018,Camara2020,gu2021,camara2022asymmetric}, 
$S\oplus S^{*}$\cite{camara2021dual,timotin2020invariant,gu2022invariant}, etc.
Generalized Cauchy singular integral operators extended these operators.
To better understand the algebraic and spectral properties of these operators, 
this paper investigates when the product of two generalized Cauchy singular integral operators 
remains a generalized Cauchy singular integral operator and when such operators commute.

Let $L^{2}$ be the Hilbert space of all square-integrable functions on the unit circle $\mathbb{T}=\{z\in\mathbb{C}: |z|=1\}$.
The Hardy space $H^{2}$ is the closed subspace
of $L^{2}$ consisting of functions whose negative Fourier coefficients vanish.
Let $P_{+}$ be the orthogonal projection from $L^{2}$ onto $H^{2},P_{-}=I_{L^{2}}-P_{+}.$

For $f,g\in L^{\infty},$
the space of essentially bounded measurable functions on $\mathbb{T}$,
 the singular integral operator $S_{f,g}$ on $L^{2}$ is defined by
\begin{align}\label{Sfg}
S_{f,g}x=fP_{+}x+gP_{-}x,\quad x\in L^{2}.
\end{align}

However, $S^{*}_{f,g}$ generally cannot be expressed in form \eqref{Sfg}.

In a given linear space $X$, denote by $X_N$ the space of all $N$-dimensional vectors with components in 
$X$, and by $X_{N \times N}$ the space of $N \times N$ matrices with entries in $X$.
\begin{definition}
If $H=\begin{bmatrix} f& u\\ g & v\end{bmatrix} \in L_{2\times2}^{\infty}$, the generalized Cauchy singular integral operator with symbol $H$ is denoted by $R_{H}$ and defined by:
\begin{align*}
R_{H}x = P_{+}fP_{+}x + P_{-}gP_{+}x + P_{+}u P_{-}x + P_{-}v P_{-}x,\quad x\in L^{2}.
\end{align*}
\end{definition}
In this paper, we use the abbreviations SIO for the singular integral operator 
and GSIO for the generalized Cauchy singular integral operator. 
For more general form of singular integral operators, please refer to \cite[(15)]{Ehrhardt}.
The GSIO
 $R_{\begin{bmatrix}
\begin{smallmatrix}
 f & u \\
g & v
\end{smallmatrix}
\end{bmatrix}}$
can be expressed as an operator matrix with respect to
the decomposition
$L^{2}=H^{2}\oplus \bar{z}\overline{H^{2}},$ that is,
\begin{gather}\label{M}
\begin{bmatrix}T_{f}& H^{*}_{\bar {u}}\\
		H_{g} &   \tilde{T}_{v} \end{bmatrix},
\end{gather}
where 
$T_{f}$ denotes the Toeplitz operator on $H^{2}$ such that
\begin{align*}
T_{f}x=P_{+}(f x),\quad  x\in H^{2};
\end{align*}
$H_{g}$ denotes the Hankel operator on $H^{2}$ such that
\begin{align}\label{HD}
H_{g}x=P_{-}(g x),\quad  x\in H^{2};
\end{align}
$H^{*}_{\bar {u}}$ denotes the adjoint of Hankel operator  such that
\begin{align*}
H^{*}_{\bar {u}}y=P_{+}(u y),\quad  y\in \bar{z}\overline{H^{2}};
\end{align*}
$\tilde{T}_{v}$ denotes the dual Toeplitz operator on $\bar{z}\overline{H^{2}}$ such that
\begin{align*}
 \tilde{T}_{v}y=P_{-}(v y),\quad  y\in \bar{z}\overline{H^{2}}.
\end{align*}

Note that 
$S_{f,g}
=R_{\begin{bmatrix}
\begin{smallmatrix}
 f & g \\
f & g
\end{smallmatrix}
\end{bmatrix}}.$ 
In contrast to singular integral operators \eqref{Sfg}, the adjoint of a GSIO remains 
a GSIO, that is, $R^*_{\begin{bmatrix}
\begin{smallmatrix}
 f & u \\
g & v
\end{smallmatrix}
\end{bmatrix}}
=R_{\begin{bmatrix}
\begin{smallmatrix}
 \bar{f} & \bar{g} \\
\bar{u} & \bar{v}
\end{smallmatrix}
\end{bmatrix}}.$ 
In addition, GSIOs\cite[page 3]{sang2023algebras} 
are closely associated with multiplication operators, 
the Hilbert transform, (Dual) Truncated Toeplitz operators,
Toeplitz+Hankel operators,
block Toeplitz operators, the dilation of truncated Toeplitz operators, Foguel-Hankel operators, 
duplex Toeplitz opertors, and symmetrized Hankel operators\cite{sobolev2022multichannel}.
\subsection{Motivation}

Since a GSIO is a $2 \times 2$ operator matrix composed of Toeplitz and Hankel operators, 
we will separately review studies on product problems for these operators. Furthermore, 
we will review existing studies on algebraic properties related to SIOs, 
truncated Toeplitz operators, and  dual truncated Toeplitz operators.

On Toeplitz operators.
The Brown-Halmos theorems refer to a pair of theorems established by Brown and Halmos \cite{Brown1963} 
regarding Toeplitz operators on the Hardy space. 
For Toeplitz operators $T_{f}$ and $T_{g}$, which state that:
\begin{itemize}
	\item \textbf{Semi-commutativity}: $T_{f}T_{g}=T_{fg}$ if and only if either $g$ is analytic or $f$ is coanalytic.
    \item \textbf{Commutativity}: $T_{f}T_{g}=T_{g}T_{f}$ if and only if either both $f$ and $g$ are analytic, or both are coanalytic, or a nontrivial linear combination of $f$ and $g$ is constant.
\end{itemize}
Furthermore, Gu and Zheng \cite{gu1998products} investigated Brown-Halmos-type theorems for block Toeplitz operators.

Halmos' Problem 5 \cite{halmos1970ten} states that: "Is every subnormal Toeplitz operator either normal or analytic?" 
Abrahamse \cite{abrahamse1976subnormal} gave a general sufficient condition for an affirmative answer to Halmos' Problem 5. 
Halmos \cite{halmos1979ten} almost certainly believed the answer to his question would be yes. However, 
Halmos' Problem 5 was answered in the negative by Sun \cite{shunhua1984hyponormal,shunhua1985hyponormal,shunhua1985toeplitz} 
and Cowen and Long \cite{cowen1984some}.
Since there is currently no complete characterization of the symbols for subnormal Toeplitz operators, 
we revisit quasinormality—an intermediate property between normality and subnormality.
Amemiya, Ito, and Wong \cite{amemiya1975quasinormal} proved that 
a quasinormal Toeplitz operator is either normal or analytic.

On Hankel operators. Yoshino \cite{yoshino1999conditions} characterizes when the product of 
two Hankel operators is also a Hankel operator through the properties of their symbols.
 Martínez-Avendaño and Rosenthal \cite{martinez2007introduction} addressed this problem,
  stating that the product of two nonzero Hankel operators is a Hankel operator if and only if 
  both operators are constant multiples of the same rank-one Hankel operator. Additionally, 
  Gu \cite{gu2003product} characterized when the product of two block Hankel operators is a Hankel operator. 
  Gu and Zheng \cite{gu1998products} proved that two nonzero Hankel operators commute if and only if one is
   a constant multiple of the other. A quasinormal Hankel operator must be normal and a scalar multiple 
   of a self-adjoint Hankel operator \cite[4.4.9,4.4.11]{martinez2007introduction}. 
   The Hankel operators defined here differ from those in \eqref{HD}, but both types share the same properties.

On SIOs. Nakazi and Yamamoto \cite{Nakazi2014} investigated the normality of SIOs.
Nakazi \cite{nakazi2018hyponormal} studied the hyponormality of SIOs with special symbols.
Ko and Lee \cite{ko2021quasinormality} studied the quasinormality of SIOs.
Gu \cite{Gu2016} established characteristic equations for SIOs and used them to study algebraic properties 
such as commutativity and semi-commutativity of SIOs.
Recently, C{\^a}mara, Guimar{\~a}es, and Partington \cite{camara2023paired} have provided a concise proof for the semi-commutativity of SIOs.

Within the study of operator theory on model spaces, 
Câmara and Partington \cite{camara2017} established that asymmetric truncated Toeplitz operators 
are equivalent after extension to Toeplitz operators with $2 \times 2$ matrix-valued symbols. 
They also demonstrated \cite{Camara2020} that dual truncated Toeplitz operators are equivalent after extension 
to  SIOs with matrix-valued symbols. Thus, both truncated and dual 
truncated Toeplitz operators are equivalent after extension to SIOs with matrix-valued symbols. 
Furthermore, GSIOs are equivalent after extension to 
SIOs with matrix-valued symbols \cite{Ehrhardt,sang2023algebras}.

Sedlock \cite{zbMATH05914336} discovered that, except in trivial cases, the product of two truncated Toeplitz operators is itself 
a truncated Toeplitz operator if and only if both operators belong to the same Sedlock class, 
and their product also belongs to that class.
Chalendar and Timotin \cite{chalendar2013commutation} proved that two truncated Toeplitz operators commute, except in trivial cases, if and only if they belong to the same Sedlock class.
Similarly, analogous results hold for the semi-commutativity and commutativity of dual truncated Toeplitz operators, 
with the Sedlock class replaced by a new function class \cite{Sang2018,sang2020theorem,gu2021}.

Naturally, we  consider the following two questions.
\begin{enumerate}[label=\textbf{Question \arabic*}, leftmargin=3cm]
    \item: When is the product of two GSIOs also a GSIO? \label{Q11}
    \item: Under what conditions do two GSIOs commute?    \label{Q22}
\end{enumerate}

\subsection{The main results}
The primary challenge in studying \ref{Q11} and \ref{Q22} lies in the need to simultaneously characterize the conditions that eight functions must satisfy.

Let $V$ be the anti-unitary operator on $L^{2}$ defined by $Vf=\bar{z}\bar{f}$ for $f\in L^{2}.$ Here, $f_{+}$ denotes $P_{+}f$ and $f_{-}$ denotes $P_{-}f.$

The following theorem provides a complete solution to \ref{Q11}.

\noindent{\bf{Theorem \ref{main1}}}\,\,\,
If ${H_{1}=\begin{bmatrix}
	\begin{smallmatrix}
	   f_{1}& u_{1}\\
	   g_{1} & v_{1} \end{smallmatrix}
	\end{bmatrix}},
	{H_{2}=
	\begin{bmatrix}
	\begin{smallmatrix}
	   f_{2}& u_{2}\\
	   g_{2} & v_{2}
	\end{smallmatrix}
	\end{bmatrix}}
	\in L_{2\times2}^{\infty},$
	the following statements are equivalent.
	\begin{enumerate}
	\item The product $R_{H_{1}}R_{H_{2}}$ is a generalized Cauchy singular integral operator; 
	\item The following equation holds:
		\begin{align*}
			\brackbinom{V(\bar{f_{1}})_{-}}{g_{1-}}\otimes \brackbinom{Vf_{2-}}{(\bar{u}_{2})_{-}}
			=\brackbinom{V(\bar{u}_{1})_{-}}{v_{1-}}\otimes \brackbinom{Vg_{2-}}{(\bar{v}_{2})_{-}}.
			\end{align*}
	\item One of the following conditions holds.
		   \begin{enumerate}[label=(\arabic{enumi}.\arabic*)]
			\item Either 
	$\brackbinom{\bar{f}_{1}}{g_{1}}\in H_{2}^{\infty}$ or $\brackbinom{f_{2}}{\bar{u}_{2}}\in H_{2}^{\infty},$ and either
	  $\brackbinom{\bar{u}_{1}}{v_{1}}\in H_{2}^{\infty}$ or $\brackbinom{g_{2}}{\bar{v}_{2}}\in H_{2}^{\infty};$
			\item  If (3.1) does not happen, there exists a nonzero constant $\lambda$ such that
			$\bar{f}_{1}-\bar{\lambda}\bar{u}_{1},g_{1}-\lambda v_{1},g_{2}-\lambda f_{2},
			\bar{v}_{2}-\bar{\lambda}\bar{u}_{2}\in H^{\infty}.$
		   \end{enumerate}
\end{enumerate}

Theorems \ref{main2}, \ref{mainc1}, \ref{mainc2}, and \ref{R2} collectively provide a complete solution to Question \ref{Q22}. 
Given their lengthy formulations, the precise statements of these theorems are omitted here for brevity.

The following theorem characterizes when the product of two asymmetric dual truncated Toeplitz operators equals 
an asymmetric dual truncated Toeplitz operator.

\noindent{\bf{Theorem \ref{ADTP}}}\,\,\,
Let $\phi,\psi,\sigma\in L^{\infty},$ and let $\alpha,\beta,\theta$ be three non-constant inner functions.
Then $D^{\alpha,\beta}_{\psi} D^{\theta,\alpha}_{\phi}=D^{\theta,\beta}_{\sigma}$ if and only if
one of the following statements holds.
\begin{enumerate}
\item $\psi,\bar{\alpha}\beta\bar{\psi}\in H^{\infty};$
\item $\bar{\phi},\theta\bar{\alpha}\phi\in H^{\infty};$
\item $\bar{\phi},(1-\lambda \bar{\alpha})\theta\phi,(\bar{\alpha}-\bar{\lambda})\beta\bar{\psi},(\alpha-\lambda)\psi,
(1-\lambda\bar{\alpha})\theta\phi\psi
\in H^{\infty},$ where $\lambda$ is a constant with $|\lambda|< 1;$
\item $\psi,(1-\lambda \bar{\alpha})\beta\bar{\psi},(\bar{\alpha}-\bar{\lambda})\theta\phi,(\alpha-\lambda)\bar{\phi},
(1-\lambda\bar{\alpha})\beta\bar{\phi}\bar{\psi}
\in H^{\infty},$ where $\lambda$ is a constant with $|\lambda|< 1.$
\end{enumerate}
In this case, $\sigma=\phi\psi.$

The following theorem provides the necessary and sufficient conditions for a SIO to be quasinormal.

\noindent{\bf{Theorem \ref{quasinormal}}}\,\,\,
	Let $f$ and $g$ be functions in $L^{\infty}.$ 
	Then $S_{f,g}$ is quasinormal 
	if and only if one of the following statements holds.
\begin{enumerate}
	\item $|f|$ and $|g|$ are constants, $f,\bar{g}\in H^{\infty};$
	\item $|f|=|g|$ is constants, $f\bar{g}\in H^{\infty};$
	\item $f-g$ and $f|g|^2 -g|f|^2$ are constants, and $|f|^{2}_{-}=|g|^{2}_{-}=(f\bar{g})_{-};$
	\item $\brackbinom{Vf_{-}}{(\bar{g})_{-}}\neq 0,\brackbinom{V(|f|^{2}-f\bar{g})_{-}}{(f\bar{g}-|g|^{2})_{-}}\neq 0$ and one of the following statements holds.
	\begin{enumerate}
			\item There are  constants $\mu$ and $\alpha(\neq 0)$ such that $f-\mu g$ is constant,
	 $|f|^{2}-f\bar{g}-\bar{\alpha}\bar{g},|g|^{2}-f\bar{g}+\alpha g,
	 f(\mu \bar{f}-\bar{g}-\alpha),
	 \bar{g}(\bar{\mu}f-g-\bar{\alpha}),
	 fg(\mu \bar{f}-\bar{g}-\alpha),
	 \bar{g}^{2}(\bar{\mu}f-g-\bar{\alpha})\in H^{\infty};$
	\item There are  constants $\mu$ and $\alpha(\neq 0)$ such that $g-\mu f$ is constant,
 $|f|^{2}-f\bar{g}-\bar{\alpha}\bar{f},|g|^{2}-f\bar{g}+\alpha f,f(\bar{f}-\mu \bar{g}-\alpha),
 \bar{g}(f-\bar{\mu}g-\bar{\alpha}),
 \bar{f}\bar{g}(f -\bar{\mu}g-\bar{\alpha}),
f^2(\bar{f}-\mu \bar{g}-\alpha)\in H^{\infty};$
   \item There are  constants $\mu$ and $\alpha(\neq 0)$ such that $f(\bar{f}-\bar{\mu}\bar{g}),\bar{g}(f-\mu g)$,
 $|f|^{2}-f\bar{g}-\bar{\alpha}(\mu\bar{f}-\bar{g}),
 f\bar{g}-|g|^{2}-\alpha (\bar{\mu}f-g),
 f\bar{g}-\alpha f,
 |g|^2 -\bar{\alpha}\bar{g},
\bar{g}(|f|^2-|g|^2-\bar{\alpha}\mu\bar{f}+\bar{\alpha}\bar{g}),
f(|f|^{2}-|g|^2-\alpha\bar{\mu} f+\alpha  g)\in H^{\infty};$
  \item There are  constants $\mu$ and $\alpha(\neq 0)$ such that $f(\bar{g}-\bar{\mu}\bar{f}),\bar{g}(g-\mu f),$
$|f|^{2}-f\bar{g}-\bar{\alpha}(\bar{f}-\mu\bar{g}),
f\bar{g}-|g|^{2}-\alpha (f-\bar{\mu}g),
f\bar{g}-\bar{\alpha}\bar{g},
|f|^2 -\alpha f,
\bar{g}(|f|^2-|g|^2+\bar{\alpha}\mu\bar{g}-\bar{\alpha}\bar{f}),
f(|f|^{2}-|g|^2-\alpha \bar{\mu} g+\alpha  f)\in H^{\infty}.$
	\end{enumerate}
\end{enumerate}

The paper is structured as follows: Section \ref{Pre} introduces three key lemmas for this paper. 
Section \ref{4} proves Theorem \ref{main1}, establishing the necessary and sufficient conditions under which the product of two GSIOs 
is a GSIO. Section \ref{commute} presents the commutativity conditions for two GSIOs. Section \ref{5} applies the conclusions 
from Sections \ref{4} and \ref{commute} to SIOs and asymmetric dual truncated Toeplitz operators. 
This section establishes the necessary and sufficient conditions for when the product of 
two asymmetric dual truncated Toeplitz operators is an asymmetric dual truncated Toeplitz operator, 
characterizes quasinormal SIOs, and, using a unified approach, reproves the commutativity and semi-commutativity 
of both SIOs and dual truncated Toeplitz operators separately. Appendix A presents a simple application of 
Lemma \ref{key} to zero products of two $2\times 2$ matrices. 
Appendix B provides a major portion of the computations required for the proof of Theorem \ref{mainc2}.

\section{Preliminaries}\label{Pre}
Write $M_{f}$ for the multiplication operator defined on $L^{2}$ by $M_{f}x=fx,$ for all $x\in L^{2}.$
If $M_{f}$ is expressed as an operator matrix for the decomposition
$L^{2}=H^{2}\oplus \bar{z}\overline{H^{2}},$ the result is of the form
\begin{gather*}
M_f=\begin{bmatrix}T_{f}&{H^{*}_{\bar{f}}} \\ H_{f} & \tilde{T}_{f}\end{bmatrix}.\quad
\end{gather*}
If $f$ and $g$ are in $L^{\infty},$ then $M_{f}M_{g}=M_{fg},$ and therefore
\begin{align}
T_{fg}&=T_{f}T_{g}+H^{*}_{\bar{f}}H_{g};  \label{eq:1m} \\
H_{fg}&=H_{f}T_{g}+\tilde{T}_{f}H_{g};            \label{eq:2m}      \\
H^{*}_{\bar{f}\bar{g}}&=T_{f}H^{*}_{\bar{g}}+H^{*}_{\bar{f}}\tilde{T}_{g};        \label{eq:4m}      \\
\tilde{T}_{fg}&=\tilde{T}_{f}\tilde{T}_{g}+H_{f}H^{*}_{\bar{g}}.    \label{eq:3m}
\end{align}

Let $h \in L^{\infty}$. $H_{h}=0$ if and only if $h \in H^{\infty}$,
which is defined as $L^{\infty}\cap H^{2}$, by the definition of the Hankel operator \eqref{HD}.
Hence,
\begin{enumerate}
	\item if either $\bar{f}\in H^{\infty}$ or $g\in H^{\infty},$ then  \eqref{eq:1m} can be reduced to
	\begin{align}\label{TT}
	T_{f}T_{g}=T_{fg};
	\end{align}
	\item if $f\in H^{\infty},$ then \eqref{eq:2m} implies that
\begin{align}\label{HT}
\tilde{T}_{f}H_g = H_g T_{f}=H_{fg};
\end{align}
\item if either $f\in H^{\infty}$ or $\bar{g}\in H^{\infty},$ then  \eqref{eq:3m} can be reduced to
\begin{align}\label{DD}
\tilde{T}_{f}\tilde{T}_{g}=\tilde{T}_{fg}.
\end{align}
\end{enumerate}

The operator $T$ is a Toeplitz operator on $H^{2}$ if and only if it satisfies the equation $T_{\bar{z}}TT_{z}=T$ \cite[Theorem 6]{Brown1963}.
A dual Toeplitz operator is anti-unitarily equivalent to a Toeplitz operator. That is,
for $f\in L^{\infty},$ we have 
\begin{align}\label{TVV}
T_{f}=V\tilde{T}_{\bar{f}}V.
\end{align}
Therefore, we can obtain a characterization of the dual Toeplitz operator.
\begin{proposition}\label{dual}
The operator $A$ is a dual Toeplitz operator on $\bar{z}\overline{H^{2}}$ if and only if it satisfies the equation $\tilde{T}_{z}A\tilde{T}_{\bar{z}}=A.$
\end{proposition}
Moreover, $V$ exhibits the following properties: 
\begin{align}
	P_{+}V&=VP_{-};\nonumber \\
	H_{\phi}&=VH^{*}_{\phi}V.\quad \forall \phi\in L^{\infty}\label{VV}
\end{align}
For vectors $p$ and $q$ are in a Hilbert space $\mathcal{H},$  the rank one operator $p\otimes q$ is defined by
$(p\otimes q)(x)=\left\langle x,q\right\rangle p\ \ (x\in \mathcal{H}).$
The identity operator on the Hilbert space $\mathcal{H}$ is represented by $I_{\mathcal{H}}.$
\begin{lemma}\label{k1}
Let $f_i,g_i\in L^{\infty}(i=1,2,\cdots,n),$ 	the following statements are equivalent.
\begin{enumerate}
	\item $\sum\limits_{i=1}^{n}H^{*}_{f_i}H_{g_i}$ is a  Toeplitz operator;
	\item $\sum\limits_{i=1}^{n}H_{f_i}H^{*}_{g_i}$ is a dual Toeplitz operator;
	\item $\sum\limits_{i=1}^{n}H^{*}_{f_i}H_{g_i}=0;$
	\item $\sum\limits_{i=1}^{n}H_{f_i}H^{*}_{g_i}=0;$
	\item $\sum\limits_{i=1}^{n}(f_{i-}\otimes g_{i-})=0.$
\end{enumerate}
\end{lemma}
\begin{proof}
From formulas \eqref{TVV} and \eqref{VV}, it can be seen that (1) and (2) are equivalent, and (3) and (4) are equivalent.
Using symbol mapping \cite[7.11]{douglas1998banach}, (1) implies (3). Moreover, (3) implies (1).
Recall that $I_{H^2}=T_{z}T_{\bar{z}}+1\otimes 1$, we have
\begin{align*}
\sum_{i=1}^{n}H_{f_i}H^{*}_{g_i}=&
\sum_{i=1}^{n}H_{f_i}(T_{z}T_{\bar{z}}+1\otimes 1)H^{*}_{g_i}\\
=&\tilde{T}_{z}\sum_{i=1}^{n}(H_{f_i}H^{*}_{g_i})\tilde{T}_{\bar{z}}+\sum_{i=1}^{n}H_{f_i}1\otimes H_{g_i}1,
\end{align*}
where the equality above comes from \eqref{HT}.
Proposition \ref{dual} shows that
$\sum_{i=1}^{n}H_{f_i}H^{*}_{g_i}$ is a dual Toeplitz operator
if and only if
$\sum_{i=1}^{n}H_{f_i}1\otimes H_{g_i}1=0.$
$H_{f_i}1=f_{i-}$ and $H_{g_i}1=g_{i-}$ give the desired result.
\end{proof}
\begin{theorem}\cite[Theorem 1.8]{peller2003hankel}\label{peller}
Let $R$ be a bounded operator from $H^2$ to $(H^2)^{\bot}$. Then $R$ is a Hankel operator if and only if it 
	satisfies the following commutation 
\begin{align}
\tilde{T}_{z}R=RT_{z}.
\end{align}
\end{theorem}
\begin{lemma}\label{k2}
If $a_i,b_i,c_i,d_i\in L^{\infty}(i=1,2,\cdots,n),$ then 
$\sum\limits_{i=1}^{n}(H_{a_i}T_{b_i}+\tilde{T}_{c_i}H_{d_i})$ is a Hankel operator  if and only if 
\[\sum_{i=1}^{n} (a_{i-}\otimes Vb_{i-}-c_{i-}\otimes Vd_{i-})=0.\]
\end{lemma}
\begin{proof}
Apply $I_{H^2}=T_{z}T_{\bar{z}}+1\otimes 1$, we have
\begin{align*}
H_{a_i}T_{b_i}T_{z}&=H_{a_i}(T_{z}T_{\bar{z}}+1\otimes 1)T_{b_i}T_{z}\\
&=H_{a_i}T_{z}T_{\bar{z}}T_{b_i}T_{z}+H_{a_i}(1\otimes 1)T_{b_i}T_{z}\\
&=H_{a_i}T_{z}T_{b_i}+H_{a_i}(1\otimes 1)T_{zb_i}\\
&=\tilde{T}_{z}H_{a_i}T_{b_i}+a_{i-}\otimes Vb_{i-},
\end{align*}
where the third equality comes form \eqref{TT}, the last equality follows from \eqref{HT}, thus
\begin{align}\label{a}
H_{a_i}T_{b_i}T_{z}-\tilde{T}_{z}H_{a_i}T_{b_i}=a_{i-}\otimes Vb_{i-}.
\end{align}

Applying operator V on both sides of the equation $I_{H^2}=T_{z}T_{\bar{z}}+1\otimes 1$ yields
$I_{(H^2)^{\bot}}=\tilde{T}_{\bar{z}}\tilde{T}_{z}+\bar{z}\otimes\bar{z}.$
\begin{align*}
\tilde{T}_{z}\tilde{T}_{c_i}H_{d_i}
&=\tilde{T}_{z}\tilde{T}_{c_i}(\tilde{T}_{\bar{z}}\tilde{T}_{z}+\bar{z}\otimes\bar{z})H_{d_i}\\
&=\tilde{T}_{z}\tilde{T}_{c_i}\tilde{T}_{\bar{z}}\tilde{T}_{z}H_{d_i}
+\tilde{T}_{z}\tilde{T}_{c_i}(\bar{z}\otimes\bar{z})H_{d_i}\\
&=\tilde{T}_{c_i}\tilde{T}_{z}H_{d_i}+\tilde{T}_{zc_i}(\bar{z}\otimes\bar{z})H_{d_i}\\
&=\tilde{T}_{c_i}H_{d_i}T_{z}+c_{i-}\otimes Vd_{i-},
\end{align*}
where the third equality comes form  \eqref{DD}, the last equality follows from \eqref{HT}, thus
\begin{align}\label{b}
\tilde{T}_{z}\tilde{T}_{c_i}H_{d_i}-\tilde{T}_{c_i}H_{d_i}T_{z}=c_{i-}\otimes Vd_{i-}.
\end{align}
Combining \eqref{a} and \eqref{b} yields
\begin{align*}
\sum_{i=1}^{n}(H_{a_i}T_{b_i}+\tilde{T}_{c_i}H_{d_i})T_{z}
-&\tilde{T}_{z}\sum_{i=1}^{n}(H_{a_i}T_{b_i}+\tilde{T}_{c_i}H_{d_i})\\
=&\sum_{i=1}^{n}(a_{i-}\otimes Vb_{i-}-c_{i-}\otimes Vd_{i-}).
\end{align*}
Applying Theorem \ref{peller} to the equation above shows that the conclusion holds.
\end{proof}

\begin{lemma}\label{key}
If vectors $p_i,q_i,r_i,s_i(i=1,2,\cdots,n)$ are in a Hilbert space $\mathcal{H},$ then 
\begin{align*}
\sum_{i=1}^{n}\brackbinom{p_i}{q_i}\otimes \brackbinom{r_i}{s_i}=0
\end{align*}
if and only if 
\begin{align*}
\sum_{i=1}^{n} p_i \otimes r_i=0;\ \
\sum_{i=1}^{n} p_i\otimes s_i=0;\ \
\sum_{i=1}^{n} q_i\otimes r_i=0;\ \
\sum_{i=1}^{n} q_i\otimes s_i=0.
\end{align*}
\end{lemma}
\begin{proof}
We only need to prove the case when $n=1$, and the other cases are similar.
For any vectors $x,y,a$ and $b,$ a direct computation yields.
\begin{align}
&\bigg\langle \brackbinom{p_1}{q_1}\otimes \brackbinom{r_1}{s_1}\brackbinom{x}{y},\brackbinom{a}{b}\bigg\rangle \nonumber\\
=&\langle (p_1\otimes r_1)x,a\rangle +\langle (p_1\otimes s_1)y,a\rangle+\langle (q_1\otimes r_1)x,b\rangle 
+\langle (q_1\otimes s_1)y,b\rangle. \label{keyi}
\end{align}
If $\brackbinom{p_1}{q_1}\otimes \brackbinom{r_1}{s_1}=0,$ let $y=b=0,$ then $p_1\otimes r_1=0.$ Similarly, we can obtain
$p_1\otimes s_1=0,q_1\otimes r_1=0$ and $q_1\otimes s_1=0.$ The necessity follows from the equality \eqref{keyi}.
\end{proof}
\section{Semi-commutativity of GSIOs}\label{4}
\begin{theorem}\label{main1}
If ${H_{1}=\begin{bmatrix}
\begin{smallmatrix}
   f_{1}& u_{1}\\
   g_{1} & v_{1} \end{smallmatrix}
\end{bmatrix}},
{H_{2}=
\begin{bmatrix}
\begin{smallmatrix}
   f_{2}& u_{2}\\
   g_{2} & v_{2}
\end{smallmatrix}
\end{bmatrix}}
\in L_{2\times2}^{\infty},$
the following statements are equivalent.
\begin{enumerate}
	\item The product $R_{H_{1}}R_{H_{2}}$ is a GSIO; 
	\item The following operator equality holds:
		\begin{align}\label{rr}
			\brackbinom{V(\bar{f_{1}})_{-}}{g_{1-}}\otimes \brackbinom{Vf_{2-}}{(\bar{u}_{2})_{-}}
			=\brackbinom{V(\bar{u}_{1})_{-}}{v_{1-}}\otimes \brackbinom{Vg_{2-}}{(\bar{v}_{2})_{-}}.
		\end{align}
	\item One of the following conditions holds:
		   \begin{enumerate}[label=(\arabic{enumi}.\arabic*)]
			\item either 
	$\brackbinom{\bar{f}_{1}}{g_{1}}\in H_{2}^{\infty}$ or $\brackbinom{f_{2}}{\bar{u}_{2}}\in H_{2}^{\infty},$ and either
	  $\brackbinom{\bar{u}_{1}}{v_{1}}\in H_{2}^{\infty}$ or $\brackbinom{g_{2}}{\bar{v}_{2}}\in H_{2}^{\infty};$
			\item  If (3.1) does not happen, there exists a nonzero constant $\lambda$ such that
			$\bar{f}_{1}-\bar{\lambda}\bar{u}_{1},g_{1}-\lambda v_{1},g_{2}-\lambda f_{2},
			\bar{v}_{2}-\bar{\lambda}\bar{u}_{2}\in H^{\infty}.$
		   \end{enumerate}
	\end{enumerate}
\end{theorem}
\begin{proof}
Since
\begin{align}\label{RHRH}
R_{H_1}R_{H_2}
=\begin{bmatrix}T_{f_{1}}T_{f_{2}}+H^{*}_{\bar{u}_{1}}H_{g_{2}}& T_{f_{1}}H^{*}_{\bar{u}_{2}}+H^{*}_{\bar{u}_{1}}\tilde{T}_{v_{2}}\\
	H_{g_{1}}T_{f_{2}}+\tilde{T}_{v_{1}}H_{g_{2}} &  H_{g_{1}}H^{*}_{\bar{u}_{2}}+\tilde{T}_{v_{1}}\tilde{T}_{v_{2}}
\end{bmatrix},
\end{align}
we have $R_{H_1}R_{H_2}$ is a GSIO if and only if all the following four statements hold.
\begin{itemize}
	\item $T_{f_{1}}T_{f_{2}}+H^{*}_{\bar{u}_{1}}H_{g_{2}}$ is a Toeplitz operator;
	\item $(T_{f_{1}}H^{*}_{\bar{u}_{2}}+H^{*}_{\bar{u}_{1}}\tilde{T}_{v_{2}})^{*}$ is a Hankel operator;
	\item $H_{g_{1}}T_{f_{2}}+\tilde{T}_{v_{1}}H_{g_{2}}$ is a Hankel operator; 
	\item $H_{g_{1}}H^{*}_{\bar{u}_{2}}+\tilde{T}_{v_{1}}\tilde{T}_{v_{2}}$ is a dual Toeplitz operator.
\end{itemize}

By symbol mapping \cite[7.11]{douglas1998banach}, $T_{f_{1}}T_{f_{2}}+H^{*}_{\bar{u}_{1}}H_{g_{2}}$ is a Toeplitz operator, then it equals
$T_{f_1 f_2}.$ Using \eqref{TT}, we have $T_{f_1 f_2}-T_{f_{1}}T_{f_{2}}=H^{*}_{\bar{f}_1}H_{f_2}.$
So, $T_{f_{1}}T_{f_{2}}+H^{*}_{\bar{u}_{1}}H_{g_{2}}$ is a Toeplitz operator if and only if 
$H^{*}_{\bar{f_{1}}}H_{f_{2}}=H^{*}_{\bar{u}_{1}}H_{g_{2}}.$
Similarly, $H_{g_{1}}H^{*}_{\bar{u}_{2}}+\tilde{T}_{v_{1}}\tilde{T}_{v_{2}}$ is a dual Toeplitz operator if and only if 
$H_{g_{1}}H^{*}_{\bar{u}_{2}}=H_{v_{1}}H^{*}_{\bar{v}_{2}}.$
 By Lemma \ref{k1} and  Lemma \ref{k2}, we have 
 $R_{H_1}R_{H_2}$ is a GSIO if and only if all the following four statements hold.
\begin{align*}
V(\bar{f_{1}})_{-}\otimes Vf_{2-}-V(\bar{u}_{1})_{-}\otimes Vg_{2-}&=0;\\
V(\bar{f}_{1})_{-}\otimes(\bar{u}_{2})_{-} -V(\bar{u}_{1})_{-}\otimes (\bar{v}_{2})_{-}&=0;\\
g_{1-}\otimes Vf_{2-}-v_{1-}\otimes Vg_{2-} &=0;\\
g_{1-}\otimes(\bar{u}_{2})_{-}-v_{1-}\otimes(\bar{v}_{2})_{-}&=0.
\end{align*}
Lemma \ref{key} implies that $R_{H_1}R_{H_2}$ is a GSIO if and only if \eqref{rr} holds.
When both sides of the equation \eqref{rr} are equal to zero, we can obtain conclusion (3.1); 
when both sides are rank-one operators, we can obtain conclusion (3.2).
\end{proof}

Next, we will calculate the matrix representations of $R_{H_1}R_{H_2}$ in the five cases of Theorem \ref{main1}(3).
Note that, by symbol mapping \cite[7.11]{douglas1998banach}, 
the entries at the upper-left and lower-right corners of the matrix representation of $R_{H_1}R_{H_2}$
are $T_{f_{1}f_{2}}$ and $\tilde{T}_{v_{1} v_{2}},$ respectively.
The calculation for the first four cases only requires substituting the conditions into equation 
\eqref{RHRH}, while utilizing the fact that $H_{h}=0$ if and only if $h\in H^{\infty}$, together with \eqref{HT}.
\begin{enumerate}[label=\textbf{Case \Alph*:}, leftmargin=2cm]
	\item If 
	$\brackbinom{\bar{f}_{1}}{g_{1}},\brackbinom{\bar{u}_{1}}{v_{1}}\in H_{2}^{\infty},$ 
	then
\begin{align}\label{1.1.}
R_{H_1}R_{H_2}
&=
\begin{bmatrix}T_{f_{1}f_{2}}& H^{*}_{\bar{f}_1\bar{u}_{2}}\\
	H_{v_{1}g_{2}} &  \tilde{T}_{v_{1} v_{2}}
 \end{bmatrix}.
\end{align}
    \item If 
$\brackbinom{\bar{f}_{1}}{g_{1}},\brackbinom{g_{2}}{\bar{v}_{2}}\in H_{2}^{\infty},$	then
\begin{align}\label{1.2.}
R_{H_1}R_{H_2}=\begin{bmatrix}T_{f_{1}f_{2}}& H^{*}_{\bar{f}_1 \bar{u}_2+\bar{u}_1 \bar{v}_2}\\
	0 &  \tilde{T}_{v_{1} v_{2}}
\end{bmatrix}.
\end{align}
    \item If 
 $\brackbinom{f_{2}}{\bar{u}_{2}},\brackbinom{\bar{u}_{1}}{v_{1}}\in H_{2}^{\infty},$ then
\begin{align}\label{2.1.}
	R_{H_1}R_{H_2}=\begin{bmatrix}T_{f_{1}f_{2}}& 0\\
		H_{g_1 f_2 +v_1 g_2} &  \tilde{T}_{v_{1} v_{2}}
	\end{bmatrix}.
\end{align}
    \item If 
 $\brackbinom{f_{2}}{\bar{u}_{2}},\brackbinom{g_{2}}{\bar{v}_{2}}\in H_{2}^{\infty},$	then
\begin{align}\label{2.2.}
	R_{H_1}R_{H_2}=\begin{bmatrix}T_{f_{1}f_{2}}& H^{*}_{\bar{u}_1 \bar{v}_2 }\\
		H_{g_1 f_2} &  \tilde{T}_{v_{1} v_{2}}
	\end{bmatrix}.
\end{align}
  \item If the rank on both sides of \eqref{rr} is 1, then there exists a non-zero constant $\lambda$  such that 
$\bar{f}_{1}-\bar{\lambda}\bar{u}_{1},g_{1}-\lambda v_{1},g_{2}-\lambda f_{2},
\bar{v}_{2}-\bar{\lambda}\bar{u}_{2}\in H^{\infty}.$
Next, we will calculate  the secondary diagonal elements in the operator matrix \eqref{RHRH} separately.
\begin{align*}
T_{f_{1}}H^{*}_{\bar{u}_{2}}+H^{*}_{\bar{u}_{1}}\tilde{T}_{v_{2}}
&=T_{f_{1}}H^{*}_{\bar{v}_{2}/\bar{\lambda}}+H^{*}_{\bar{f}_{1}/\bar{\lambda}}\tilde{T}_{v_{2}}
=H^{*}_{(1/\bar{\lambda})\bar{f}_1 \bar{v}_2}.
\end{align*}
The second equality follows from  \eqref{eq:4m}. \eqref{eq:2m} implies that
\begin{align*}
	H_{g_{1}}T_{f_{2}}+\tilde{T}_{v_{1}}H_{g_{2}}=H_{\lambda v_{1}}T_{f_{2}}+\tilde{T}_{v_{1}}H_{\lambda f_{2}}
=H_{\lambda v_{1}f_2}.
\end{align*}
Hence
\begin{align}\label{casec}
R_{H_{1}}R_{H_{2}}&=
	\begin{bmatrix}T_{f_1 f_2}& H^{*}_{(1/\bar{\lambda}) \bar{f}_1 \bar{v}_2}\\
		H_{\lambda v_{1}f_2} &   \tilde{T}_{v_1 v_2} \end{bmatrix}.
\end{align}
\end{enumerate}
\begin{example}
If $f_1=u_1,g_1=v_1,f_2=g_2,u_2=v_2,$ then $\eqref{rr}$ holds. In this case, $R_{H_1}R_{H_2}$ is a GSIO.
In particular, since
$S_{f,g}
=R_{\begin{bmatrix}
\begin{smallmatrix}
 f & g \\
f & g
\end{smallmatrix}
\end{bmatrix}}$
and 
$S^{*}_{f,g}
=R_{\begin{bmatrix}
\begin{smallmatrix}
 \bar{f} & \bar{f} \\
\bar{g} & \bar{g}
\end{smallmatrix}
\end{bmatrix}},$ $S^{*}_{f,g}S_{f,g}$ satisfies
Case E with $\lambda=1,$ and \eqref{casec} implies
\begin{align}\label{StarS}
S^{*}_{f,g}S_{f,g}
=R_{\begin{bmatrix}
	\begin{smallmatrix}
	   |f|^2 & \bar{f}g\\
	   f\bar{g} & |g|^{2}
	\end{smallmatrix}
	\end{bmatrix}}
\end{align}
is a GSIO. This simple observation helps us characterize 
quasinormal singular integral operators (see Theorem \ref{quasinormal}).
\end{example}

\begin{proposition}
Let
${H=\begin{bmatrix}
	\begin{smallmatrix}
	   f& u\\
	   g & v \end{smallmatrix}
	\end{bmatrix}}
	\in L_{2\times2}^{\infty}.$ Then $R_{H}$ is an isometry if and only if
$H$ satisfies the following two conditions.
\begin{enumerate}
	\item $|f|=|v|=1;$
	\item One of the following conditions holds.
	\begin{enumerate}
		\item $f,g,\bar{u},\bar{v}\in H^{\infty};$ 
		\item There exists a unimodular constant $\lambda,$ such that $f-\lambda g,\bar{u}-\bar{\lambda}\bar{v},f\bar{v}\in H^{\infty}$.
	\end{enumerate}
\end{enumerate}
\end{proposition}
\begin{proof}
Suppose that $R_{H}$ is an isometry, we have 
$
R_{{\begin{bmatrix}
	\begin{smallmatrix}
	   \bar{f}& \bar{g}\\
	   \bar{u} & \bar{v}
	\end{smallmatrix}
	\end{bmatrix}}}
R_{{\begin{bmatrix}
	\begin{smallmatrix}
	   f& u\\
	   g & v
	\end{smallmatrix}
	\end{bmatrix}}}=I_{L^2}.
$
By Theorem \ref{main1}, we have 
	$\brackbinom{Vf_{-}}{(\bar{u})_{-}}\otimes \brackbinom{Vf_{-}}{(\bar{u})_{-}}
	=\brackbinom{Vg_{-}}{(\bar{v})_{-}}\otimes \brackbinom{Vg_{-}}{(\bar{v})_{-}}.$
\eqref{1.1.},\eqref{1.2.},\eqref{2.1.},\eqref{2.2.} and \eqref{casec} implies $|f|=|v|=1.$
There are exactly two  cases: either $\brackbinom{Vf_{-}}{(\bar{u})_{-}}=\brackbinom{Vg_{-}}{(\bar{v})_{-}}=0$ 
or $\brackbinom{Vf_{-}}{(\bar{u})_{-}}=\lambda \brackbinom{Vg_{-}}{(\bar{v})_{-}}$ for some  unimodular constant $\lambda.$

In the first case, $f,g,\bar{u},\bar{v}\in H^{\infty}.$ Conversely, if $|f|=|v|=1,$ and $f,g,\bar{u},\bar{v}\in H^{\infty},$ 
then $R_{H}=T_{f}\oplus \tilde{T}_{v}.$ Using \eqref{TT} and \eqref{DD}, we have  $R_H$ is isometric.

In the second case, using \eqref{casec}, we have $R^{*}_{H}R_{H}=\begin{bmatrix}T_{|f|^2}& H^{*}_{{\lambda}f\bar{v}}\\
	H_{\lambda f\bar{v}} &   \tilde{T}_{|v|^2} \end{bmatrix},$ and so, $f\bar{v}\in H^{\infty}.$
The above reasoning shows that $R_H$ is isometric if $H$ satisfies conditions (1) and (2)(b).
\end{proof}

\section{commuting GSIOs}\label{commute}
\begin{theorem}\label{main2}
	If ${H_{1}=\begin{bmatrix}
	\begin{smallmatrix}
	   f_{1}& u_{1}\\
	   g_{1} & v_{1} \end{smallmatrix}
	\end{bmatrix}},{H_{2}=\begin{bmatrix}
	\begin{smallmatrix}
	   f_{2}& u_{2}\\
	   g_{2} & v_{2} \end{smallmatrix}
	\end{bmatrix}}$  in $L_{2\times2}^{\infty},$
 then $R_{H_{1}}R_{H_{2}}=R_{H_{2}}R_{H_{1}}$ if and only if the following three operator equations hold.
\begin{align}
	&H_{\bar{u}_{2}}T_{\bar{f}_{1}}+\tilde{T}_{\bar{v}_{2}}H_{\bar{u}_{1}}
	-H_{\bar{u}_{1}}T_{\bar{f}_{2}}-\tilde{T}_{\bar{v}_{1}}H_{\bar{u}_{2}}=0;\label{1.2}\\
	&H_{g_{1}}T_{f_{2}}+\tilde{T}_{v_{1}}H_{g_{2}}-H_{g_{2}}T_{f_{1}}-\tilde{T}_{v_{2}}H_{g_{1}}=0;\label{2.1}\\
	&\brackbinom{V(\bar{f_{1}})_{-}}{g_{1-}}\otimes \brackbinom{Vf_{2-}}{(\bar{u}_{2})_{-}}
	-\brackbinom{V(\bar{u}_{1})_{-}}{v_{1-}}\otimes \brackbinom{Vg_{2-}}{(\bar{v}_{2})_{-}} \nonumber \\
	=&\brackbinom{V(\bar{f_{2}})_{-}}{g_{2-}}\otimes \brackbinom{Vf_{1-}}{(\bar{u}_{1})_{-}}
	-\brackbinom{V(\bar{u}_{2})_{-}}{v_{2-}}\otimes \brackbinom{Vg_{1-}}{(\bar{v}_{1})_{-}}.\label{WW}
	\end{align}
\end{theorem}
	\begin{proof}
Using \eqref{RHRH}, we have 
		\begin{align*}
		R_{H_{2}}R_{H_{1}}=\begin{bmatrix}T_{f_{2}}T_{f_{1}}+H^{*}_{\bar{u}_{2}}H_{g_{1}}&
		T_{f_{2}}H^{*}_{\bar{u}_{1}}+H^{*}_{\bar{u}_{2}}\tilde{T}_{v_{1}}\\
		   H_{g_{2}}T_{f_{1}}+\tilde{T}_{v_{2}}H_{g_{1}} &
		   H_{g_{2}}H^{*}_{\bar{u}_{1}}+\tilde{T}_{v_{2}}\tilde{T}_{v_{1}}
		\end{bmatrix}.
		\end{align*}
Using \eqref{eq:1m} and \eqref{eq:3m}, the equality $R_{H_{1}}R_{H_{2}}=R_{H_{2}}R_{H_{1}}$ holds if and only if the following four operator equations are satisfied.
		\begin{align}
		H^{*}_{\bar{f}_{1}}H_{f_{2}}-H^{*}_{\bar{u}_{1}}H_{g_{2}}
		&=H^{*}_{\bar{f}_{2}}H_{f_{1}}-H^{*}_{\bar{u}_{2}}H_{g_{1}}; \label{1.1}\\
		H_{\bar{u}_{2}}T_{\bar{f}_{1}}+\tilde{T}_{\bar{v}_{2}}H_{\bar{u}_{1}}
		&=H_{\bar{u}_{1}}T_{\bar{f}_{2}}+\tilde{T}_{\bar{v}_{1}}H_{\bar{u}_{2}};\nonumber\\
		H_{g_{1}}T_{f_{2}}+\tilde{T}_{v_{1}}H_{g_{2}}&=H_{g_{2}}T_{f_{1}}+\tilde{T}_{v_{2}}H_{g_{1}};\nonumber\\
		H_{g_{1}}H^{*}_{\bar{u}_{2}}
		-H_{v_{1}}H^{*}_{\bar{v}_{2}}
		&=H_{g_{2}}H^{*}_{\bar{u}_{1}}-H_{v_{2}}H^{*}_{\bar{v}_{1}}.\label{2.2}
		\end{align}

By Lemma \ref{k1}, \eqref{1.1} is equivalent to
\begin{align*}
V(\bar{f_{1}})_{-}\otimes Vf_{2-}- V(\bar{u}_{1})_{-}\otimes Vg_{2-}
= V(\bar{f_{2}})_{-}\otimes Vf_{1-}- V(\bar{u}_{2})_{-}\otimes Vg_{1-}.
\end{align*}
Similarly, \eqref{2.2} is equivalent to
\begin{align*}
g_{1-}\otimes (\bar{u}_{2})_{-}-v_{1-}\otimes (\bar{v}_{2})_{-}&=g_{2-}\otimes (\bar{u}_{1})_{-}-v_{2-}\otimes (\bar{v}_{1})_{-}.
\end{align*}
By Lemma \ref{k2}, \eqref{1.2} and \eqref{2.1} imply that
\begin{align*}
V(\bar{f_{1}})_{-}\otimes (\bar{u}_{2})_{-}- V(\bar{u}_{1})_{-}\otimes (\bar{v}_{2})_{-}
&= V(\bar{f_{2}})_{-}\otimes (\bar{u}_{1})_{-}- V(\bar{u}_{2})_{-}\otimes (\bar{v}_{1})_{-};\\
g_{1-}\otimes Vf_{2-}-v_{1-}\otimes Vg_{2-}&=g_{2-}\otimes Vf_{1-}-v_{2-}\otimes Vg_{1-}.
\end{align*}
Thus,
\begin{align*}
&\brackbinom{V(\bar{f_{1}})_{-}}{g_{1-}}\otimes \brackbinom{Vf_{2-}}{(\bar{u}_{2})_{-}}
-\brackbinom{V(\bar{u}_{1})_{-}}{v_{1-}}\otimes \brackbinom{Vg_{2-}}{(\bar{v}_{2})_{-}}\\
=&
\brackbinom{V(\bar{f_{2}})_{-}}{g_{2-}}\otimes \brackbinom{Vf_{1-}}{(\bar{u}_{1})_{-}}
-\brackbinom{V(\bar{u}_{2})_{-}}{v_{2-}}\otimes \brackbinom{Vg_{1-}}{(\bar{v}_{1})_{-}}
\end{align*}
by Lemma \ref{key}.
\end{proof}
\begin{corollary}\label{odd}
Let ${H_{1}=\begin{bmatrix}
		\begin{smallmatrix}
		   f_{1}& u_{1}\\
		   g_{1} & v_{1} \end{smallmatrix}
		\end{bmatrix}},{H_{2}=\begin{bmatrix}
		\begin{smallmatrix}
		   f_{2}& u_{2}\\
		   g_{2} & v_{2} \end{smallmatrix}
		\end{bmatrix}}$  in $L_{2\times2}^{\infty},$
if $R_{H_{1}}R_{H_{2}}=R_{H_{2}}R_{H_{1}}$ and $f_1=v_1,f_2=v_2,$ then
$\brackbinom{V(\bar{f_{1}})_{-}}{g_{1-}}\otimes \brackbinom{Vf_{2-}}{(\bar{u}_{2})_{-}}
-\brackbinom{V(\bar{u}_{1})_{-}}{f_{1-}}\otimes \brackbinom{Vg_{2-}}{(\bar{f}_{2})_{-}}$
cannot be a rank-one operator.
\end{corollary}
\begin{proof}
For $p,q\in L^{2},$ the antilinear operator $C$ is defined by $C\brackbinom{p}{q}=\brackbinom{Vq}{Vp}.$
Furthermore, $C$ is an isometric involution. Let
 \[W=\brackbinom{V(\bar{f_{1}})_{-}}{g_{1-}}\otimes \brackbinom{Vf_{2-}}{(\bar{u}_{2})_{-}}
 -\brackbinom{V(\bar{u}_{1})_{-}}{f_{1-}}\otimes \brackbinom{Vg_{2-}}{(\bar{f}_{2})_{-}},\]
\eqref{WW} implies that $W=-CW^{*}C.$ Therefore, $W$ is a complex skew symmetric operator.
As shown in \cite[Theorem 2]{chen2015ranks}, a finite-rank complex skew-symmetric operator cannot have an odd rank, and thus 
$W$ cannot be a rank-one.
\end{proof}
Next, we will solve for the functions in the three operator equations from Theorem \ref{main2}. 
First, we solve \eqref{WW}, then substitute these solutions into  \eqref{1.2} and \eqref{2.1} , 
ultimately identifying the solutions that satisfy both \eqref{1.2} and \eqref{2.1}. When solving  \eqref{WW}, 
we will analyze three cases: when both sides of the equation \eqref{WW} are zero, rank one, and rank two, 
respectively. This approach will yield the following three theorems.

\begin{theorem}\label{mainc1}
If  both sides of \eqref{WW} are zero, then $R_{H_{1}}R_{H_{2}}=R_{H_{2}}R_{H_{1}}$ 
	if and only if one of the following statements holds.
\begin{enumerate}
	\item [$(1)$]$\bar{f}_1,\bar{f}_2,g_1,g_2,\bar{u}_1,\bar{u}_2,v_1,v_2\in H^{\infty}.$
	\item [$(2)$]$\bar{f}_1,\bar{f}_2,g_1,g_2,\bar{u}_2,(\bar{f}_2 - \bar{v}_2)\bar{u}_1\in H^{\infty},$ and $v_2$ is constant.
	\item [$(\hat{2})$] $\bar{f}_1,\bar{f}_2,g_1,g_2,\bar{u}_1,(\bar{f}_1 - \bar{v}_1)\bar{u}_2\in H^{\infty},$ and $v_1$ is constant.
	\item [$(3)$]$g_2,\bar{u}_1,\bar{u}_2,v_1,v_2,(f_2 -v_2)g_1\in H^{\infty},$ and $f_2$ is  constant.
	\item [$(\hat{3})$]$g_1,\bar{u}_1,\bar{u}_2,v_1,v_2,(f_1 -v_1)g_2\in H^{\infty},$ and $f_1$ is  constant.
	\item [$(4)$]$g_2,\bar{u}_2,(\bar{f}_2-\bar{v}_2 )\bar{u}_1,(f_2 -v_2)g_1\in H^{\infty},f_2$ and $v_2$ are  constant.
	\item [$(\hat{4})$]$g_1,\bar{u}_1,(\bar{f}_1-\bar{v}_1 )\bar{u}_2,(f_1 -v_1)g_2\in H^{\infty},f_1$ and $v_1$ are  constant.
	\item [$(5)$]$\bar{f}_1.\bar{f}_2,g_1,g_2,\bar{v}_1,\bar{v}_2,(\bar{f}_1-\bar{v}_1)\bar{u}_2-(\bar{f}_2-\bar{v}_2)\bar{u}_1\in H^{\infty}.$
	\item [$(6)$]$g_1,g_2,\bar{u}_1,\bar{u}_2\in H^{\infty},$  $f_2$ and $v_1$ are constants.
	\item [$(\hat{6})$]$g_1,g_2,\bar{u}_1,\bar{u}_2\in H^{\infty},$  $f_1$ and $v_2$ are constants.
	\item [$(7)$]$g_1,g_2,\bar{u}_2,\bar{v}_1,\bar{v}_2,(\bar{f}_2 -\bar{v}_2)\bar{u}_1\in H^{\infty},$ and $f_2$ is constant.
	\item [$(\hat{7})$]$g_1,g_2,\bar{u}_1,\bar{v}_1,\bar{v}_2,(\bar{f}_1 -\bar{v}_1)\bar{u}_2\in H^{\infty},$ and $f_1$ is constant.
	\item [$(8)$]There exists a non-zero constant $\lambda$ such that $g_1,g_2,\bar{f}_1-\bar{\lambda}\bar{u}_1,$ 
	$\bar{v}_2-\bar{\lambda}\bar{u}_2,\bar{f}_1 \bar{f}_2 + \bar{v}_1 \bar{v}_2- \bar{f}_1 \bar{v}_2  \in H^{\infty},f_2$ and $v_1$ are constant.
	\item [$(\hat{8})$]There exists a non-zero constant $\mu$ such that $g_1,g_2,\bar{f}_2-\bar{\mu}\bar{u}_2,
	\bar{v}_1-\bar{\mu}\bar{u}_1,\bar{f}_1 \bar{f}_2 + \bar{v}_1 \bar{v}_2- \bar{f}_2 \bar{v}_1 \in H^{\infty},f_1$ and $v_2$ are constant.
	\item [$(9)$]$f_1,f_2,\bar{u}_1,\bar{u}_2,v_1,v_2,(f_2 -v_2)g_1-(f_1-v_1)g_2\in H^{\infty}.$
	\item [$(10)$]$f_1,f_2,g_2,\bar{u}_1,\bar{u}_2,(f_2-v_2)g_1\in H^{\infty},$ and $v_2$ is constant.
	\item [$(\hat{10})$]$f_1,f_2,g_1,\bar{u}_1,\bar{u}_2,(f_1-v_1)g_2\in H^{\infty},$ and $v_1$ is constant.
	\item [$(11)$]There exists a non-zero constant $\lambda$ such that $\bar{u}_1,\bar{u}_2,
	g_1-\lambda v_1,g_2-\lambda f_2,f_1 f_2+v_1v_2-v_1f_2\in H^{\infty},f_1$ and $v_2$ are constant.
	\item [$(\hat{11})$]There exists a non-zero constant $\mu$ such that $\bar{u}_1,\bar{u}_2,g_2-\mu v_2,
g_1-\mu f_1,f_1 f_2+v_1v_2-v_2f_1\in H^{\infty},f_2$ and $v_1$ are constant.
	\item [$(12)$]$f_1,f_2,g_1,g_2,\bar{u}_1,\bar{u}_2,\bar{v}_1,\bar{v}_2\in H^{\infty}.$
	\item [$(13)$]There exist two  non-zero constants $\lambda$ and $\mu$ such that 
	$\bar{f}_1-\bar{\lambda}\bar{u}_1,g_1-\lambda v_1,\\ g_2-\lambda f_2,\bar{v}_2-\bar{\lambda}\bar{u}_2, 
	\bar{f}_2-\bar{\mu}\bar{u}_2,g_2-\mu v_2,g_1-\mu f_1,\bar{v}_1-\bar{\mu}\bar{u}_1\in H^{\infty},$ and $\lambda f_2 v_1-\mu f_1 v_2$ is constant. 
\end{enumerate}
\end{theorem}
\begin{remark}
$(\hat{2})$ is obtained by swapping the function subscripts 1 and 2 in (2)(changing subscript 1 to 2 and 2 to 1) .
$(\hat{8})$ not only involves swapping the function subscripts in (8) but also requires replacing the constant 
$\lambda$ with $\mu.$ The others are similar.
\end{remark}
\begin{proof}
Assume that $R_{H_{1}}R_{H_{2}}=R_{H_{2}}R_{H_{1}}$  and both sides of \eqref{WW} are zero.
Note that  
$\brackbinom{V(\bar{f_{1}})_{-}}{g_{1-}}\otimes \brackbinom{Vf_{2-}}{(\bar{u}_{2})_{-}}
=\brackbinom{V(\bar{u}_{1})_{-}}{v_{1-}}\otimes \brackbinom{Vg_{2-}}{(\bar{v}_{2})_{-}}$ 
if and only if 
one of the following statements holds.
\begin{enumerate}[label=(L\arabic*)]
	\item $\brackbinom{V(\bar{f_{1}})_{-}}{g_{1-}}=\brackbinom{V(\bar{u}_{1})_{-}}{v_{1-}}=0,\
	 R_{H_1}R_{H_2}=\begin{bmatrix}\begin{smallmatrix}T_{f_{1}f_{2}}& H^{*}_{\bar{f}_1\bar{u}_{2}}\\
		H_{v_{1}g_{2}} &  \tilde{T}_{v_{1} v_{2}} \end{smallmatrix}\end{bmatrix}\ \ (see \eqref{1.1.}).$
	\item $\brackbinom{V(\bar{f_{1}})_{-}}{g_{1-}}=\brackbinom{Vg_{2-}}{(\bar{v}_{2})_{-}}=0,\ 
	R_{H_1}R_{H_2}=\begin{bmatrix}\begin{smallmatrix}T_{f_{1}f_{2}}& H^{*}_{\bar{f}_1 \bar{u}_2+\bar{u}_1 \bar{v}_2}\\
		0 &  \tilde{T}_{v_{1} v_{2}} \end{smallmatrix}\end{bmatrix}\ \ (see \eqref{1.2.}).$
	\item $\brackbinom{Vf_{2-}}{(\bar{u}_{2})_{-}}=\brackbinom{V(\bar{u}_{1})_{-}}{v_{1-}}=0,\ 
	R_{H_1}R_{H_2}=\begin{bmatrix}\begin{smallmatrix}T_{f_{1}f_{2}}& 0\\
		H_{g_1 f_2 +v_1 g_2} &  \tilde{T}_{v_{1} v_{2}} \end{smallmatrix}\end{bmatrix}\ \ (see \eqref{2.1.}).$
	\item $\brackbinom{Vf_{2-}}{(\bar{u}_{2})_{-}}=\brackbinom{Vg_{2-}}{(\bar{v}_{2})_{-}}=0,\
	 R_{H_1}R_{H_2}=\begin{bmatrix}\begin{smallmatrix}T_{f_{1}f_{2}}& H^{*}_{\bar{u}_1 \bar{v}_2 }\\
		H_{g_1 f_2} &  \tilde{T}_{v_{1} v_{2}} \end{smallmatrix}\end{bmatrix} \ \ (see \eqref{2.2.}).$
	\item There exists a non-zero constant $\lambda$ such that $\brackbinom{V(\bar{f_{1}})_{-}}{g_{1-}}=\lambda \brackbinom{V(\bar{u}_{1})_{-}}{v_{1-}},
	\brackbinom{Vg_{2-}}{(\bar{v}_{2})_{-}}=\bar{\lambda}\brackbinom{Vf_{2-}}{(\bar{u}_{2})_{-}}, \ \ 
	R_{H_{1}}R_{H_{2}}=
	\begin{bmatrix}\begin{smallmatrix}T_{f_1 f_2}& H^{*}_{1/\bar{\lambda}\bar{f}_1 \bar{v}_2}\\
		H_{\lambda v_{1}f_2} &   \tilde{T}_{v_1 v_2} \end{smallmatrix}\end{bmatrix}\ \ (see \eqref{casec}).$
\end{enumerate}

Similarly,
$\brackbinom{V(\bar{f_{2}})_{-}}{g_{2-}}\otimes \brackbinom{Vf_{1-}}{(\bar{u}_{1})_{-}}
=\brackbinom{V(\bar{u}_{2})_{-}}{v_{2-}}\otimes \brackbinom{Vg_{1-}}{(\bar{v}_{1})_{-}}$ if and only if 
one of the following statements holds.
\begin{enumerate}[label=(R\arabic*)]
	\item $\brackbinom{V(\bar{f_{2}})_{-}}{g_{2-}}=\brackbinom{V(\bar{u}_{2})_{-}}{v_{2-}}=0,\
	 R_{H_2}R_{H_1}=\begin{bmatrix}\begin{smallmatrix}T_{f_{1}f_{2}}& H^{*}_{\bar{f}_2\bar{u}_{1}}\\
	H_{v_{2}g_{1}} &  \tilde{T}_{v_{1} v_{2}} \end{smallmatrix}\end{bmatrix}\ \ (see \eqref{1.1.}).$
	\item $\brackbinom{V(\bar{f_{2}})_{-}}{g_{2-}}=\brackbinom{Vg_{1-}}{(\bar{v}_{1})_{-}}=0, 
	\ R_{H_2}R_{H_1}=\begin{bmatrix}\begin{smallmatrix}T_{f_{1}f_{2}}& H^{*}_{\bar{f}_2 \bar{u}_1+\bar{u}_2 \bar{v}_1}\\
	0 &  \tilde{T}_{v_{1} v_{2}} \end{smallmatrix}\end{bmatrix}\ \ (see \eqref{1.2.}).$
	\item $\brackbinom{Vf_{1-}}{(\bar{u}_{1})_{-}}=\brackbinom{V(\bar{u}_{2})_{-}}{v_{2-}}=0,
	R_{H_2}R_{H_1}=\begin{bmatrix}\begin{smallmatrix}T_{f_{1}f_{2}}& 0\\
	H_{g_2 f_1 +v_2 g_1} &  \tilde{T}_{v_{1} v_{2}} \end{smallmatrix}\end{bmatrix}\ \ (see \eqref{2.1.}).$
	\item $\brackbinom{Vf_{1-}}{(\bar{u}_{1})_{-}}=\brackbinom{Vg_{1-}}{(\bar{v}_{1})_{-}}=0,
	\ R_{H_2}R_{H_1}=\begin{bmatrix}\begin{smallmatrix}T_{f_{1}f_{2}}& H^{*}_{\bar{u}_2 \bar{v}_1 }\\
	H_{g_2 f_1} &  \tilde{T}_{v_{1} v_{2}} \end{smallmatrix}\end{bmatrix}\ \ (see \eqref{2.2.}).$
	\item There exists a non-zero constant $\mu$ such that $\brackbinom{V(\bar{f_{2}})_{-}}{g_{2-}}=\mu \brackbinom{V(\bar{u}_{2})_{-}}{v_{2-}},
	\brackbinom{Vg_{1-}}{(\bar{v}_{1})_{-}}=\bar{\mu}\brackbinom{Vf_{1-}}{(\bar{u}_{1})_{-}}.$
	$R_{H_{2}}R_{H_{1}}=
\begin{bmatrix}\begin{smallmatrix}T_{f_1 f_2}& H^{*}_{1/\bar{\mu}\bar{f}_2 \bar{v}_1}\\
	H_{\mu v_{2}f_1} &   \tilde{T}_{v_1 v_2} \end{smallmatrix}\end{bmatrix}\ \ (see \eqref{casec}).$
\end{enumerate}

Combining the conditions of (L1) to (L5) with those of (R1) to (R5) yields the following table.
For example, we denote the combination of condition (L1) and condition (R1) as (L1)+(R1). 
The following analysis shows that this case corresponds to  (1) in the conclusion.
\renewcommand{\arraystretch}{1.5}
\begin{center}
\begin{tabular}{|c|c|c|c|c|c|}
	\hline
	 & (L1)& (L2) & (L3) & (L4)  &  (L5)\\ 
	\hline
	(R1)              & (1)           & (2)          & (3)         & (4)   & (4) \\ 
	\hline
	(R2)                      & $(\hat{2})$     & (5)          & (6)         & (7)   & (8) \\ 
	\hline
	(R3)                             &  $(\hat{3})$   &  $(\hat{6})$  & (9)         & (10)  & (11) \\ 
	\hline
	(R4)                            &  $(\hat{4})$   &  $(\hat{7})$   & $(\hat{10})$ & (12)  & $(\hat{4})$  \\ 
	\hline
	(R5)                           &  $(\hat{4})$  &  $(\hat{8})$  &  $(\hat{11})$ &  (4)  & (13) \\ 
	\hline
	\end{tabular} 
\end{center}

In case (L1)+(R1), $\bar{f}_1,\bar{f}_2,g_1,g_2,\bar{u}_1,\bar{u}_2,v_1,v_2\in H^{\infty},$ and  
$H_{v_1 g_2}=H_{\bar{f}_1\bar{u}_2}=H_{v_2 g_1}=H_{\bar{f}_2\bar{u}_1}=0.$
Thus, the operator matrices in (L1) and (R1) are equal, and we have derived condition (1) stated in the conclusion.

We proceed to examine three representative cases, the remaining cases admit analogous analysis. 
In each case, the analysis confirms that the equality $R_{H_{1}}R_{H_{2}}=R_{H_{2}}R_{H_{1}}$ 
necessarily requires the stated condition.

In case (L2)+(R1), $\bar{f}_1,\bar{f}_2,g_1,g_2,\bar{u}_2\in H^{\infty},$ $v_2$ is constant and 
$H_{v_2g_1}=H_{\bar{f}_1\bar{u}_2}=0.$ By comparing the operator matrices in condition (L2) and condition (R1), we have
$H_{\bar{f}_2\bar{u}_1}=H_{\bar{u}_1\bar{v}_2}.$ Thus, $(\bar{f}_2 - \bar{v}_2)\bar{u}_1\in H^{\infty}.$

In case (L5)+(R1), $\bar{f}_1 -\bar{\lambda}\bar{u}_1,g_1-\lambda v_1,g_2-\lambda f_2,
\bar{v}_2-\bar{\lambda}\bar{u}_2,\bar{f}_2,g_2,\bar{u}_2,v_2\in H^{\infty}.$
$g_2-\lambda f_2,g_2\in H^{\infty}$ implies $f_2\in H^{\infty},f_2$ is constant.
$\bar{v}_2-\bar{\lambda}\bar{u}_2,\bar{u}_2\in H^{\infty}$ implies $\bar{v}_2\in H^{\infty},v_2$ is constant.
By comparing the operator matrices in condition (L2) and condition (R1),
 we have $v_2g_1-\lambda v_1 f_2,\bar{f}_2\bar{u}_1-(1/\lambda)\bar{f}_1\bar{v}_2\in H^{\infty}.$
Substituting $g_1-\lambda v_1\in H^{\infty}$ into $v_2g_1-\lambda v_1 f_2\in H^{\infty}$ yields $g_1(f_2-v_2)\in H^{\infty}.$
Substituting $\bar{f}_1 -\bar{\lambda}\bar{u}_1\in H^{\infty}$ into 
$\bar{f}_2\bar{u}_1-(1/\lambda)\bar{f}_1\bar{v}_2\in H^{\infty}$ yields $\bar{u}_1(\bar{f}_2-\bar{v}_2)\in H^{\infty}.$
and we have derived condition (4) stated in the conclusion.

In case (L5)+(R5), $\bar{f}_1-\bar{\lambda}\bar{u}_1,g_1-\lambda v_1, g_2-\lambda f_2,\bar{v}_2-\bar{\lambda}\bar{u}_2,
\bar{f}_2-\bar{\mu}\bar{u}_2,g_2-\mu v_2,g_1-\mu f_1,\bar{v}_1-\bar{\mu}\bar{u}_1\in H^{\infty}.$ 
By comparing the operator matrices in condition (L5) and condition (R5), we have
$H_{\lambda v_{1}f_2}=H_{\mu v_{2}f_1}$ and $ H_{1/\bar{\mu}\bar{f}_2 \bar{v}_1}=H_{1/\bar{\lambda}\bar{f}_1 \bar{v}_2}.$ 
Thus, $\lambda v_{1}f_2-\mu v_{2}f_1$ is constant.
\end{proof}
\begin{theorem}\label{mainc2}
	If  both sides of \eqref{WW} are rank one operators,  then $R_{H_{1}}R_{H_{2}}=R_{H_{2}}R_{H_{1}}$ 
	if and only if one of the following statements holds.
\begin{enumerate}
	\item 
	There exist constants $\lambda,\mu,\alpha(\neq 0)$ such that \\ 
	$\bar{f}_1 -\bar{\lambda}\bar{u}_1,g_1 -\lambda v_1,\bar{f}_2 -\bar{\mu}\bar{u}_2,g_2-\mu v_2,$\\
	$\bar{u}_1 -\bar{\alpha}\bar{u}_2,v_1-\alpha v_2,\mu f_1 - g_1-\alpha(\lambda f_2 -g_2),
	\bar{\mu}\bar{u}_1-\bar{v}_1-\bar{\alpha}(\bar{\lambda}\bar{u}_2-\bar{v}_2),$\\
	$g_2(v_1-\alpha v_2)+v_2(\alpha\lambda f_2 -\mu f_1),
	\bar{u}_2\bar{f}_1 -\bar{u}_1 \bar{f}_2 +\bar{\mu}\bar{u}_1\bar{u}_2-\bar{v}_1\bar{u}_2+\bar{\alpha}\bar{v}_2\bar{u}_2-\bar{\lambda}\bar{\alpha}\bar{u}^{2}_2\in H^{\infty}.$
    \item There exist constants $\lambda,\mu,\alpha(\neq 0)$ such that \\ 
    $\bar{u}_1-\bar{\lambda}\bar{f}_1, v_1-\lambda g_1,\bar{f}_2 -\bar{\mu}\bar{u}_2,g_2-\mu v_2,$\\
$\bar{f}_1 -\bar{\alpha}\bar{u}_2,g_1-\alpha v_2,\mu f_1 -g_1-\alpha (f_2 -\lambda g_2),\bar{\mu}_1-\bar{v}_1 -\bar{\alpha}(\bar{u}_2 -\lambda \bar{v}_2),\\
g_2 (v_2-\lambda \alpha v_2)+v_2(\alpha f_2 -\mu f_1),\\
\bar{u}_2\bar{f}_1 -\bar{u}_1 \bar{f}_2 -\bar{\alpha}\bar{u}^{2}_2 +\bar{\lambda}\bar{\alpha}\bar{v}_2\bar{u}_2
+\bar{\mu}\bar{u}_1\bar{u}_2-\bar{v}_1\bar{u}_2\in H^{\infty}.$
\item There exist constants $\lambda,\mu,\alpha(\neq 0)$ such that\\
$f_2 -\bar{\lambda} g_2,\bar{u}_2 -\lambda \bar{v}_2,\bar{f}_2 -\bar{\mu}\bar{u}_2,g_2-\mu v_2,$\\
$\lambda \bar{f}_1 -\bar{u}_1 -\bar{\alpha}\bar{u}_2,\bar{\lambda}g_1-v_1-\alpha v_2,\mu f_1 -g_1-\alpha g_2,
\bar{\mu}\bar{u}_1-\bar{v}_1-\bar{\alpha}\bar{v}_2,$\\
$g_1f_2+v_1g_2-\bar{\lambda}g_1g_2+\alpha v_2 g_2-\mu v_2f_1,\lambda \bar{v}_2\bar{f}_1+\bar{\mu}\bar{u}_1\bar{u}_2-\bar{v}_1\bar{u}_2-\bar{\alpha}\bar{v}_2\bar{u}_2-\bar{u}_1\bar{f}_2\in H^{\infty}.$
\item There exist constants $\lambda,\mu,\alpha(\neq 0)$ such that\\
$g_2 -\bar{\lambda}f_2,\bar{v}_2 -\lambda \bar{u}_2,\bar{f}_2 -\bar{\mu}\bar{u}_2,g_2-\mu v_2,$\\
$\bar{f}_1-\lambda \bar{u}_1-\bar{\alpha}\bar{u}_2,g_1-\bar{\lambda}v_1-\alpha v_2,\mu f_1-g_1-\alpha f_2,
\bar{\mu}\bar{u}_1-\bar{v}_1-\bar{\alpha}\bar{u}_2,$\\
$\bar{\lambda}v_1f_2-\mu v_2f_1+\alpha v_2 f_2,\bar{v}_2\bar{u}_1-\bar{u}_1\bar{f}_2+\bar{u}_2\bar{f}_1-\lambda \bar{u}_1\bar{u}_2
+\bar{\mu}\bar{u}_1\bar{u}_2-\bar{u}_2\bar{v}_1-\bar{\alpha}\bar{u}^{2}_2\in H^{\infty}.$
\item There exist constants $\lambda,\mu,\alpha(\neq 0)$ such that\\
$\bar{f}_1 -\bar{\lambda}\bar{u}_1,g_1 -\lambda v_1,\bar{u}_2-\bar{\mu}\bar{f}_2,v_2-\mu g_2,$\\
$\bar{u}_1 -\bar{\alpha}\bar{f}_2,v_1 -\alpha g_2,f_1-\mu g_1-\alpha (\lambda f_2 -g_2),
\bar{u}_1 -\bar{\mu}\bar{v}_1 -\bar{\alpha}(\bar{\lambda}\bar{u}_2-\bar{v}_2),
\bar{u}_2 f_1 -\mu \bar{v}_1 \bar{f}_2 +\alpha \bar{v}_2 \bar{f}_2 -\bar{\lambda}\bar{\alpha}\bar{u}_2\bar{f}_2,
v_1 g_2 -v_2 g_1 -f_1 g_2+\mu g_1 g_2-\alpha g^{2}_2 +\lambda \alpha g_2 f_2\in H^{\infty}.$
\item There exist constants $\lambda,\mu,\alpha(\neq 0)$ such that\\
$\bar{u}_1-\bar{\lambda}\bar{f}_1,v_1-\lambda g_1,\bar{u}_2-\bar{\mu}\bar{f}_2,v_2-\mu g_2,\\
\bar{f}_1 -\bar{\alpha}\bar{f}_2,g_1 -\alpha g_2,$
$f_1 -\mu g_1 -\alpha(f_2 -\lambda g_2), \bar{u}_1 -\bar{\mu} \bar{v}_1 -\bar{\alpha}(\bar{u}_2 -\bar{\lambda}\bar{v}_2),\\
\bar{u}_2\bar{f}_1 -\bar{\mu}\bar{v}_1 \bar{f}_2 -\bar{\alpha}\bar{u}_2\bar{f}_2 +\bar{\lambda}\bar{\alpha}\bar{v}_2\bar{f}_2,
v_1g_2-v_2g_1-g_2f_1+\mu g_1 g_2+\alpha g_2f_2-\lambda \alpha g^{2}_2\in H^{\infty}.$
\item There exist constants $\lambda,\mu,\alpha(\neq 0)$ such that\\
$f_2 -\bar{\lambda} g_2,\bar{u}_2 -\lambda \bar{v}_2,\bar{u}_2-\bar{\mu}\bar{f}_2,v_2-\mu g_2,$\\
$\lambda \bar{f}_1 -\bar{u}_1-\bar{\alpha}\bar{f}_2,\bar{\lambda}g_1-v_1-\alpha g_2,
f_1-\mu g_1-\alpha g_2,\bar{u}_1-\bar{\mu}\bar{v}_1-\bar{\alpha}\bar{v}_2,$\\
$\lambda \bar{v}_2\bar{f}_1-\bar{\mu}\bar{v}_1\bar{f}_2-\bar{\alpha}\bar{v}_2\bar{f}_2,
g_1f_2-v_2g_1+(v_1 -\bar{\lambda} g_1)g_2+g_2(\mu g_1 -f_1)+\alpha g^{2}_2\in H^{\infty}.$
\item There exist constants $\lambda,\mu,\alpha(\neq 0)$ such that\\
$g_2 -\lambda f_2,\bar{v}_2 -\lambda \bar{u}_2,\bar{u}_2-\bar{\mu}\bar{f}_2,v_2-\mu g_2,$\\
$\bar{f}_1-\bar{\lambda}\bar{u}_1-\bar{\alpha}\bar{f}_2,g_1-\bar{\lambda}v_1-\alpha g_2,
f_1-\mu g_1-\alpha f_2,\bar{u}_1-\bar{\mu}\bar{v}_1-\bar{\alpha}\bar{u}_2,$\\
$\bar{v}_2\bar{u}_1+\bar{u}_2\bar{f}_1-\lambda \bar{u}_1\bar{u}_2-\bar{\mu}\bar{v}_1\bar{f}_2-\bar{\alpha}\bar{u}_2\bar{f}_2,
g_1v_2+g_2f_1-\bar{\lambda}v_1f_2-\mu g_1g_2-\alpha g_2f_2\in H^{\infty}.$
\item There exist constants $\lambda,\mu,\alpha(\neq 0)$ such that\\
$\bar{f}_1 -\bar{\lambda}\bar{u}_1,g_1 -\lambda v_1,f_1 -\bar{\mu}g_1,\bar{u}_1-\mu \bar{v}_1,$\\
$\bar{u}_1-\bar{\alpha}(\mu \bar{f}_2-\bar{u}_2),v_1-\alpha (\bar{\mu}g_2-v_2),
g_1-\alpha(\lambda f_2-g_2),\bar{v}_1-\bar{\alpha}(\bar{\lambda}\bar{u}_2-\bar{v}_2),$\\
$\bar{u}_2\bar{f}_2-\bar{v}_1\bar{u}_2+\alpha (\bar{\lambda}\bar{u}_2-\bar{v}_2)(\bar{u}_2-\mu\bar{f}_2),
v_1g_2-g_2f_1+\alpha(\bar{\mu}g_2-v_2)(\lambda f_2 -g_2)\in H^{\infty}.$
\item There exist constants $\lambda,\mu,\alpha(\neq 0)$ such that\\
$\bar{u}_1-\bar{\lambda}\bar{f}_1, v_1-\lambda g_1,f_1 -\bar{\mu}g_1,\bar{u}_1-\mu \bar{v}_1,$\\
$\bar{f}_1-\bar{\alpha}(\mu \bar{f}_2-\bar{u}_2),g_1-\alpha (\bar{\mu}g_2 -v_2),
g_1-\alpha (f_2-\lambda g_2),\bar{v}_1-\bar{\alpha}(\bar{u}_2-\bar{\lambda}\bar{v}_2),$\\
$\bar{u}_2\bar{f}_1-\bar{v}_1\bar{u}_2+\bar{\alpha}(\bar{\lambda}\bar{v}_2-\bar{u}_2)(\mu \bar{f}_2-\bar{u}_2),
v_1g_2-g_2f_1+\alpha (\bar{\mu}g_2-v_2)(f_2-\lambda g_2)\in H^{\infty}.$
\item There exist constants $\lambda,\mu,\alpha(\neq 0)$ such that\\
$f_2 -\bar{\lambda} g_2,\bar{u}_2 -\lambda \bar{v}_2,f_1 -\bar{\mu}g_1,\bar{u}_1-\mu \bar{v}_1,$\\
$\lambda \bar{f}_1-\bar{u}_1-\bar{\alpha}(\mu \bar{f}_2-\bar{u}_2),\bar{\lambda}g_1-v_1-\alpha (\bar{\mu}g_2-v_2),g_1 -\alpha g_2, \bar{v}_1-\bar{\alpha}\bar{v}_2,$\\
$\bar{v}_1\bar{u}_2-\lambda \bar{v}_2 \bar{f}_1-\bar{\alpha}\bar{v}_2\bar{u}_2+\bar{\alpha}\mu \bar{v}_2\bar{f}_2,
g_1f_2-g_2f_1+(v_1 -\bar{\lambda}g_1)g_2+\alpha g_2(\bar{\mu}g_2 -v_2)\in H^{\infty}.$
\item There exist constants $\lambda,\mu,\alpha(\neq 0)$ such that\\
$g_2 -\bar{\lambda}f_2,\bar{v}_2 -\lambda \bar{u}_2,f_1 -\bar{\mu}g_1,\bar{u}_1-\mu \bar{v}_1,$\\
$\bar{f}_1-\bar{\lambda}\bar{u}_1-\bar{\alpha}(\mu \bar{f}_2-\bar{u}_2),g_1-\bar{\lambda}v_1-\alpha(\bar{\mu}g_2-v_2),
g_1-\alpha f_2,\bar{v}_1-\bar{\alpha}\bar{u}_2,$\\
$\bar{v}_2\bar{u}_1-\bar{v}_1\bar{u}_2+\bar{u}_2(\bar{f}_1-\lambda \bar{u}_1)+\bar{\alpha}\bar{u}_2(\bar{u}_2-\mu \bar{f}_2),
g_2f_1-\bar{\lambda} v_1 f_2-\bar{\mu}\alpha f_2 g_2+\alpha v_2 f_2\in H^{\infty}.$
\item There exist constants $\lambda,\mu,\alpha(\neq 0)$ such that\\
$\bar{f}_1 -\bar{\lambda}\bar{u}_1,g_1 -\lambda v_1,g_1-\bar{\mu}f_1,\bar{v}_1-\mu \bar{u}_1,$\\
$\bar{u}_1-\bar{\alpha}(\bar{f}_2-\mu \bar{u}_2),v_1-\alpha (g_2-\bar{\mu}v_2),
f_1-\alpha(\lambda f_2-g_2),\bar{u}_1-\bar{\alpha}(\bar{\lambda}\bar{u}_2-\bar{v}_2)$\\
$\bar{u}_2\bar{f}_1-\bar{v}_1\bar{u}_2+\bar{\alpha}(\bar{v}_2-\bar{\lambda}\bar{u}_2)(\bar{f}_2-\mu\bar{u}_2),
v_1g_2-g_2f_1+\alpha(g_2-\bar{\mu}v_2)(\lambda f_2-g_2)\in H^{\infty}.$
\item There exist constants $\lambda,\mu,\alpha(\neq 0)$ such that\\
$\bar{u}_1-\bar{\lambda}\bar{f}_1, v_1-\lambda g_1,g_1-\bar{\mu}f_1,\bar{v}_1-\mu \bar{u}_1,$\\
$\bar{f}_1-\bar{\alpha}(\bar{f}_2-\mu \bar{u}_2),g_1-\alpha (g_2-\bar{\mu}v_2),
f_1-\alpha (f_2-\lambda g_2),\bar{u}_1-\bar{\alpha}(\bar{u}_2-\bar{\lambda}\bar{v}_2),$\\
$(\bar{f}_1-\bar{v}_1)\bar{u}_2+\bar{\alpha}(\bar{\lambda}\bar{v}_2-\bar{u}_2)(\bar{f}_2-\mu \bar{u}_2),
(v_1-f_1)g_2+\alpha(g_2-\bar{\mu} v_2)(f_2-\lambda g_2)\in H^{\infty}.$
\item There exist constants $\lambda,\mu,\alpha(\neq 0)$ such that\\
$f_2 -\bar{\lambda} g_2,\bar{u}_2 -\lambda \bar{v}_2,g_1-\bar{\mu}f_1,\bar{v}_1-\mu \bar{u}_1,$\\
$\lambda \bar{f}_1-\bar{u}_1-\bar{\alpha}(\bar{f}_2-\mu\bar{u}_2),\bar{\lambda}g_1-v_1-\alpha(g_2-\bar{\mu}v_2),f_1-\alpha g_2,\bar{u}_1-\bar{\alpha} \bar{v}_2,$\\
$\bar{v}_1\bar{u}_2-\lambda\bar{f}_1\bar{v}_2-\bar{\alpha}\bar{v}_2(\mu \bar{u}_2-\bar{f}_2),
g_1f_2-g_2f_1+(v_1-\bar{\lambda}g_1)g_2+\alpha g_2(g_2-\bar{\mu}v_2)\in H^{\infty}.$
\item There exist constants $\lambda,\mu,\alpha(\neq 0)$ such that\\
$g_2 -\bar{\lambda}f_2,\bar{v}_2 -\lambda \bar{u}_2,g_1-\bar{\mu}f_1,\bar{v}_1-\mu \bar{u}_1,$\\
$\bar{f}_1-\lambda \bar{u}_1-\bar{\alpha}(\bar{f}_2-\mu\bar{u}_2),g_1-\bar{\lambda}v_1-\alpha(g_2-\bar{\mu}v_2),f_1-\alpha f_2,\bar{u}_1-\bar{\alpha}\bar{u}_2,$\\
$\bar{u}_1\bar{v}_2-\bar{v}_1\bar{u}_2+\bar{u}_2(\bar{f}_1-\lambda\bar{u}_1)+\bar{\alpha}\bar{u}_2(\mu\bar{u}_2-\bar{f}_2),
f_1g_2-\bar{\lambda}v_1f_2-\alpha f_2(\bar{\mu}v_2-g_2)\in H^{\infty}.$
\end{enumerate}
\end{theorem}
\begin{remark}
In fact, Condition (1) in Theorem \ref{mainc2} should be written as:
there exist constants $\lambda,\mu,\alpha(\neq 0)$ such that \\ 
$\bar{f}_1 -\bar{\lambda}\bar{u}_1,g_1 -\lambda v_1,\bar{f}_2 -\bar{\mu}\bar{u}_2,g_2-\mu v_2,$\\
$\bar{u}_1 -\bar{\alpha}\bar{u}_2,v_1-\alpha v_2,\mu f_1 - g_1-\alpha(\lambda f_2 -g_2),
\bar{\mu}\bar{u}_1-\bar{v}_1-\bar{\alpha}(\bar{\lambda}\bar{u}_2-\bar{v}_2),$\\
$g_2(v_1-\alpha v_2)+v_2(\alpha\lambda f_2 -\mu f_1),
\bar{u}_2\bar{f}_1 -\bar{u}_1 \bar{f}_2 +\bar{\mu}\bar{u}_1\bar{u}_2-\bar{v}_1\bar{u}_2+\bar{\alpha}\bar{v}_2\bar{u}_2-\bar{\lambda}\bar{\alpha}\bar{u}^{2}_2\in H^{\infty},$
 $\brackbinom{V(\bar{u}_{1})_{-}}{v_{1-}}\neq 0,\bar{\lambda}\brackbinom{Vf_{2-}}{(\bar{u}_{2})_{-}}\neq \brackbinom{Vg_{2-}}{(\bar{v}_{2})_{-}},$ and 
$\brackbinom{V(\bar{u}_{2})_{-}}{v_{2-}}\neq 0,\bar{\mu}\brackbinom{Vf_{1-}}{(\bar{u}_{1})_{-}}\neq \brackbinom{Vg_{1-}}{(\bar{v}_{1})_{-}}.$
However, for the sake of conciseness, we omit 
 $\brackbinom{V(\bar{u}_{1})_{-}}{v_{1-}}\neq 0,\bar{\lambda}\brackbinom{Vf_{2-}}{(\bar{u}_{2})_{-}}\neq \brackbinom{Vg_{2-}}{(\bar{v}_{2})_{-}},$ and 
$\brackbinom{V(\bar{u}_{2})_{-}}{v_{2-}}\neq 0,\bar{\mu}\brackbinom{Vf_{1-}}{(\bar{u}_{1})_{-}}\neq \brackbinom{Vg_{1-}}{(\bar{v}_{1})_{-}},$
as this condition is implied by the rank-1 assumption. The other conditions are similar.
\end{remark}
\begin{proof}
Assume  both sides of \eqref{WW} are rank one operators.
$\brackbinom{V(\bar{f_{1}})_{-}}{g_{1-}}\otimes \brackbinom{Vf_{2-}}{(\bar{u}_{2})_{-}}
-\brackbinom{V(\bar{u}_{1})_{-}}{v_{1-}}\otimes \brackbinom{Vg_{2-}}{(\bar{v}_{2})_{-}}$  has rank one if and only if 
one of the following statements holds.
\begin{enumerate}[label=(L\arabic*)]
	\item $\brackbinom{V(\bar{f_{1}})_{-}}{g_{1-}}=\lambda \brackbinom{V(\bar{u}_{1})_{-}}{v_{1-}},\brackbinom{V(\bar{u}_{1})_{-}}{v_{1-}}\neq 0,\bar{\lambda}\brackbinom{Vf_{2-}}{(\bar{u}_{2})_{-}}\neq \brackbinom{Vg_{2-}}{(\bar{v}_{2})_{-}};$
	\item $\brackbinom{V(\bar{u}_{1})_{-}}{v_{1-}}=\lambda \brackbinom{V(\bar{f_{1}})_{-}}{g_{1-}},\brackbinom{V(\bar{f_{1}})_{-}}{g_{1-}}\neq 0,\bar{\lambda}\brackbinom{Vg_{2-}}{(\bar{v}_{2})_{-}}\neq \brackbinom{Vf_{2-}}{(\bar{u}_{2})_{-}};$
	\item $\brackbinom{Vf_{2-}}{(\bar{u}_{2})_{-}}=\lambda \brackbinom{Vg_{2-}}{(\bar{v}_{2})_{-}},\brackbinom{Vg_{2-}}{(\bar{v}_{2})_{-}}\neq 0,\bar{\lambda}\brackbinom{V(\bar{f_{1}})_{-}}{g_{1-}}\neq \brackbinom{V(\bar{u}_{1})_{-}}{v_{1-}};$
	\item $\brackbinom{Vg_{2-}}{(\bar{v}_{2})_{-}}=\lambda \brackbinom{Vf_{2-}}{(\bar{u}_{2})_{-}},\brackbinom{Vf_{2-}}{(\bar{u}_{2})_{-}}\neq 0,\bar{\lambda}\brackbinom{V(\bar{u}_{1})_{-}}{v_{1-}}\neq \brackbinom{V(\bar{f_{1}})_{-}}{g_{1-}}.$
\end{enumerate}
Similarly,
$\brackbinom{V(\bar{f_{2}})_{-}}{g_{2-}}\otimes \brackbinom{Vf_{1-}}{(\bar{u}_{1})_{-}}
-\brackbinom{V(\bar{u}_{2})_{-}}{v_{2-}}\otimes \brackbinom{Vg_{1-}}{(\bar{v}_{1})_{-}}$ has rank one if and only if 
one of the following statements holds.
\begin{enumerate}[label=(R\arabic*)]
	\item $\brackbinom{V(\bar{f_{2}})_{-}}{g_{2-}}=\mu \brackbinom{V(\bar{u}_{2})_{-}}{v_{2-}},\brackbinom{V(\bar{u}_{2})_{-}}{v_{2-}}\neq 0,\bar{\mu}\brackbinom{Vf_{1-}}{(\bar{u}_{1})_{-}}\neq \brackbinom{Vg_{1-}}{(\bar{v}_{1})_{-}};$
	\item $\brackbinom{V(\bar{u}_{2})_{-}}{v_{2-}}=\mu \brackbinom{V(\bar{f_{2}})_{-}}{g_{2-}},\brackbinom{V(\bar{f_{2}})_{-}}{g_{2-}}\neq 0,\bar{\mu}\brackbinom{Vg_{1-}}{(\bar{v}_{1})_{-}}\neq \brackbinom{Vf_{1-}}{(\bar{u}_{1})_{-}};$
	\item $\brackbinom{Vf_{1-}}{(\bar{u}_{1})_{-}}=\mu \brackbinom{Vg_{1-}}{(\bar{v}_{1})_{-}},\brackbinom{Vg_{1-}}{(\bar{v}_{1})_{-}}\neq 0,\bar{\mu}\brackbinom{V(\bar{f_{2}})_{-}}{g_{2-}}\neq \brackbinom{V(\bar{u}_{2})_{-}}{v_{2-}};$
	\item $\brackbinom{Vg_{1-}}{(\bar{v}_{1})_{-}}=\mu \brackbinom{Vf_{1-}}{(\bar{u}_{1})_{-}},\brackbinom{Vf_{1-}}{(\bar{u}_{1})_{-}}\neq 0,\bar{\mu}\brackbinom{V(\bar{u}_{2})_{-}}{v_{2-}}\neq \brackbinom{V(\bar{f_{2}})_{-}}{g_{2-}}.$
\end{enumerate}

Combining the conditions of (L1) to (L4) with those of (R1) to (R4) yields the following table.
For example, we denote the combination of condition (L1) and condition (R1) as (L1)+(R1). 
The following analysis shows that this case corresponds to  (1) in the conclusion.
Here we only verify several representative cases, and the full verification is provided in Appendix II.
\renewcommand{\arraystretch}{1.2}
\begin{center}
\begin{tabular}{|c|c|c|c|c|c|}
	\hline
	 & (L1)                               & (L2)           & (L3)         & (L4) \\ 
	\hline
	(R1)                              & (1)           & (2)          & (3)         & (4)\\ 
	\hline
	(R2)                           & (5)           & (6)          & (7)         & (8)   \\ 
	\hline
	(R3)                             &  (9)          & (10)         & (11)         & (12)   \\ 
	\hline
	(R4)                               &  (13)        &  (14)         & (15)         & (16)   \\ 
	\hline
	\end{tabular}
\end{center}

In case (L1)+(R1). 
Substituting $\brackbinom{V(\bar{f_{1}})_{-}}{g_{1-}}=\lambda \brackbinom{V(\bar{u}_{1})_{-}}{v_{1-}}
$ and $\brackbinom{V(\bar{f_{2}})_{-}}{g_{2-}}=\mu \brackbinom{V(\bar{u}_{2})_{-}}{v_{2-}}$ into \eqref{WW} 
shows that there exists a nonzero constant $\alpha$ such that\\
$\brackbinom{V(\bar{u}_{1})_{-}}{v_{1-}}=\alpha \brackbinom{V(\bar{u}_{2})_{-}}{v_{2-}},
\bar{\mu}\brackbinom{Vf_{1-}}{(\bar{u}_{1})_{-}}-\brackbinom{Vg_{1-}}{(\bar{v}_{1})_{-}}
=\bar{\alpha}\big(\bar{\lambda}\brackbinom{Vf_{2-}}{(\bar{u}_{2})_{-}}-\brackbinom{Vg_{2-}}{(\bar{v}_{2})_{-}}\big).$ Thus, 
\begin{align}
&\bar{f}_1 -\bar{\lambda}\bar{u}_1\in H^{\infty};\label{7}\\
&g_1 -\lambda v_1\in H^{\infty};\label{8}\\
&\bar{f}_2 -\bar{\mu}\bar{u}_2\in H^{\infty};\label{9}\\
&g_2-\mu v_2\in H^{\infty};\label{10}\\
&\bar{u}_1 -\bar{\alpha}\bar{u}_2\in H^{\infty};\label{11}\\
&v_1-\alpha v_2\in H^{\infty};\label{12}\\
&\mu f_1 - g_1-\alpha(\lambda f_2 -g_2)\in H^{\infty};\label{13}\\
&\bar{\mu}\bar{u}_1-\bar{v}_1-\bar{\alpha}(\bar{\lambda}\bar{u}_2-\bar{v}_2)\in H^{\infty}.\label{14}
\end{align}

Substituting the above conditions into  \eqref{1.2} and \eqref{2.1}, respectively, we obtain
\begin{align*}
&H_{\bar{u}_{2}}T_{\bar{f}_{1}}+\tilde{T}_{\bar{v}_{2}}H_{\bar{u}_{1}}-H_{\bar{u}_{1}}T_{\bar{f}_{2}}-\tilde{T}_{\bar{v}_{1}}H_{\bar{u}_{2}}\\
=&H_{\bar{u}_2\bar{f}_1}-\tilde{T}_{\bar{u}_2}H_{\bar{f}_1}
+\tilde{T}_{\bar{v}_{2}}H_{\bar{u}_{1}}-H_{\bar{u}_1\bar{f}_2}+\tilde{T}_{\bar{u}_1}H_{\bar{f}_2}-\tilde{T}_{\bar{v}_{1}}H_{\bar{u}_{2}}   
&\text{(using \eqref{eq:2m})}\\
=&H_{\bar{u}_2\bar{f}_1-\bar{u}_1\bar{f}_2}-\tilde{T}_{\bar{u}_2}H_{\bar{\lambda}\bar{u}_1}+\tilde{T}_{\bar{v}_{2}}H_{\bar{u}_{1}}
+\tilde{T}_{\bar{u}_1}H_{\bar{\mu}\bar{u}_2}-\tilde{T}_{\bar{v}_{1}}H_{\bar{u}_{2}}  &\text{(using \eqref{7} and \eqref{9})}\\
=&H_{\bar{u}_2\bar{f}_1-\bar{u}_1\bar{f}_2}+\tilde{T}_{\bar{v}_{2}-\bar{\lambda}\bar{u}_2}H_{\bar{u}_1}+
\tilde{T}_{\bar{\mu}\bar{u}_1-\bar{v}_{1}}H_{\bar{u}_{2}}\\
=&H_{\bar{u}_2\bar{f}_1-\bar{u}_1\bar{f}_2}
+H_{(\bar{v}_{2}-\bar{\lambda}\bar{u}_2)\bar{u}_1}-H_{\bar{v}_{2}-\bar{\lambda}\bar{u}_2}T_{\bar{u}_1}
+H_{(\bar{\mu}\bar{u}_1-\bar{v}_{1})\bar{u}_{2}}-H_{\bar{\mu}\bar{u}_1-\bar{v}_{1}}T_{\bar{u}_{2}}  &\text{(using \eqref{eq:2m})}\\
=&H_{\bar{u}_2\bar{f}_1-\bar{u}_1\bar{f}_2+(\bar{v}_{2}-\bar{\lambda}\bar{u}_2)\bar{u}_1+(\bar{\mu}\bar{u}_1-\bar{v}_{1})\bar{u}_{2}}
-H_{\bar{v}_{2}-\bar{\lambda}\bar{u}_2}T_{\bar{u}_1}
-H_{\bar{\mu}\bar{u}_1-\bar{v}_{1}}T_{\bar{u}_{2}}\\
=&H_{\bar{u}_2\bar{f}_1-\bar{u}_1\bar{f}_2+(\bar{v}_{2}-\bar{\lambda}\bar{u}_2)\bar{u}_1+(\bar{\mu}\bar{u}_1-\bar{v}_{1})\bar{u}_{2}}
-H_{\bar{v}_{2}-\bar{\lambda}\bar{u}_2}T_{\bar{u}_1}
+H_{\bar{\alpha}(\bar{v}_{2}-\bar{\lambda}\bar{u}_2)}T_{\bar{u}_{2}} &\text{(using \eqref{14})}\\
=&H_{\bar{u}_2\bar{f}_1-\bar{u}_1\bar{f}_2+(\bar{v}_{2}-\bar{\lambda}\bar{u}_2)\bar{u}_1+(\bar{\mu}\bar{u}_1-\bar{v}_{1})\bar{u}_{2}}
+H_{\bar{v}_{2}-\bar{\lambda}\bar{u}_2}T_{\bar{\alpha}\bar{u}_{2}-\bar{u}_1}\\
=&H_{\bar{u}_2\bar{f}_1-\bar{u}_1\bar{f}_2+(\bar{v}_{2}-\bar{\lambda}\bar{u}_2)\bar{u}_1+(\bar{\mu}\bar{u}_1-\bar{v}_{1})\bar{u}_{2}}
+H_{(\bar{v}_{2}-\bar{\lambda}\bar{u}_2)(\bar{\alpha}\bar{u}_{2}-\bar{u}_1)} &\text{(using \eqref{HT} and \eqref{11})}\\
=&H_{\bar{u}_2\bar{f}_1-\bar{u}_1\bar{f}_2+(\bar{v}_{2}-\bar{\lambda}\bar{u}_2)\bar{u}_1+(\bar{\mu}\bar{u}_1-\bar{v}_{1})\bar{u}_{2}
+(\bar{v}_{2}-\bar{\lambda}\bar{u}_2)(\bar{\alpha}\bar{u}_{2}-\bar{u}_1)}\\
=&H_{\bar{u}_2\bar{f}_1 -\bar{u}_1 \bar{f}_2 +\bar{\mu}\bar{u}_1\bar{u}_2-\bar{v}_1\bar{u}_2+\bar{\alpha}\bar{v}_2\bar{u}_2
-\bar{\lambda}\bar{\alpha}\bar{u}^{2}_2},
\end{align*}
and 
\begin{align*}
&H_{g_{1}}T_{f_{2}}+\tilde{T}_{v_{1}}H_{g_{2}}-H_{g_{2}}T_{f_{1}}-\tilde{T}_{v_{2}}H_{g_{1}}\\
=&H_{\lambda v_{1}}T_{f_{2}}+\tilde{T}_{v_{1}}H_{g_{2}}-H_{\mu v_{2}}T_{f_{1}}-\tilde{T}_{v_{2}}H_{g_{1}}
\  \  &\text{(using \eqref{8} and \eqref{10})}\\
=&H_{\lambda v_{1}f_{2}}-\tilde{T}_{\lambda v_{1}}H_{f_{2}}+\tilde{T}_{v_{1}}H_{g_{2}}-H_{\mu v_2 f_1}+
\tilde{T}_{\mu v_{2}}H_{f_{1}}-\tilde{T}_{v_{2}}H_{g_{1}}\\
=&H_{\lambda v_{1}f_{2}-\mu v_2 f_1}+\tilde{T}_{v_{1}}H_{g_{2}-\lambda f_2}-\tilde{T}_{ v_{2}}H_{g_1 -\mu f_{1}}\\
=&H_{\lambda v_{1}f_{2}-\mu v_2 f_1}+\tilde{T}_{v_{1}}H_{g_{2}-\lambda f_2}-\tilde{T}_{ v_{2}}H_{\alpha(g_2 -\lambda f_{2})}
\  \  &\text{(using \eqref{13}) }\\
=&H_{\lambda v_{1}f_{2}-\mu v_2 f_1}+\tilde{T}_{v_{1}-\alpha v_2}H_{g_{2}-\lambda f_2}\\
=&H_{\lambda v_{1}f_{2}-\mu v_2 f_1+(v_{1}-\alpha v_2)(g_{2}-\lambda f_2)}\  \  
\ \  &\text{(using \eqref{12} and \eqref{HT})}\\
=&H_{g_2(v_1-\alpha v_2)+v_2(\alpha \lambda f_2-\mu f_1)}.
\end{align*}
By Theorem \ref{main2}, we get condition (1) in the conclusion.
\end{proof}
\begin{theorem}\label{R2}
Assume  both sides of \eqref{WW} are rank two operators, $R_{H_{1}}R_{H_{2}}=R_{H_{2}}R_{H_{1}}$ if and only if
there exist constants $a,b,c,d$ such that
\begin{align*}
&\bar{f}_1 -\bar{a}\bar{f}_2 -\bar{b}\bar{u}_2,\ \
g_1 -a g_2 -b v_2,\ \
\bar{u}_1 -\bar{c}\bar{f}_2 -\bar{d}\bar{u}_2,\ \
v_1 -c g_2 - dv_2,\\
&f_1 -a f_2+ cg_2,\ \
\bar{u}_1 -\bar{a}\bar{u}_2+\bar{c}\bar{v}_2,\ \
g_1 +bf_2 - dg_2,\  \
\bar{v}_1+\bar{b}\bar{u}_2-\bar{d}\bar{v}_2,\\
&\bar{u}_2\bar{f}_1 -\bar{v}_1 \bar{u}_2 -\bar{a}\bar{u}_2 \bar{f}_2 -\bar{b}\bar{u}^{2}_2 +\bar{c}\bar{v}_2\bar{f}_2+\bar{d}\bar{v}_2\bar{u}_2,\\
&v_1 g_2 - f_1 g_2 +ag_2 f_2 +bv_2 f_2-cg_2^{2}-dv_2 g_2 \in H^{\infty}.
\end{align*}
\end{theorem}
\begin{proof}
If  both sides of \eqref{WW} are rank two operators,
then there are constants $a,b,c$ and $d$, such that 
\begin{align}\label{abcd}
\brackbinom{V(\bar{f_{1}})_{-}}{g_{1-}}=a\brackbinom{V(\bar{f_{2}})_{-}}{g_{2-}}+b\brackbinom{V(\bar{u}_{2})_{-}}{v_{2-}},\quad 
\brackbinom{V(\bar{u}_{1})_{-}}{v_{1-}}=c\brackbinom{V(\bar{f_{2}})_{-}}{g_{2-}}+d\brackbinom{V(\bar{u}_{2})_{-}}{v_{2-}}.
\end{align}
Substituting \eqref{abcd} into  \eqref{WW}, we obtain
\begin{align}\label{babcd}
\brackbinom{Vf_{1-}}{(\bar{u}_{1})_{-}}=\bar{a}\brackbinom{Vf_{2-}}{(\bar{u}_{2})_{-}}-\bar{c}\brackbinom{Vg_{2-}}{(\bar{v}_{2})_{-}},\quad
\brackbinom{Vg_{1-}}{(\bar{v}_{1})_{-}}=-\bar{b}\brackbinom{Vf_{2-}}{(\bar{u}_{2})_{-}}+\bar{d}\brackbinom{Vg_{2-}}{(\bar{v}_{2})_{-}}.	
\end{align}
By \eqref{abcd} and \eqref{babcd}, we have
\begin{align}
	&\bar{f}_1 -\bar{a}\bar{f}_2 -\bar{b}\bar{u}_2\in H^{\infty},\label{a1}\\
	&g_1 -a g_2 -b v_2\in H^{\infty},\label{a2}\\
	&\bar{u}_1 -\bar{c}\bar{f}_2 -\bar{d}\bar{u}_2\in H^{\infty},\label{a3}\\
	&v_1 -c g_2 - dv_2\in H^{\infty},\label{a4}\\
	&f_1 -a f_2+ cg_2\in H^{\infty},\label{a5}\\
	&\bar{u}_1 -\bar{a}\bar{u}_2+\bar{c}\bar{v}_2\in H^{\infty},\label{a6}\\
	&g_1 +bf_2 - dg_2\in H^{\infty},\label{a7}\\
	&\bar{v}_1+\bar{b}\bar{u}_2-\bar{d}\bar{v}_2\in H^{\infty}.\label{a8}
\end{align}
Substituting the above conditions into  \eqref{1.2} and \eqref{2.1}, respectively, we obtain
\begin{align*}
&H_{\bar{u}_{2}}T_{\bar{f}_{1}}+\tilde{T}_{\bar{v}_{2}}H_{\bar{u}_{1}}-H_{\bar{u}_{1}}T_{\bar{f}_{2}}-\tilde{T}_{\bar{v}_{1}}H_{\bar{u}_{2}}\\
=&H_{\bar{u}_2\bar{f}_1}-\tilde{T}_{\bar{u}_{2}}H_{\bar{f}_{1}}+\tilde{T}_{\bar{v}_{2}}H_{\bar{u}_{1}}-H_{\bar{u}_{1}}T_{\bar{f}_{2}}
-H_{\bar{v}_{1}\bar{u}_{2}}+H_{\bar{v}_{1}}T_{\bar{u}_{2}} &\text{(using \eqref{eq:2m})}\\
=&H_{\bar{u}_2\bar{f}_1-\bar{v}_{1}\bar{u}_{2}}-\tilde{T}_{\bar{u}_{2}}H_{\bar{a}\bar{f}_2+\bar{b}\bar{u}_2}
+\tilde{T}_{\bar{v}_{2}}H_{\bar{c}\bar{f}_2+\bar{d}\bar{u}_2}-H_{\bar{a}\bar{u}_2-\bar{c}\bar{v}_2}T_{\bar{f}_{2}}
+H_{-\bar{b}\bar{u}_2+\bar{d}\bar{v}_2}T_{\bar{u}_{2}}\\
 &\hspace{15em}\text{(using  \eqref{a1},\eqref{a3},\eqref{a6},\eqref{a8})}\\
=&H_{\bar{u}_2\bar{f}_1-\bar{v}_{1}\bar{u}_{2}}-\bar{a}(\tilde{T}_{\bar{u}_{2}}H_{\bar{f}_2}+H_{\bar{u}_{2}}T_{\bar{f}_2})
-\bar{b}(\tilde{T}_{\bar{u}_{2}}H_{\bar{u}_2}+H_{\bar{u}_{2}}T_{\bar{u}_2})\\
&+\bar{c}(\tilde{T}_{\bar{v}_2}H_{\bar{f}_2}+H_{\bar{v}_2}T_{\bar{f}_2})
+\bar{d}(\tilde{T}_{\bar{v}_2}H_{\bar{u}_2}+H_{\bar{v}_2}T_{\bar{u}_2})\\
=&H_{\bar{u}_2\bar{f}_1-\bar{v}_{1}\bar{u}_{2}}-\bar{a}H_{\bar{u}_{2}\bar{f}_2}
-\bar{b}H_{\bar{u}^{2}_{2}}+\bar{c}H_{\bar{v}_2\bar{f}_2}+\bar{d}H_{\bar{v}_2\bar{u}_2} & \text{(using \eqref{eq:2m})}\\
=&H_{\bar{u}_2\bar{f}_1-\bar{v}_{1}\bar{u}_{2}-\bar{a}\bar{u}_{2}\bar{f}_2-\bar{b}\bar{u}^{2}_{2}
+\bar{c}\bar{v}_2\bar{f}_2+\bar{d}\bar{v}_2\bar{u}_2},
\end{align*}
and
\begin{align*}
&H_{g_{1}}T_{f_{2}}+\tilde{T}_{v_{1}}H_{g_{2}}-H_{g_{2}}T_{f_{1}}-\tilde{T}_{v_{2}}H_{g_{1}}\\
=&H_{g_{1}}T_{f_{2}}+H_{v_{1}g_2}-H_{v_1}T_{g_2}-H_{f_1 g_2}+\tilde{T}_{g_{2}}H_{f_{1}}-\tilde{T}_{v_{2}}H_{g_{1}} &\text{(using \eqref{eq:2m})}\\
=&H_{v_{1}g_2-f_1 g_2}+H_{g_{1}}T_{f_{2}}-H_{v_1}T_{g_2}+\tilde{T}_{g_{2}}H_{f_{1}}-\tilde{T}_{v_{2}}H_{g_{1}}\\
=&H_{v_{1}g_2-f_1 g_2}+H_{ag_{2}+bv_2}T_{f_{2}}-H_{cg_2+dv_2}T_{g_2}+\tilde{T}_{g_{2}}H_{af_2-cg_2}-\tilde{T}_{v_{2}}H_{dg_2-bf_2}\\
&\hspace{9em}\text{(using \eqref{a2},\eqref{a4},\eqref{a5},\eqref{a7})}\\
=&H_{v_{1}g_2-f_1 g_2}+a(H_{g_{2}}T_{f_{2}}+\tilde{T}_{g_{2}}H_{f_2})+b(H_{v_2}T_{f_{2}}+\tilde{T}_{v_{2}}H_{f_2})\\
&-c(H_{g_2}T_{g_2}+\tilde{T}_{g_{2}}H_{g_2})-d(H_{v_2}T_{g_2}+\tilde{T}_{v_{2}}H_{g_2})  \\
=&H_{v_{1}g_2-f_1 g_2}+aH_{g_{2}f_{2}}+bH_{v_2f_{2}}-cH_{g^{2}_2}-dH_{v_2g_2} &\text{(using \eqref{eq:2m})}\\
=&H_{v_{1}g_2-f_1 g_2+ag_{2}f_{2}+bv_2f_{2}-cg^{2}_2-dv_2g_2}.
\end{align*}
By Theorem \ref{main2},  the proof is complete.
\end{proof}
\section{Applications}\label{5}
\subsection{Toeplitz+Hankel operators}

$\mathbb{H}_{f}$ denotes the Hankel operator on $H^{2}$ such that
\begin{align}
\mathbb{H}_{f}x=P_{+}(fJ x),\quad  x\in H^{2},
\end{align}
where $J$ is the flip operator on $L^{2}$ defined by
$Jf(z)=\bar{z}f(\bar{z}).$ For $f\in L^{\infty},$ let $\tilde{f}(z)=f(\bar{z})$ and $f^{*}(z)=\overline{f(\bar{z})}.$
\begin{lemma}
Let $T$ is a Toeplitz operator, $\mathbb{H}$ is a Hankel operator, then the operator matrices 
$\begin{bmatrix}T+ \mathbb{H}& 0\\
  0 &  T-\mathbb{H}\end{bmatrix}$
and
$\begin{bmatrix}T&  \mathbb{H}\\
   \mathbb{H} &  T \end{bmatrix}$
are unitarily equivalent via the unitary operator matrix $\frac{1}{\sqrt{2}}\begin{bmatrix}I& I\\
  I & -I\end{bmatrix}.$
\end{lemma}
\begin{lemma}
If $f,g,u,v\in L^{\infty},$ then 
$\begin{bmatrix}T_{f}& \mathbb{H}_{u}\\
  \mathbb{H}_{g} &  T_{v}\end{bmatrix}$
and
$R_{\begin{bmatrix}
	\begin{smallmatrix}
	f& u\\
	  \tilde{g} & \tilde{v}
	\end{smallmatrix}
	\end{bmatrix}}$
are unitarily equivalent via the unitary operator matrix $\begin{bmatrix}I& 0\\
  0 & J\end{bmatrix}.$
\end{lemma}
Hence, 
$\begin{bmatrix}T_f+ \mathbb{H}_g& 0\\
  0 &  T_f-\mathbb{H}_g\end{bmatrix}$ is unitarily equivalent to $R_{\begin{bmatrix}
	\begin{smallmatrix}
	f& g\\
	  \tilde{g} & \tilde{f}
	\end{smallmatrix}
	\end{bmatrix}}.$

\begin{theorem}\label{TH}
Let $f_1,g_1,f_2,g_2\in L^{\infty}.$ The operators
$T_{f_1}+ \mathbb{H}_{g_1}$ and $T_{f_2}+ \mathbb{H}_{g_2}$, and $T_{f_1}- \mathbb{H}_{g_1}$ and $T_{f_2}- \mathbb{H}_{g_2}$
commute simultaneously if and only if one of the following statements holds.
\begin{align*}
	&H_{\bar{g}_{2}}T_{\bar{f}_{1}}+\tilde{T}_{f^{*}_{2}}H_{\bar{g}_{1}}
	-H_{\bar{g}_{1}}T_{\bar{f}_{2}}-\tilde{T}_{f^{*}_{1}}H_{\bar{g}_{2}}=0;\\
	&H_{\tilde{g}_{1}}T_{f_{2}}+\tilde{T}_{\tilde{f}_{1}}H_{\tilde{g}_{2}}-H_{\tilde{g}_{2}}T_{f_{1}}-\tilde{T}_{\tilde{f}_{2}}H_{\tilde{g}_{1}}=0;\\
	&\brackbinom{V(\bar{f_{1}})_{-}}{(\tilde{g}_{1})_{-}}\otimes \brackbinom{V(f_{2})_{-}}{(\bar{g}_{2})_{-}}
	-\brackbinom{V(\bar{g}_{1})_{-}}{(\tilde{f}_{1})_{-}}\otimes \brackbinom{V(\tilde{g}_{2})_{-}}{(f^{*}_{2})_{-}} \nonumber \\
	=&\brackbinom{V(\bar{f_{2}})_{-}}{(\tilde{g}_{2})_{-}}\otimes \brackbinom{V(f_{1})_{-}}{(\bar{g}_{1})_{-}}
	-\brackbinom{V(\bar{g}_{2})_{-}}{(\tilde{f}_{2})_{-}}\otimes \brackbinom{V(\tilde{g}_{1})_{-}}{(f^{*}_{1})_{-}}.
	\end{align*}

\end{theorem}

In Theorem \ref{TH}, if $g_1=g_2=0,$ then $T_{f_1}$ and $T_{f_2}$ commute if and only if $(\bar{f}_1)_{-}\otimes (f_2)_{-}=(\bar{f}_2)_{-}\otimes (f_1)_{-}.$ 
Thus we obtain the clasical Brown-Halmos theorem.
\begin{theorem}\cite[Theorem 9]{Brown1963}
A necessary and sufficient condition that two Toeplitz operators commute is that either both be analytic, or both be co-analytic, or one be a linear function of the other.
\end{theorem}
In Theorem \ref{TH}, if $f_1=f_2=0,$ then $\mathbb{H}_{g_1}$ and $\mathbb{H}_{g_2}$ commute if and only if
$(\tilde{g}_1)_{-}\otimes (\bar{g}_2)_{-}=(\tilde{g}_2)_{-}\otimes (\bar{g}_1)_{-}.$ Thus we get the following theorem.
\begin{theorem}\cite[Theorem 4.4.8]{martinez2007introduction}
Let $g_1,g_2\in L^{\infty},$ and suppose that $\mathbb{H}_{g_2}\neq 0.$ $\mathbb{H}_{g_1}$ and $\mathbb{H}_{g_2}$ commute if and only if 
there exists a complex number $c$ such that $\mathbb{H}_{g_1}=c\mathbb{H}_{g_2}.$
\end{theorem}

\subsection{(Asymmetric) Dual Truncated Toeplitz operators}

To each nonconstant inner function $\theta$, we denote the model space
$K_{\theta}:=H^{2}\ominus \theta H^{2}.$
The operator $P_{+}-M_{\theta}P_{+}M_{\bar{\theta}}$ is the orthogonal projection of $L^{2}$ onto $K^{2}_{\theta}.$
Note that $K^{\bot}_{\theta}=L^{2}\ominus K^{2}_{\theta}=\theta H^{2}\oplus \bar{z}\overline{H^{2}}.$ 

For $\phi$ in $L^{\infty},$ and two nonconstant inner functions $\theta$ and $\alpha,$ 
the asymmetric dual truncated Toeplitz operator $D^{\theta,\alpha}_{\phi}$ is defined by
\begin{align*}
D^{\theta,\alpha}_{\phi}:   K^{\bot}_{\theta} & \longmapsto  K^{\bot}_{\alpha},\\
 x &\longmapsto (P_{-}+\alpha P_{+}\bar{\alpha})(\phi x).
\end{align*}
If $\theta=\alpha,$ then $D^{\theta,\alpha}_{\phi}=D^{\theta}_{\phi}$ is called a dual truncated Toeplitz operator. Note that
\begin{align}\label{DTM}
D^{\theta,\alpha}_{\phi}
&=\begin{bmatrix}M_{\alpha} & 0\\
  0 &  I_{(H^{2})^{\perp}}\end{bmatrix}
\begin{bmatrix}T_{\bar{\alpha}\theta\phi} & H^{*}_{\bar{\phi}\alpha}\\
H_{\phi\theta} &  \tilde{T}_{\phi}\end{bmatrix}
\begin{bmatrix}M_{\bar{\theta}} & 0\\
  0 &  I_{(H^{2})^{\perp}}\end{bmatrix},
\end{align}
where
\begin{align*}
\begin{bmatrix}M_{\bar{\theta}} & 0\\
  0 &  I_{(H^{2})^{\perp}}\end{bmatrix}:\theta H^{2}\oplus (H^{2})^{\perp}\rightarrow L^{2}=H^{2}\oplus (H^{2})^{\perp}
\end{align*}
and
\begin{align*}
\begin{bmatrix}M_{\alpha} & 0\\
  0 &  I_{(H^{2})^{\perp}}\end{bmatrix}: L^{2}=H^{2}\oplus (H^{2})^{\perp}\rightarrow \alpha H^{2}\oplus (H^{2})^{\perp}
\end{align*}
are unitary \cite{sang2023algebras}.

We will establish the necessary and sufficient conditions for 
the product of two asymmetric dual truncated Toeplitz operators being a asymmetric dual truncated Toeplitz operator.

\begin{theorem}\label{ADTP}
Let $\phi,\psi,\sigma\in L^{\infty},$ and let $\alpha,\beta,\theta$ be three non-constant inner functions.
Then $D^{\alpha,\beta}_{\psi} D^{\theta,\alpha}_{\phi}=D^{\theta,\beta}_{\sigma}$ if and only if
one of the following statements holds.
\begin{enumerate}
\item $\psi,\bar{\alpha}\beta\bar{\psi}\in H^{\infty};$
\item $\bar{\phi},\theta\bar{\alpha}\phi\in H^{\infty};$
\item $\bar{\phi},(1-\lambda \bar{\alpha})\theta\phi,(\bar{\alpha}-\bar{\lambda})\beta\bar{\psi},(\alpha-\lambda)\psi,
(1-\lambda\bar{\alpha})\theta\phi\psi
\in H^{\infty},$ where $\lambda$ is a constant with $|\lambda|< 1;$
\item $\psi,(1-\lambda \bar{\alpha})\beta\bar{\psi},(\bar{\alpha}-\bar{\lambda})\theta\phi,(\alpha-\lambda)\bar{\phi},
(1-\lambda\bar{\alpha})\beta\bar{\phi}\bar{\psi}
\in H^{\infty},$ where $\lambda$ is a constant with $|\lambda|< 1.$
\end{enumerate}
In this case, $\sigma=\phi\psi.$
\end{theorem}
\begin{proof}
By \eqref{DTM}, we have 
$D^{\alpha,\beta}_{\psi} D^{\theta,\alpha}_{\phi}=D^{\theta,\beta}_{\sigma}$ if and only if 
\begin{align}\label{RD}
	R_{\begin{bmatrix}
	\begin{smallmatrix}
		\bar{\beta}\alpha\psi & \bar{\beta}\psi\\
	   \alpha \psi & \psi
	\end{smallmatrix}
	\end{bmatrix}}
	R_{\begin{bmatrix}
	\begin{smallmatrix}
		\bar{\alpha}\theta\phi & \bar{\alpha}\phi\\
	   \theta \phi & \phi
	\end{smallmatrix}
	\end{bmatrix}}
=R_{\begin{bmatrix}
	\begin{smallmatrix}
	\bar{\beta}\theta\sigma & \bar{\beta}\sigma\\
	   \theta \sigma & \sigma
	\end{smallmatrix}
	\end{bmatrix}}.
\end{align}
\eqref{RD} implies that
\begin{align}\label{ADT}
\brackbinom{V(\beta\bar{\alpha}\bar{\psi})_{-}}{(\alpha \psi)_{-}}\otimes \brackbinom{V(\bar{\alpha}\theta\phi)_{-}}{(\alpha\bar{\phi})_{-}}
=
\brackbinom{V(\beta\bar{\psi})_{-}}{\psi_{-}}\otimes \brackbinom{V(\theta\phi)_{-}}{(\bar{\phi})_{-}},
\end{align}
by Theorem \ref{main1}, and $\sigma=\phi\psi$ by \eqref{1.1.},\eqref{1.2.},\eqref{2.1.},\eqref{2.2.} and \eqref{casec}.
We consider the following cases.
Let $\Psi={\begin{bmatrix}
\begin{smallmatrix}
	\bar{\beta}\alpha\psi & \bar{\beta}\psi\\
   \alpha \psi & \psi
\end{smallmatrix}
\end{bmatrix}},\Phi={\begin{bmatrix}
\begin{smallmatrix}
	\bar{\alpha}\theta\phi & \bar{\alpha}\phi\\
   \theta \phi & \phi
\end{smallmatrix}
\end{bmatrix}}.$
\begin{enumerate}[label=\textbf{Case \Roman*:}, leftmargin=2cm]
\item
If $\brackbinom{V(\beta\bar{\alpha}\bar{\psi})_{-}}{(\alpha \psi)_{-}}=\brackbinom{V(\beta\bar{\psi})_{-}}{\psi_{-}}=0,$
then $\beta\bar{\alpha}\bar{\psi},\alpha \psi,\beta\bar{\psi},\psi\in H^{\infty}$, which is equivalent to $\beta\bar{\alpha}\bar{\psi},\psi\in H^{\infty}.$
Using \eqref{1.1.}, 
$R_{\Psi}R_{\Phi}=\begin{bmatrix}
	\begin{smallmatrix}
	T_{\bar{\beta}\theta\phi\psi}& H^{*}_{\beta\bar{\phi}\bar{\psi}}\\
	H_{\theta \phi\psi} &  \tilde{T}_{\phi\psi}
	\end{smallmatrix}
 \end{bmatrix}.$ Thus we have obtained condition (1) in the conclusion.
\item
If $\brackbinom{V(\beta\bar{\alpha}\bar{\psi})_{-}}{(\alpha \psi)_{-}}=\brackbinom{V(\theta\phi)_{-}}{(\bar{\phi})_{-}}=0,$ 
then $\beta\bar{\alpha}\bar{\psi},\alpha \psi,\theta\phi,\bar{\phi}\in H^{\infty}.$ Using \eqref{1.2.}, we have
$R_{\Psi}R_{\Phi}=
\begin{bmatrix}
	\begin{smallmatrix}
	T_{\bar{\beta}\theta\phi\psi}& 2H^{*}_{\beta\bar{\phi}\bar{\psi}}\\
	0 &  \tilde{T}_{\phi\psi}
	\end{smallmatrix}
\end{bmatrix}.$  If \eqref{RD} holds, then $
R_{\Psi}R_{\Phi}=\begin{bmatrix}
	\begin{smallmatrix}
	T_{\bar{\beta}\theta\phi\psi}& H^{*}_{\beta\bar{\phi}\bar{\psi}}\\
	H_{\theta \phi\psi} &  \tilde{T}_{\phi\psi}
	\end{smallmatrix}
 \end{bmatrix}.$ Hence, $\theta \phi\psi,\beta\bar{\phi}\bar{\psi}\in H^{\infty}.$
Note that $\beta\bar{\alpha}\bar{\psi},\bar{\phi}\in H^{\infty}$ implies that $\beta\bar{\phi}\bar{\psi}\in H^{\infty}.$
This case corresponds to condition (3) in the conclusion with $\lambda=0.$
\item
If $\brackbinom{V(\bar{\alpha}\theta\phi)_{-}}{(\alpha\bar{\phi})_{-}}=\brackbinom{V(\beta\bar{\psi})_{-}}{\psi_{-}}=0,$
then $\bar{\alpha}\theta\phi,\alpha\bar{\phi},\beta\bar{\psi},\psi\in H^{\infty}.$
Using \eqref{2.1.},
$R_{\Psi}R_{\Phi}=
\begin{bmatrix}
	\begin{smallmatrix}
	T_{\bar{\beta}\theta\phi\psi}& 0\\
	2H_{\theta \phi\psi} &  \tilde{T}_{\phi\psi}
	\end{smallmatrix}
\end{bmatrix}.$ If \eqref{RD} holds,  then $
R_{\Psi}R_{\Phi}=\begin{bmatrix}
	\begin{smallmatrix}
	T_{\bar{\beta}\theta\phi\psi}& H^{*}_{\beta\bar{\phi}\bar{\psi}}\\
	H_{\theta \phi\psi} &  \tilde{T}_{\phi\psi}
	\end{smallmatrix}
 \end{bmatrix}.$ Hence, $\theta \phi\psi,\beta\bar{\phi}\bar{\psi}\in H^{\infty}.$
Note that $\bar{\alpha}\theta\phi,\psi\in H^{\infty}$ implies that $\theta \phi\psi\in H^{\infty}.$
This case corresponds to the statement (4) with $\lambda=0.$

\item
If $\brackbinom{V(\bar{\alpha}\theta\phi)_{-}}{(\alpha\bar{\phi})_{-}}=\brackbinom{V(\theta\phi)_{-}}{(\bar{\phi})_{-}}=0,$
then $\bar{\alpha}\theta\phi,\alpha\bar{\phi},\theta\phi,\bar{\phi}\in H^{\infty},$ which is equivalent to $\bar{\phi}, \bar{\alpha}\theta\phi\in H^{\infty}.$
Using \eqref{2.2.},
$R_{\Psi}R_{\Phi}=\begin{bmatrix}
	\begin{smallmatrix}
	T_{\bar{\beta}\theta\phi\psi}& H^{*}_{\beta\bar{\phi}\bar{\psi}}\\
	H_{\theta \phi\psi} &  \tilde{T}_{\phi\psi}
	\end{smallmatrix}
 \end{bmatrix}.$ Thus we have obtained condition (2) in the conclusion.
 \item
If the both sides of \eqref{ADT} are nonzero, then
there exists a nonzero constant $\lambda$ such that
$\brackbinom{V(\beta\bar{\alpha}\bar{\psi})_{-}}{(\alpha \psi)_{-}}=\lambda \brackbinom{V(\beta\bar{\psi})_{-}}{\psi_{-}}$
and $\brackbinom{V(\theta\phi)_{-}}{(\bar{\phi})_{-}}=\bar{\lambda}\brackbinom{V(\bar{\alpha}\theta\phi)_{-}}{(\alpha\bar{\phi})_{-}}.$
Therefore,
$(\bar{\alpha}-\bar{\lambda})\beta\bar{\psi},(\alpha-\lambda)\psi,(1-\lambda\bar{\alpha})\theta\phi,
(1-\bar{\lambda}\alpha)\bar{\phi}\in H^{\infty}.$
Using \eqref{casec},
$R_{\Psi}R_{\Phi}=
\begin{bmatrix}
	\begin{smallmatrix}
	T_{\bar{\beta}\theta\phi\psi}& H^{*}_{1/\bar{\lambda}\beta\bar{\alpha}\bar{\phi}\bar{\psi}}\\
	H_{\lambda\theta\bar{\alpha} \phi\psi} &  \tilde{T}_{\phi\psi}
	\end{smallmatrix}
\end{bmatrix}.$ If \eqref{RD} holds, then $
R_{\Psi}R_{\Phi}=\begin{bmatrix}
	\begin{smallmatrix}
	T_{\bar{\beta}\theta\phi\psi}& H^{*}_{\beta\bar{\phi}\bar{\psi}}\\
	H_{\theta \phi\psi} &  \tilde{T}_{\phi\psi}
	\end{smallmatrix}
 \end{bmatrix}.$ Hence, $(1-\lambda \bar{\alpha})\theta \phi\psi,(\bar{\alpha}-\bar{\lambda})\bar{\phi}\bar{\psi}\in H^{\infty}.$
In this case, $D^{\alpha,\beta}_{\psi} D^{\theta,\alpha}_{\phi}=D^{\theta,\beta}_{\phi\psi}$ if and only if
\begin{align*}
&(\bar{\alpha}-\bar{\lambda})\beta\bar{\psi},(\alpha-\lambda)\psi,(1-\lambda\bar{\alpha})\theta\phi,\\
&(1-\bar{\lambda}\alpha)\bar{\phi},(1-\lambda \bar{\alpha})\theta \phi\psi,(\bar{\alpha}-\bar{\lambda})\bar{\phi}\bar{\psi}\in H^{\infty}.
\end{align*}

If $|\lambda|=1,$
then $\frac{1}{\bar{\lambda}\alpha - 1}\in H^{p}(0<p<1)$ by \cite[Corollary 4.26]{fricain2016theory},
where $H^{p}$ is  the closed subspace spanned by analytic polynomials in the $L^{p}$ norm.
$(\bar{\alpha}-\bar{\lambda})\beta\bar{\psi},(\alpha-\lambda)\psi\in H^{\infty}$ implies that $\beta\bar{\psi},\psi\in H^{\infty}.$
This contradicts the assumption that both sides of \eqref{ADT} are non-zero.

If  $|\lambda|<1,$ then $\frac{1}{1-\bar{\lambda}\alpha}\in H^{\infty}.$ Therefore,
$(1-\bar{\lambda}\alpha)\bar{\phi}\in H^{\infty}$ is equivalent to $\bar{\phi}\in H^{\infty}.$ Thus
$(\bar{\alpha}-\bar{\lambda})\beta\bar{\psi}\in H^{\infty}$
implies $(\bar{\alpha}-\bar{\lambda})\bar{\phi}\bar{\psi}\in H^{\infty}.$
Thus we have obtained condition (3) in the conclusion.

If $|\lambda|>1,$ then $\frac{1}{\alpha-\lambda}\in H^{\infty}.$ Therefore, $(\alpha -\lambda)\psi\in H^{\infty}$ is equivalent to
$\psi\in H^{\infty}.$
$(1-\lambda\bar{\alpha})\theta\phi\in H^{\infty}$  implies
$(1-\lambda\bar{\alpha})\theta\phi\psi\in H^{\infty}.$
Here, $1/\bar{\lambda}$ corresponds to $\lambda$ in (4) of conclusion.
\end{enumerate}

\end{proof}

The problem of when the product of two dual truncated Toeplitz operators equals a dual truncated Toeplitz operator
 has been studied in \cite{Sang2018} and \cite{gu2021}.
Now, we provide a new proof of this problem.

If $\theta=\alpha=\beta$ in Theorem \ref{ADTP}, we can derive the following corollary.
\begin{corollary}
Let $\phi$ and $\psi$ be in $L^{\infty}.$
Then $D^{\theta}_{\phi}D^{\theta}_{\psi}=D^{\theta}_{\phi\psi}$ if and only if
one of the following statements holds.
\begin{enumerate}
\item either $\phi$ or $\psi$ is constant;
\item $\bar{\phi},\bar{\psi},(\theta-\lambda)\phi,(\theta-\lambda)\psi,
(\theta-\lambda)\phi\psi\in H^{\infty},$ where $\lambda$ is a constant with $|\lambda|< 1;$
\item $\phi,\psi,(\theta-\lambda)\bar{\psi},(\theta-\lambda)\bar{\phi},
(\theta-\lambda)\bar{\phi}\bar{\psi}\in H^{\infty},$ where $\lambda$ is a constant with $|\lambda|< 1.$
\end{enumerate}
\end{corollary}
\begin{proof}
If $\theta=\alpha,$ by Theorem \ref{ADTP}, then 
\begin{enumerate}[label=(\alph*)]
	\item $\bar{\psi},\psi\in H^{\infty};$
	\item  $\phi,\bar{\phi}\in H^{\infty};$
	\item $\bar{\phi},(1-\bar{\lambda}\theta)\bar{\psi},(\theta-\lambda)\phi,(\theta-\lambda)\psi,
	(\theta-\lambda)\phi\psi,
	\in H^{\infty},$ where $\lambda$ is a constant with $|\lambda|< 1;$
	\item $\psi,(1-\bar{\lambda}\theta)\phi,(\theta-\lambda)\bar{\psi},(\theta-\lambda)\bar{\phi},
	(\theta-\lambda)\bar{\phi}\bar{\psi}\in H^{\infty},$
	 where $\lambda$ is a constant with $|\lambda|<1.$
	\end{enumerate}	
 (a) and  (b) respectively imply that $\phi$ is constant and $\psi$ is constant. 
In (c), $(1-\bar{\lambda}\theta)\bar{\psi}\in H^{\infty}$ 
is equivalent to $\bar{\psi}\in H^{\infty}.$
Similarly, in (d), $(1-\bar{\lambda}\theta)\phi\in H^{\infty}$ 
is equivalent to $\phi\in H^{\infty}.$
\end{proof}
\begin{theorem}
Let $\phi,\psi$  be non-constant functions in $L^{\infty}.$
Then $D^{\theta}_{\phi}D^{\theta}_{\psi}=D^{\theta}_{\psi}D^{\theta}_{\phi}$ if and only if
\begin{enumerate}
\item $\phi, \psi, \bar{\phi}(\theta-\lambda),\bar{\psi}(\theta-\lambda)\in H^{\infty},$ where $\lambda$ is a constant with $|\lambda|< 1;$
\item $\bar{\phi}, \bar{\psi}, \phi(\theta-\lambda),\psi(\theta-\lambda)\in H^{\infty},$ where $\lambda$ is a constant with $|\lambda|< 1;$
\item a nontrivial linear combination of $\phi$ and $\psi$ is constant.
\end{enumerate}
\end{theorem}
\begin{proof}
Assume that $D^{\theta}_{\phi}D^{\theta}_{\psi}=D^{\theta}_{\psi}D^{\theta}_{\phi},$ 
by \eqref{DTM}, this is equivalent to 
$R_{\begin{bmatrix}
	\begin{smallmatrix}
		\phi & \bar{\theta}\phi\\
	   \theta \phi & \phi
	\end{smallmatrix}
	\end{bmatrix}}$ and $
	R_{\begin{bmatrix}
	\begin{smallmatrix}
		\psi & \bar{\theta}\psi\\
	   \theta \psi & \psi
	\end{smallmatrix}
	\end{bmatrix}}$ commuting.
 Theorem \ref{main2} implies that
\begin{align}\label{W}
	&\brackbinom{V(\bar{\phi})_{-}}{(\theta \phi)_{-}}\otimes \brackbinom{V\psi_{-}}{(\theta\bar{\psi})_{-}}
	-\brackbinom{V(\theta\bar{\phi})_{-}}{\phi_{-}}\otimes \brackbinom{V(\theta \psi)_{-}}{(\bar{\psi})_{-}}\nonumber \\
	=&
	\brackbinom{V(\bar{\psi})_{-}}{(\theta \psi)_{-}}\otimes \brackbinom{V\phi_{-}}{(\theta\bar{\phi})_{-}}
	-\brackbinom{V(\theta\bar{\psi})_{-}}{\psi_{-}}\otimes \brackbinom{V(\theta \phi)_{-}}{(\bar{\phi})_{-}}.
\end{align}
Now use Corollary \ref{odd} to see that
$\brackbinom{V(\bar{\phi})_{-}}{(\theta \phi)_{-}}\otimes \brackbinom{V\psi_{-}}{(\theta\bar{\psi})_{-}}
-\brackbinom{V(\theta\bar{\phi})_{-}}{\phi_{-}}\otimes \brackbinom{V(\theta \psi)_{-}}{(\bar{\psi})_{-}}$ 
necessarily either zero or  rank two.

If 
$\brackbinom{V(\bar{\phi})_{-}}{(\theta \phi)_{-}}\otimes \brackbinom{V\psi_{-}}{(\theta\bar{\psi})_{-}}
=\brackbinom{V(\theta\bar{\phi})_{-}}{\phi_{-}}\otimes \brackbinom{V(\theta \psi)_{-}}{(\bar{\psi})_{-}}.$ 
From Theorem \ref{ADTP}’s discussion of \eqref{ADT}, we derive (1) and (2) in conclusion.
Using \eqref{1.1.},\eqref{2.2.},\eqref{2.1.},\eqref{1.2.} and \eqref{casec}, we conclude that either (1) or (2) in conclusion ensures
$D^{\theta}_{\phi}D^{\theta}_{\psi}=D^{\theta}_{\psi}D^{\theta}_{\phi}$ holds. 

If 
$\brackbinom{V(\bar{\phi})_{-}}{(\theta \phi)_{-}}\otimes \brackbinom{V\psi_{-}}{(\theta\bar{\psi})_{-}}
-\brackbinom{V(\theta\bar{\phi})_{-}}{\phi_{-}}\otimes \brackbinom{V(\theta \psi)_{-}}{(\bar{\psi})_{-}}$ 
is a rank two operator,
then there exist exist constants $a,b,c,d$ such that
$\bar{\phi}-\bar{a}\bar{\psi}-\bar{b}\theta\bar{\psi},
\bar{\phi}+\bar{b}\theta\bar{\psi}-\bar{d}\bar{\psi},
\phi-c\theta\psi-d\psi,
\phi-a\psi+c\theta\psi\in H^{\infty}$
by the 1st, 4th, 5th, and 8th formulas in Theorem \ref{R2}.
Therefore,
$2\phi-(a+d)\psi,2\bar{\phi}-(\bar{a}+\bar{d})\bar{\psi}\in H^{\infty}.$
It follows that $\phi-\frac{a+d}{2} \psi$ is a constant. In this case, direct computation shows that 
$D^{\theta}_{\phi}D^{\theta}_{\psi}=D^{\theta}_{\psi}D^{\theta}_{\phi}.$ 
\end{proof}
\subsection{Singular integral operators}
The problem of products of singular integral operators has long been of interest,
and it has been studied in \cite{Gu2016,camara2023paired}.
For bounded functions $f,g$, we have
\begin{align*}
S_{f,g}=fP_{+}+gP_{-}=P_{+}fP_{+}+P_{-}fP_{+}+P_{+}gP_{-}+P_{-}gP_{-}=
R_{\begin{bmatrix}
	\begin{smallmatrix}
	   f& g\\
	   f & g
	\end{smallmatrix}
	\end{bmatrix}}.
\end{align*}
For the product of two singular integral operators to be a singular integral operator, 
applying Theorem \ref{main1} yields $\brackbinom{V(\bar{f_{1}}_{-}-\bar{g}_{1-})}{f_{1-}-g_{1-}}\otimes \brackbinom{Vf_{2-}}{(\bar{g}_{2})_{-}}=0,$ 
which closely matches the conclusion of the following Theorem, only the proof of $f_1=g_1$ remains required.
\begin{theorem}\label{S1}
Let $f_1,f_2,g_1$ and $g_2$ be functions in $L^{\infty}.$ Then $S_{f_1,g_1}S_{f_2,g_2}=S_{f_1 f_2,g_1 g_2}$
if and only if one of the following statements holds
\begin{enumerate}
	\item $f_1=g_1;$
	\item $f_2,\bar{g}_2\in H^{\infty}.$
\end{enumerate}
\end{theorem}
\begin{proof}
If $S_{f_1,g_1}S_{f_2,g_2}=S_{f_1 f_2,g_1 g_2},$
by Theorem \ref{main1}, then
$\brackbinom{V(\bar{f_{1}}_{-}-\bar{g}_{1-})}{f_{1-}-g_{1-}}\otimes \brackbinom{Vf_{2-}}{(\bar{g}_{2})_{-}}=0.$

If $\brackbinom{Vf_{2-}}{(\bar{g}_{2})_{-}}\neq 0,$ then $f_1 -g_1$ is constant.
Using \eqref{casec} with $\lambda=1$, we have 
$S_{f_1,g_1}S_{f_2,g_2}=
\begin{bmatrix}
\begin{smallmatrix}
	T_{f_1 f_2}& H^{*}_{\bar{f}_1\bar{g}_2}\\
	H_{g_1 f_2} &  \tilde{T}_{g_1 g_2}
\end{smallmatrix}
 \end{bmatrix}  
=
\begin{bmatrix}
\begin{smallmatrix}
	T_{f_1 f_2}& H^{*}_{\bar{g}_1\bar{g}_2}\\
	H_{f_1 f_2} &  \tilde{T}_{g_1 g_2}
\end{smallmatrix}
 \end{bmatrix}  
=S_{f_1 f_2,g_1 g_2}.$
Thus 
$H_{g_1 f_2}=H_{f_1 f_2}$ and $H^{*}_{\bar{f}_1 g_{2}}=H^{*}_{\bar{g}_1 \bar{g}_2}.$
This implies that $(f_{1}-g_{1})f_2\in H^{2}$ and $(\bar{f}_{1}-\bar{g}_{1})\bar{g}_2\in H^{2}.$
If $f_{1}\neq g_{1},$ then $f_2,\bar{g}_2\in H^{\infty},$ which contradicts our assumption.

Conversely, if $f_1=g_1,$ then $S_{f_1,g_1}S_{f_2,g_2}=f_1(f_2 P_{+}+g_2 P_{-})=S_{f_1 f_2,g_1 g_2}.$
If $f_2,\bar{g}_2\in H^{\infty},$ then $S_{f_1,g_1}S_{f_2,g_2}=f_1P_{+}f_2P_{+}
+f_1P_{+}g_2P_{-}+g_1P_{-}f_2P_{+}+g_1P_{-}g_2P_{-}=f_1f_2P_{+}+g_1g_2P_{-}=S_{f_1 f_2,g_1 g_2}.$
\end{proof}
\begin{theorem}
	Let $f_1,f_2,g_1$ and $g_2$ be functions in $L^{\infty}.$ 
	Then $S_{f_1,g_1}S_{f_2,g_2}=S_{f_2,g_2}S_{f_1,g_1}$
	if and only if one of the following statements holds.
	\begin{enumerate}
		\item $f_1=g_1$ and $f_2=g_2;$ \label{5.6.1}
		\item $f_1,\bar{g}_1,f_2,\bar{g}_2\in H^{\infty};$ \label{5.6.2}
		\item There exist constatns $c_1,c_2,c_3$ with $|c_1|+|c_2|\neq 0$ such that 
		$c_1\brackbinom{f_1}{g_1}-c_2\brackbinom{f_2}{g_2}=\brackbinom{c_3}{c_3}.$
		In this case, $c_1 S_{f_1.g_1}=c_2S_{f_2,g_2}+c_3I_{L^2}.$ \label{5.6.3}
	\end{enumerate}
	\end{theorem}
\begin{proof}
Suppose that  $S_{f_1,g_1}S_{f_2,g_2}=S_{f_2,g_2}S_{f_1,g_1},$ by Theorem \ref{main2}, we have
\begin{align}
\brackbinom{V(\bar{f_{1}}_{-}-\bar{g}_{1-})}{f_{1-}-g_{1-}}\otimes \brackbinom{Vf_{2-}}{(\bar{g}_{2})_{-}}
=&\brackbinom{V(\bar{f_{2}}_{-}-\bar{g}_{2-})}{f_{2-}-g_{2-}}\otimes \brackbinom{Vf_{1-}}{(\bar{g}_{1})_{-}};\label{SSSS}\\
H_{\bar{g}_2}T_{\bar{f}_1}+\tilde{T}_{\bar{g}_2}H_{\bar{g}_1}-&H_{\bar{g}_1}T_{\bar{f}_2}-\tilde{T}_{\bar{g}_1}H_{\bar{g}_2}=0;\label{H1T2}\\
H_{f_1}T_{f_2}+\tilde{T}_{g_1}H_{f_2}-&H_{f_2}T_{f_1}-\tilde{T}_{g_2}H_{f_1}=0.\label{H2T1}
\end{align}

First, we discuss the case where both sides of \eqref{SSSS} are zero. 
If $\brackbinom{Vf_{2-}}{(\bar{g}_{2})_{-}}=0$, then $f_2,\bar{g}_{2}\in H^{\infty}.$ By Theorem \ref{S1}, 
we have  $S_{f_1,g_1}S_{f_2,g_2}=S_{f_1 f_2,g_1 g_2}.$
Thus $S_{f_2,g_2}S_{f_1,g_1}=S_{f_1 f_2,g_1 g_2},$ using Theorem \ref{S1} again, we have $f_2=g_2$ or $f_1,\bar{g}_1\in H^{\infty}.$
$f_2=g_2$ implies $f_2=g_2$ is a constant and $S_{f_2,g_2}=f_1 I_{L^2}.$ 
Thus, we have obtained \eqref{5.6.3} in the conclusion.
$f_1,\bar{g}_1\in H^{\infty}$ corresponds to \eqref{5.6.2} in the conclusion.
The case of $\brackbinom{Vf_{1-}}{(\bar{g}_{1})_{-}}=0$ is similarly.

If $\brackbinom{V(\bar{f_{1}}_{-}-\bar{g}_{1-})}{f_{1-}-g_{1-}}=\brackbinom{V(\bar{f_{2}}_{-}-\bar{g}_{2-})}{f_{2-}-g_{2-}}=0,$
then $f_1 -g_1\triangleq c_2$ and $f_2-g_2\triangleq c_1$ are constants, and 
$S_{f_1,g_1}S_{f_2,g_2}=
\begin{bmatrix}
\begin{smallmatrix}
	T_{f_1 f_2}& H^{*}_{\bar{f}_1\bar{g}_2}\\
	H_{g_1 f_2} &  \tilde{T}_{g_1 g_2}
\end{smallmatrix}
 \end{bmatrix}  
=
\begin{bmatrix}
\begin{smallmatrix}
	T_{f_1 f_2}& H^{*}_{\bar{f}_2\bar{g}_1}\\
	H_{g_2 f_1} &  \tilde{T}_{g_1 g_2}
\end{smallmatrix}
 \end{bmatrix}  
=S_{f_2,g_2}S_{f_1,g_1}$ by  \eqref{casec} with $\lambda=1.$
Thus, $H_{g_1 f_2}=H_{g_2 f_1}$ and $H^{*}_{\bar{f}_1\bar{g}_2}=H^{*}_{\bar{f}_2\bar{g}_1}.$
This means that $f_2g_1-f_1 g_2\triangleq c_3$ is constant.
If $c_1=c_2=0,$ then we have obtained \eqref{5.6.1} in the conclusion. Otherwise,
$c_1f_1 -c_2 f_2=c_1g_1 -c_2 g_2=c_3.$ This case corresponds to condition (3) in the conclusion.

Next, we discuss the case where both sides of \eqref{SSSS} are rank-one operators.
There exists a constant $\lambda$ such that 
$f_1-\lambda f_2,\bar{g}_1-\bar{\lambda}\bar{g}_2,\bar{f}_1-\bar{\lambda}\bar{f}_2-(\bar{g}_1-\bar{\lambda}\bar{g}_2)$
and $f_1-\lambda f_2-(g_1-\lambda g_2)$ are in $H^{\infty}.$ Therefore,
\begin{align}\label{f1f2}
f_1-\lambda f_2=\mu_1\ \text{and} \ g_1-\lambda g_2=\mu_2
\end{align} 
are constants.
Substituting \eqref{f1f2} into \eqref{H1T2} and \eqref{H2T1} yields $(\mu_1-\mu_2)H_{\bar{g}_2}=0$ and $(\mu_1-\mu_2)H_{f_2}=0.$
 If $\mu_1\neq \mu_2,$ then $\bar{g}_2,f_2\in H^{\infty}.$
This contradicts the assumption that both sides of \eqref{SSSS} are rank one.
This case corresponds to condition (3) in the conclusion.

Conversely, conditions \eqref{5.6.3}   directly imply that $S_{f_1,g_1}$ and $S_{f_2,g_2}$ commute.
By Theorem \ref{S1},
condition \eqref{5.6.1} or \eqref{5.6.2} implies that $S_{f_1,g_1}S_{f_2,g_2}=S_{f_2,g_2}S_{f_1,g_1}=S_{f_1 f_2,g_1 g_2}.$ 
\end{proof}
\begin{theorem}\label{normals}
	Let $f$ and $g$ be functions in $L^{\infty}.$ 
	Then $S_{f,g}$ is normal 
	if and only if one of the following statements holds.
	\begin{enumerate}
		\item $f$ and $g$ are constants.
		\item There exists a unimodular constant $\lambda$  such that $f-\lambda g$ and $\lambda \bar{f}g-f\bar{g}$ are constants.
	\end{enumerate}
	\end{theorem}
\begin{proof}
Suppose that $S_{f,g}$ is normal, and since 
$S_{f,g}=R_{\begin{bmatrix}
	\begin{smallmatrix}
	   f& g\\
	   f & g
	\end{smallmatrix}
	\end{bmatrix}},$
$S^{*}_{f,g}=R_{\begin{bmatrix}
	\begin{smallmatrix}
	   \bar{f}& \bar{f}\\
	   \bar{g} & \bar{g}
	\end{smallmatrix}
	\end{bmatrix}},$ by Theorem \ref{main2}, it follows that
\begin{align*}
\brackbinom{V(\bar{f})_{-}}{f_{-}}\otimes \brackbinom{V(\bar{f})_{-}}{f_{-}}
-&\brackbinom{V(\bar{g})_{-}}{g_{-}}\otimes \brackbinom{V(\bar{g})_{-}}{g_{-}}  \\
=\brackbinom{Vf_{-}}{(\bar{g})_{-}}\otimes \brackbinom{Vf_{-}}{(\bar{g})_{-}}
-&\brackbinom{Vf_{-}}{(\bar{g})_{-}}\otimes \brackbinom{Vf_{-}}{(\bar{g})_{-}}
=0.
\end{align*}

There are two cases.

If $\brackbinom{V(\bar{f})_{-}}{f_{-}}=\brackbinom{V(\bar{g})_{-}}{g_{-}}=0,$ then $f$ and $g$ are constants.
Conversely, by direct computation, if $f$ and $g$ are constants, then $S_{f,g}$ is normal.

If there exists a unimodular constant  $\lambda$ such that $\brackbinom{V(\bar{f})_{-}}{f_{-}}=\lambda \brackbinom{V(\bar{g})_{-}}{g_{-}},$
then $f-\lambda g$ is constant. This case corresponds to  (13) of Theorem \ref{mainc1}.
Consequently,  $\mu=1$ and $\lambda \bar{f}g-f\bar{g}$ is constant. Conversely, by (13) of Theorem \ref{mainc1},
if (2) in the conclusion holds, then $S_{f,g}$ is normal.
\end{proof}
\begin{remark}
Both \cite[Theroem 3.2]{Nakazi2014} and \cite[Corollary 5.7]{Gu2016}
characterize normal singular integral operators and present the following condition:
\begin{quote}
\itshape
there exist constants $\lambda$ and $\delta$ such that
\begin{align}\label{nakazi}
|\lambda|=1, \quad f=\lambda g+\delta, \quad(\lambda-1)|g|^2+\delta \bar{g}-\bar{\delta} \lambda g \in \mathbb{C}.	
\end{align}
\end{quote}
Since
\begin{align*}
&(\lambda -1)|g|^{2}+\delta\bar{g}-\bar{\delta}\lambda g\\
=&\lambda |g|^{2}-|g|^{2}+(f-\lambda g)\bar{g}-(\bar{f}-\bar{\lambda}\bar{g})\lambda g\\
=&f\bar{g}-\lambda \bar{f}g,
\end{align*}
\eqref{nakazi} is  equivalent to the condition (2) of Theorem \ref{normals}.

\end{remark}
\begin{theorem}\label{quasinormal}
	Let $f$ and $g$ be functions in $L^{\infty}.$ 
	Then $S_{f,g}$ is quasinormal 
	if and only if one of the following statements holds.
\begin{enumerate}
	\item $|f|$ and $|g|$ are constants, $f,\bar{g}\in H^{\infty};$
	\item $|f|=|g|$ is constants, $f\bar{g}\in H^{\infty};$
	\item $f-g$ and $f|g|^2 -g|f|^2$ are constants, and $|f|^{2}_{-}=|g|^{2}_{-}=(f\bar{g})_{-};$
	\item $\brackbinom{Vf_{-}}{(\bar{g})_{-}}\neq 0,\brackbinom{V(|f|^{2}-f\bar{g})_{-}}{(f\bar{g}-|g|^{2})_{-}}\neq 0$ and one of the following statements holds.
	\begin{enumerate}
			\item There are  constants $\mu$ and $\alpha(\neq 0)$ such that $f-\mu g$ is constant,
	 $|f|^{2}-f\bar{g}-\bar{\alpha}\bar{g},|g|^{2}-f\bar{g}+\alpha g,
	 f(\mu \bar{f}-\bar{g}-\alpha),
	 \bar{g}(\bar{\mu}f-g-\bar{\alpha}),
	 fg(\mu \bar{f}-\bar{g}-\alpha),
	 \bar{g}^{2}(\bar{\mu}f-g-\bar{\alpha})\in H^{\infty};$
	\item There are constants $\mu$ and $\alpha(\neq 0)$ such that $g-\mu f$ is constant,
 $|f|^{2}-f\bar{g}-\bar{\alpha}\bar{f},|g|^{2}-f\bar{g}+\alpha f,f(\bar{f}-\mu \bar{g}-\alpha),
 \bar{g}(f-\bar{\mu}g-\bar{\alpha}),
 \bar{f}\bar{g}(f -\bar{\mu}g-\bar{\alpha}),
f^2(\bar{f}-\mu \bar{g}-\alpha)\in H^{\infty};$
   \item There are  constants $\mu$ and $\alpha(\neq 0)$ such that $f(\bar{f}-\bar{\mu}\bar{g}),\bar{g}(f-\mu g)$,
 $|f|^{2}-f\bar{g}-\bar{\alpha}(\mu\bar{f}-\bar{g}),
 f\bar{g}-|g|^{2}-\alpha (\bar{\mu}f-g),
 f\bar{g}-\alpha f,
 |g|^2 -\bar{\alpha}\bar{g},
\bar{g}(|f|^2-|g|^2-\bar{\alpha}\mu\bar{f}+\bar{\alpha}\bar{g}),
f(|f|^{2}-|g|^2-\alpha\bar{\mu} f+\alpha  g)\in H^{\infty};$
  \item There are  constants $\mu$ and $\alpha(\neq 0)$ such that $f(\bar{g}-\bar{\mu}\bar{f}),\bar{g}(g-\mu f),$
$|f|^{2}-f\bar{g}-\bar{\alpha}(\bar{f}-\mu\bar{g}),
f\bar{g}-|g|^{2}-\alpha (f-\bar{\mu}g),
f\bar{g}-\bar{\alpha}\bar{g},
|f|^2 -\alpha f,
\bar{g}(|f|^2-|g|^2+\bar{\alpha}\mu\bar{g}-\bar{\alpha}\bar{f}),
f(|f|^{2}-|g|^2-\alpha \bar{\mu} g+\alpha  f)\in H^{\infty}.$
	\end{enumerate}
\end{enumerate}
\end{theorem}
\begin{proof}
Using \eqref{StarS}, we have
$S_{f,g}$ is quasinormal is equivalent to 
$R_{\begin{bmatrix}
	\begin{smallmatrix}
	   |f|^2 & \bar{f}g\\
	   f\bar{g} & |g|^{2}
	\end{smallmatrix}
	\end{bmatrix}}$
and 
$R_{\begin{bmatrix}
	\begin{smallmatrix}
	   f & g\\
	   f & g
	\end{smallmatrix}
	\end{bmatrix}}$  commute.
By Theroem \ref{main2}, we have $S_{f,g}$ is quasinormal
if and only if 
\begin{align}
H_{\bar{g}}T_{|f|^{2}}+\tilde{T}_{\bar{g}}H_{f\bar{g}}=&H_{f\bar{g}}T_{\bar{f}}+\tilde{T}_{|g|^2}H_{\bar{g}};\label{q1}\\
H_{f\bar{g}}T_{f}+\tilde{T}_{|g|^{2}}H_{f}=&H_{f}T_{|f|^{2}}+\tilde{T}_{g}H_{f\bar{g}};\label{q2}\\
\brackbinom{V(|f|^{2}-f\bar{g})_{-}}{(f\bar{g}-|g|^{2})_{-}}\otimes \brackbinom{Vf_{-}}{(\bar{g})_{-}}
=&\brackbinom{V(\bar{f})_{-}}{f_{-}}\otimes \brackbinom{V(|f|^{2})_{-}}{(f\bar{g})_{-}} \nonumber \\
-&\brackbinom{V(\bar{g})_{-}}{g_{-}}\otimes \brackbinom{V(f\bar{g})_{-}}{(|g|^{2})_{-}}.\label{q3}
\end{align}
We examine equation \eqref{q3} in the following cases:
\begin{enumerate}[label=\textbf{Case \Roman*:}, leftmargin=2cm]
\item If $\brackbinom{Vf_{-}}{(\bar{g})_{-}}=0,$ then $f,\bar{g}\in H^{\infty}$ and $V(\bar{f})_{-}\otimes V(|f|^{2})_{-}=0,$
$g_{-}\otimes (|g|^{2})_{-}=0$ by Lemma \ref{key}. Although there are four cases, all four cases result in $|f|$ and $|g|$ being constant.
Hence, we obtain conclusion item (1).
\item If $\brackbinom{V(|f|^{2}-f\bar{g})_{-}}{(f\bar{g}-|g|^{2})_{-}}=0,$ then $|f|^{2}_{-}=|g|^{2}_{-}=(f\bar{g})_{-}$
and $\brackbinom{V(\bar{f}-\bar{g})_{-}}{(f-g)_{-}}\otimes \brackbinom{V(|f|^{2})_{-}}{(|f|^{2})_{-}}=0.$
         \begin{enumerate}[label=\textbf{Case \roman*:}, leftmargin=2cm]
			         \item If $\brackbinom{V(|f|^{2})_{-}}{(|f|^{2})_{-}}=0,$ then $|f|$ is constant.
         Moreover, $|g|$ is constant, and $f\bar{g}\in H^{\infty}.$ \eqref{q1} implies $H_{\bar{g}(|f|^2-|g|^2)}=0.$ If $\bar{g}\notin H^{\infty},$ 
		 then $|f|=|g|.$ Similarly, \eqref{q2} implies $H_{f(|f|^2-|g|^2)}=0.$ If $f\notin H^{\infty},$ then then $|f|=|g|.$ 
		 This leads us to (2) in the conclusion.
			\item If $\brackbinom{V(\bar{f}-\bar{g})_{-}}{(f-g)_{-}}=0,$ then $f-g$ is constant. Calculate both sides of Equation \eqref{q1}, we have 
\begin{align*}
H_{\bar{g}}T_{|f|^{2}}+\tilde{T}_{\bar{g}}H_{f\bar{g}}
&=H_{\bar{g}}T_{|f|^{2}}+\tilde{T}_{\bar{g}}H_{|f|^{2}}
=H_{\bar{g}|f|^2},\\
H_{f\bar{g}}T_{\bar{f}}+\tilde{T}_{|g|^2}H_{\bar{g}}
&=H_{|g|^2}T_{\bar{f}}+\tilde{T}_{|g|^2}H_{\bar{f}}
=H_{\bar{f}|g|^2}.
\end{align*}
Therefore, $\bar{g}|f|^2-\bar{f}|g|^2\in H^{\infty}.$ Similarly, calculate both sides of \eqref{q2}, we have 
$g|f|^2-f|g|^2\in H^{\infty}.$ This means that $g|f|^2-f|g|^2$ is constant. Thus we have reached (3) in the conclusion.
\end{enumerate}
\item If $\brackbinom{V(|f|^{2}-f\bar{g})_{-}}{(f\bar{g}-|g|^{2})_{-}}\otimes \brackbinom{Vf_{-}}{(\bar{g})_{-}}$ is a rank one operator, then 
this case corresponds to cases (3), (7), (11), and (15) in Theorem \ref{mainc2}, where  $\lambda=1.$ Therefore, we arrive at conclusion 
(4)(a),(4)(b),(4)(c), and (4)(d).
\end{enumerate}

Conversely, based on the above analysis, each condition in the conclusion ensures the validity of 
\eqref{q1},\eqref{q2}, and \eqref{q3}, thus establishing the necessity of the conclusion.

\end{proof}
\begin{remark}
Theorem 2.6 in \cite{ko2021quasinormality} proves that 
$S_{f,g}$ is quasinormal if and only if the following identities hold:
\begin{align*}
T_{|f|^{2}}T_f+H^{*}_{f\bar{g}}H_f&=T_{f}T_{|f|^2}+H^{*}_{\bar{g}}H_{f\bar{g}};\\
T_{|f|^{2}}H^{*}_{\bar{g}}+H^{*}_{f\bar{g}}\tilde{T}_{g}&=T_{f}H^{*}_{f\bar{g}}+H^{*}_{\bar{g}}\tilde{T}_{|g|^2};\\
H_{f\bar{g}}T_{f}+\tilde{T}_{|g|^{2}}H_{f}&=H_{f}T_{|f|^{2}}+\tilde{T}_{g}H_{f\bar{g}};\\
H_{f\bar{g}}H^{*}_{\bar{g}}+\tilde{T}_{|g|^{2}}\tilde{T}_{g}&=H_{f}H^{*}_{f\bar{g}}+\tilde{T}_{g}\tilde{T}_{|g|^2}.
\end{align*}
This means that 
$R_{\begin{bmatrix}
	\begin{smallmatrix}
	   |f|^2 & \bar{f}g\\
	   f\bar{g} & |g|^{2}
	\end{smallmatrix}
	\end{bmatrix}}$
and 
$R_{\begin{bmatrix}
	\begin{smallmatrix}
	   f & g\\
	   f & g
	\end{smallmatrix}
	\end{bmatrix}}$  commute.
\end{remark}
\begin{remark}
The conditions (4)(a) to (4)(d) in Theorem \ref{quasinormal}  appear  complex, 
but this does not imply that the functions $f$ and $g$ satisfying them are trivial. 
Below we will demonstrate how  $f$ and $g$ satisfying that $S_{f,g}$ is a normal operator fulfill condition (4)(a) in Theorem \ref{quasinormal}.

If $f$ and $g$ satisfying condition (2) in Theorem \ref{normals}.
Let $f=\mu g+\delta,$ where $\mu$ and $\delta$ are constants with $|\mu|=1.$
$\alpha=\bar{\delta}\mu.$ Thus, $\bar{\mu}f-g-\bar{\alpha}=0.$ The last four expressions in (4)(a) equal zero;
 only the first two need to be calculated.
\begin{align*}
|f|^{2}-f\bar{g}-\bar{\alpha}\bar{g}&=(1-\mu)|g|^2-\delta \bar{g}+\bar{\delta}\mu g+|\delta|^{2}\in H^{\infty}, (see \eqref{nakazi})\\
|g|^{2}-f\bar{g}+\alpha g&=(1-\mu)|g|^2-\delta \bar{g}+\bar{\delta}\mu g\in H^{\infty},(see \eqref{nakazi}).
\end{align*}
\end{remark}
\section{appendix A}
In the earlier version of the paper, we proved Proposition \ref{20} and subsequently used it to prove Theorem \ref{main1}. 
Unfortunately, we were unable to establish Theorem \ref{main2} analogously. Later, we discovered Lemma \ref{key}, 
which uniformly handles both Problem \ref{Q11} and Problem \ref{Q22}, while also yielding conclusions regarding 
the matrices we initially employed.

\begin{proposition}\label{20}
	Let $A=(a_{ij})$ and $B=(b_{ij})$ be two $2\times 2$ complex matrices.
	 Then $AB=0$ if and only if one of the following statements holds:
	\begin{enumerate}
	\item  either $A=0$ or $B=0;$
	\item $A=\begin{bmatrix}a_{11}& 0\\
			a_{21} & 0 \end{bmatrix}\neq 0,\
	B=\begin{bmatrix}0 & 0\\
			b_{21} & b_{22} \end{bmatrix}\neq 0;$
	\item there exists a constant $\mu$
	such that
	\begin{align*}
	A=\begin{bmatrix}\mu a_{12}& a_{12}\\
			\mu a_{22} & a_{22} \end{bmatrix}\neq 0,\
	B=\begin{bmatrix}b_{11} & b_{12}\\
			-\mu b_{11} & -\mu b_{12} \end{bmatrix}\neq 0.
	\end{align*}
	\end{enumerate}
	\end{proposition}
\begin{proof}
For two complex numbers $z,w,$ we have $z\otimes w=z\bar{w}.$
If $AB=0,$ by Lemma \ref{key}, then 
\begin{align*}
\brackbinom{a_{11}}{a_{21}}\otimes \brackbinom{\bar{b}_{11}}{\bar{b}_{12}}
+\brackbinom{a_{21}}{a_{22}}\otimes \brackbinom{\bar{b}_{21}}{\bar{b}_{22}}=0.
\end{align*} 
The conclusion can be reached through case analysis of the above formula.
\end{proof}
\section{appendix B}\label{AII}

This appendix presents the remaining 15 uncalculated cases for Theorem \ref{mainc2}.
Similar to case (L1)+(R1),
when computing  \eqref{1.2} and \eqref{2.1}, first use the first two relations concerning $\lambda$ and $\mu$ along with \eqref{eq:2m}, 
combine like terms, then apply the relation involving $\alpha$ to obtain the final result.

In case \textbf{(L2)+(R1)}.
Substituting $\brackbinom{V(\bar{u}_{1})_{-}}{v_{1-}}=\lambda \brackbinom{V(\bar{f_{1}})_{-}}{g_{1-}}$
and $\brackbinom{V(\bar{f_{2}})_{-}}{g_{2-}}=\mu \brackbinom{V(\bar{u}_{2})_{-}}{v_{2-}}$ into \eqref{WW}
shows that there exists a non-zero constant $\alpha$ such that
\begin{align*}
\brackbinom{V(\bar{f_{1}})_{-}}{g_{1-}}=\alpha \brackbinom{V(\bar{u}_{2})_{-}}{v_{2-}}, \ \
\bar{\mu}\brackbinom{Vf_{1-}}{(\bar{u}_{1})_{-}}-\brackbinom{Vg_{1-}}{(\bar{v}_{1})_{-}}
=\bar{\alpha}\bigg(\brackbinom{Vf_{2-}}{(\bar{u}_{2})_{-}}-\bar{\lambda}\brackbinom{Vg_{2-}}{(\bar{v}_{2})_{-}}\bigg).
\end{align*}
Substituting the above conditions into  \eqref{1.2} and \eqref{2.1}, and applying \eqref{eq:2m}, we obtain
\begin{align*}
&H_{\bar{u}_{2}}T_{\bar{f}_{1}}+\tilde{T}_{\bar{v}_{2}}H_{\bar{u}_{1}}-H_{\bar{u}_{1}}T_{\bar{f}_{2}}-\tilde{T}_{\bar{v}_{1}}H_{\bar{u}_{2}}\\
=&H_{\bar{u}_{2}}T_{\bar{f}_{1}}+\tilde{T}_{\bar{v}_{2}}H_{\bar{\lambda}\bar{f}_1}
-H_{\bar{u}_{1}}T_{\bar{f}_{2}}-\tilde{T}_{\bar{v}_{1}}H_{\bar{u}_{2}}\\
=&H_{\bar{u}_{2}}T_{\bar{f}_{1}}+H_{\bar{\lambda}\bar{v}_2\bar{f}_1}-H_{\bar{v}_2}T_{\bar{\lambda}\bar{f}_1}
-(H_{\bar{u}_1\bar{f}_2}-\tilde{T}_{\bar{u}_1}H_{\bar{f}_2})-\tilde{T}_{\bar{v}_1}H_{\bar{u}_2} \\
=& H_{\bar{\lambda}\bar{v}_2\bar{f}_1-\bar{u}_1\bar{f}_2}+H_{\bar{u}_2-\bar{\lambda}\bar{v}_2}T_{\bar{f}_1}
+\tilde{T}_{\bar{u}_1}H_{\bar{\mu} \bar{u}_2}-\tilde{T}_{\bar{v}_1}H_{\bar{u}_2}\\
=& H_{\bar{\lambda}\bar{v}_2\bar{f}_1-\bar{u}_1\bar{f}_2}+H_{\bar{u}_2-\bar{\lambda}\bar{v}_2}T_{\bar{f}_1}
+\tilde{T}_{\bar{\mu}\bar{u}_1-\bar{v}_1}H_{\bar{u}_2}\\
=& H_{\bar{\lambda}\bar{v}_2\bar{f}_1-\bar{u}_1\bar{f}_2}+H_{\bar{u}_2-\bar{\lambda}\bar{v}_2}T_{\bar{f}_1}
+H_{(\bar{\mu}\bar{u}_1-\bar{v}_1)\bar{u}_2}-H_{\bar{\mu}\bar{u}_1-\bar{v}_1}T_{\bar{u}_2}\\
=& H_{\bar{\lambda}\bar{v}_2\bar{f}_1-\bar{u}_1\bar{f}_2+(\bar{\mu}\bar{u}_1-\bar{v}_1)\bar{u}_2}
+H_{\bar{u}_2-\bar{\lambda}\bar{v}_2}T_{\bar{f}_1}-H_{\bar{\alpha}(\bar{u}_2-\bar{\lambda}\bar{v}_2)}T_{\bar{u}_2}\\
=& H_{\bar{\lambda}\bar{v}_2\bar{f}_1-\bar{u}_1\bar{f}_2+(\bar{\mu}\bar{u}_1-\bar{v}_1)\bar{u}_2}
+H_{\bar{u}_2-\bar{\lambda}\bar{v}_2}T_{\bar{f}_1-\bar{\alpha}\bar{u}_2}\\
=& H_{\bar{\lambda}\bar{v}_2\bar{f}_1-\bar{u}_1\bar{f}_2+(\bar{\mu}\bar{u}_1-\bar{v}_1)\bar{u}_2
 +(\bar{u}_2-\bar{\lambda}\bar{v}_2)(\bar{f}_1-\bar{\alpha}\bar{u}_2)}\\
=&H_{\bar{u}_2\bar{f}_1-\bar{u}_1\bar{f}_2-\bar{\alpha}\bar{u}^{2}_2
+\bar{\lambda}\bar{\alpha}\bar{v}_2\bar{u}_2+\bar{\mu}\bar{u}_1\bar{u}_2-\bar{v}_1\bar{u}_2},
\end{align*}
and
\begin{align*}
&H_{g_{1}}T_{f_{2}}+\tilde{T}_{v_{1}}H_{g_{2}}-H_{g_{2}}T_{f_{1}}-\tilde{T}_{v_{2}}H_{g_{1}}\\
=&H_{g_{1}}T_{f_{2}}+H_{v_1 g_2}-H_{v_1}T_{g_2}-H_{\mu v_{2}}T_{f_{1}}-H_{v_2 g_1}+H_{v_2}T_{g_1}\\
=&H_{v_{1}g_{2}-v_2 g_1}+H_{g_1}T_{f_2 -\lambda g_2}+H_{v_2}T_{g_1-\mu f_1}\\
=&H_{v_{1}g_{2}-v_2 g_1}+H_{g_1}T_{f_2 -\lambda g_2}+H_{v_2}T_{g_1-\mu f_1}\\
=&H_{v_{1}g_{2}-v_2 g_1}+H_{\alpha v_2}T_{f_2 -\lambda g_2}+H_{v_2}T_{g_1-\mu f_1}\\
=&H_{v_{1}g_{2}-v_2 g_1}+H_{v_2}T_{\alpha(f_2 -\lambda g_2)+g_1-\mu f_1}\\
=&H_{v_{1}g_{2}-v_2 g_1}+H_{v_2\alpha(f_2 -\lambda g_2)+v_2(g_1-\mu f_1)}\\
=&H_{v_{1}g_{2}-v_2 g_1+v_2\alpha(f_2 -\lambda g_2)+v_2(g_1-\mu f_1)}\\
=&H_{g_{2}(v_1 -\lambda \alpha v_2)+v_2 (\alpha f_2-\mu f_1)}.
\end{align*}

In case \textbf{(L3)+(R1)}. Substituting $\brackbinom{Vf_{2-}}{(\bar{u}_{2})_{-}}=\lambda \brackbinom{Vg_{2-}}{(\bar{v}_{2})_{-}}$
and $\brackbinom{V(\bar{f_{2}})_{-}}{g_{2-}}=\mu \brackbinom{V(\bar{u}_{2})_{-}}{v_{2-}}$ into \eqref{WW}
shows that there exists a non-zero constant $\alpha$ such that\\
\begin{align*}
\bar{\lambda}\brackbinom{V(\bar{f_{1}})_{-}}{g_{1-}}-\brackbinom{V(\bar{u}_{1})_{-}}{v_{1-}}=\alpha \brackbinom{V(\bar{u}_{2})_{-}}{v_{2-}},\ \
\bar{\mu}\brackbinom{Vf_{1-}}{(\bar{u}_{1})_{-}}-\brackbinom{Vg_{1-}}{(\bar{v}_{1})_{-}}
=\bar{\alpha}\brackbinom{Vg_{2-}}{(\bar{v}_{2})_{-}}.
\end{align*}
Substituting the above conditions into  \eqref{1.2} and \eqref{2.1}, and applying \eqref{eq:2m}, we obtain
\begin{align*}
&H_{\bar{u}_{2}}T_{\bar{f}_{1}}+\tilde{T}_{\bar{v}_{2}}H_{\bar{u}_{1}}-H_{\bar{u}_{1}}T_{\bar{f}_{2}}-\tilde{T}_{\bar{v}_{1}}H_{\bar{u}_{2}}\\
=&H_{\lambda\bar{v}_2}T_{\bar{f}_{1}}+H_{\bar{v}_{2}\bar{u}_{1}}-H_{\bar{v}_{2}}T_{\bar{u}_{1}}
-H_{\bar{u}_{1}\bar{f}_{2}}+\tilde{T}_{\bar{u}_{1}}H_{\bar{f}_{2}}-\tilde{T}_{\bar{v}_{1}}H_{\bar{u}_{2}}\\
=&H_{\bar{v}_{2}\bar{u}_{1}-\bar{u}_{1}\bar{f}_{2}}+H_{\lambda\bar{v}_2}T_{\bar{f}_{1}}-H_{\bar{v}_{2}}T_{\bar{u}_{1}}
+\tilde{T}_{\bar{u}_{1}}H_{\bar{\mu}\bar{u}_{2}}-\tilde{T}_{\bar{v}_{1}}H_{\bar{u}_{2}}\\
=&H_{\bar{v}_{2}\bar{u}_{1}-\bar{u}_{1}\bar{f}_{2}}+H_{\bar{v}_2}T_{\lambda\bar{f}_{1}-\bar{u}_{1}}
+\tilde{T}_{\bar{\mu}\bar{u}_{1}-\bar{v}_{1}}H_{\bar{u}_{2}}\\
=&H_{\bar{v}_{2}\bar{u}_{1}-\bar{u}_{1}\bar{f}_{2}}+H_{\bar{v}_2(\lambda\bar{f}_{1}-\bar{u}_{1})}-\tilde{T}_{\bar{v}_2}H_{\lambda\bar{f}_{1}-\bar{u}_{1}}
+H_{\bar{u}_{2}(\bar{\mu}\bar{u}_{1}-\bar{v}_{1})}-H_{\bar{\mu}\bar{u}_{1}-\bar{v}_{1}}T_{\bar{u}_{2}}\\
=&H_{\bar{v}_{2}\bar{u}_{1}-\bar{u}_{1}\bar{f}_{2}+\bar{v}_2(\lambda\bar{f}_{1}-\bar{u}_{1})+\bar{u}_{2}(\bar{\mu}\bar{u}_{1}-\bar{v}_{1})}
-\tilde{T}_{\bar{v}_2}H_{\bar{\alpha}\bar{u}_2}
-H_{\bar{\alpha}\bar{v}_2}T_{\bar{u}_{2}}\\
=&H_{\bar{v}_{2}\bar{u}_{1}-\bar{u}_{1}\bar{f}_{2}+\bar{v}_2(\lambda\bar{f}_{1}-\bar{u}_{1})+\bar{u}_{2}(\bar{\mu}\bar{u}_{1}-\bar{v}_{1})-\bar{\alpha}\bar{v}_2\bar{u}_2}\\
=&H_{-\bar{u}_{1}\bar{f}_{2}+\lambda\bar{v}_2\bar{f}_{1}+\bar{u}_{2}(\bar{\mu}\bar{u}_{1}-\bar{v}_{1})-\bar{\alpha}\bar{v}_2\bar{u}_2},
\end{align*}
and
\begin{align*}
&H_{g_{1}}T_{f_{2}}+\tilde{T}_{v_{1}}H_{g_{2}}-H_{g_{2}}T_{f_{1}}-\tilde{T}_{v_{2}}H_{g_{1}}\\
=&H_{g_{1}f_{2}}-\tilde{T}_{g_{1}}H_{f_{2}}+\tilde{T}_{v_1}H_{g_2}-H_{\mu v_{2}}T_{f_{1}}-H_{v_2 g_1}+H_{v_2}T_{g_1}\\
=&H_{g_{1}f_{2}-v_2 g_1}+\tilde{T}_{v_1-\bar{\lambda}g_1}H_{g_2}+H_{v_2}T_{g_1-\mu f_1}\\
=&H_{g_{1}f_{2}-v_2 g_1}+H_{(v_1-\bar{\lambda}g_1)g_2}-H_{v_1-\bar{\lambda}g_1}T_{g_2}+H_{v_2(g_1-\mu f_1)}-\tilde{T}_{v_2}H_{g_1-\mu f_1}\\
=&H_{g_{1}f_{2}-v_2 g_1+(v_1-\bar{\lambda}g_1)g_2+v_2(g_1-\mu f_1)}-H_{v_1-\bar{\lambda}g_1}T_{g_2}-\tilde{T}_{v_2}H_{g_1-\mu f_1}\\
=&H_{g_{1}f_{2}-v_2 g_1+(v_1-\bar{\lambda}g_1)g_2+v_2(g_1-\mu f_1)}+H_{\alpha v_2}T_{g_2}+\tilde{T}_{v_2}H_{\alpha g_2}\\
=&H_{g_{1}f_{2}-v_2 g_1+(v_1-\bar{\lambda}g_1)g_2+v_2(g_1-\mu f_1)+\alpha v_2 g_2}\\
=&H_{g_{1}f_{2}+g_2 v_1-\bar{\lambda}g_1g_2-\mu f_1 v_2+\alpha v_2 g_2}.
\end{align*}

In case \textbf{(L4)+(R1)}. Substituting $\brackbinom{Vg_{2-}}{(\bar{v}_{2})_{-}}=\lambda \brackbinom{Vf_{2-}}{(\bar{u}_{2})_{-}}$
and $\brackbinom{V(\bar{f_{2}})_{-}}{g_{2-}}=\mu \brackbinom{V(\bar{u}_{2})_{-}}{v_{2-}}$ into \eqref{WW}
shows that there exists a non-zero constant $\alpha$ such that\\
\begin{align*}
\brackbinom{V(\bar{f_{1}})_{-}}{g_{1-}}-\bar{\lambda}\brackbinom{V(\bar{u}_{1})_{-}}{v_{1-}}
=\alpha \brackbinom{V(\bar{u}_{2})_{-}}{v_{2-}},\ \
\bar{\mu}\brackbinom{Vf_{1-}}{(\bar{u}_{1})_{-}}-\brackbinom{Vg_{1-}}{(\bar{v}_{1})_{-}}
=\bar{\alpha}\brackbinom{Vf_{2-}}{(\bar{u}_{2})_{-}}.
\end{align*}
Substituting the above conditions into  \eqref{1.2} and \eqref{2.1}, and applying \eqref{eq:2m}, we obtain
\begin{align*}
&H_{\bar{u}_{2}}T_{\bar{f}_{1}}+\tilde{T}_{\bar{v}_{2}}H_{\bar{u}_{1}}-H_{\bar{u}_{1}}T_{\bar{f}_{2}}-\tilde{T}_{\bar{v}_{1}}H_{\bar{u}_{2}}\\
=&H_{\bar{u}_{2}}T_{\bar{f}_{1}}+H_{\bar{v}_2\bar{u}_1}-H_{\bar{v}_2}T_{\bar{u}_1}-H_{\bar{u}_1\bar{f}_2}+\tilde{T}_{\bar{u}_1}H_{\bar{f}_2}-\tilde{T}_{\bar{v}_{1}}H_{\bar{u}_{2}}\\
=&H_{\bar{v}_2\bar{u}_1-\bar{u}_1\bar{f}_2}+H_{\bar{u}_{2}}T_{\bar{f}_{1}}-H_{\lambda \bar{u}_2}T_{\bar{u}_1}+\tilde{T}_{\bar{u}_1}H_{\bar{\mu}\bar{u}_2}-\tilde{T}_{\bar{v}_{1}}H_{\bar{u}_{2}}\\
=&H_{\bar{v}_2\bar{u}_1-\bar{u}_1\bar{f}_2}+H_{\bar{u}_{2}}T_{\bar{f}_{1}-\lambda \bar{u}_1}+\tilde{T}_{\bar{\mu}\bar{u}_1-\bar{v}_{1}}H_{\bar{u}_2}\\
=&H_{\bar{v}_2\bar{u}_1-\bar{u}_1\bar{f}_2}+H_{\bar{u}_{2}(\bar{f}_{1}-\lambda \bar{u}_1)}-\tilde{T}_{\bar{u}_{2}}H_{\bar{f}_{1}-\lambda \bar{u}_1}+H_{\bar{u}_2(\bar{\mu}\bar{u}_1-\bar{v}_{1})}-H_{\bar{\mu}\bar{u}_1-\bar{v}_{1}}T_{\bar{u}_2}\\
=&H_{\bar{v}_2\bar{u}_1-\bar{u}_1\bar{f}_2+\bar{u}_{2}(\bar{f}_{1}-\lambda \bar{u}_1)+\bar{u}_2(\bar{\mu}\bar{u}_1-\bar{v}_{1})}-\tilde{T}_{\bar{u}_{2}}H_{\bar{\alpha}\bar{u}_2}-H_{\bar{\alpha}\bar{u}_2}T_{\bar{u}_2}\\
=&H_{\bar{v}_2\bar{u}_1-\bar{u}_1\bar{f}_2+\bar{u}_{2}(\bar{f}_{1}-\lambda \bar{u}_1)+\bar{u}_2(\bar{\mu}\bar{u}_1-\bar{v}_{1})-\bar{\alpha}\bar{u}^{2}_2},
\end{align*}
and
\begin{align*}
&H_{g_{1}}T_{f_{2}}+\tilde{T}_{v_{1}}H_{g_{2}}-H_{g_{2}}T_{f_{1}}-\tilde{T}_{v_{2}}H_{g_{1}}\\
=&H_{g_{1}f_{2}}-\tilde{T}_{g_{1}}H_{f_{2}}+\tilde{T}_{v_{1}}H_{\bar{\lambda}f_{2}}-H_{\mu v_{2}}T_{f_{1}}-H_{v_{2}g_{1}}+H_{v_{2}}T_{g_{1}}\\
=&H_{g_{1}f_{2}-v_{2}g_{1}}+\tilde{T}_{\bar{\lambda}v_{1}-g_{1}}H_{f_{2}}+H_{v_{2}}T_{g_{1}-\mu f_{1}}\\
=&H_{g_{1}f_{2}-v_{2}g_{1}}+H_{(\bar{\lambda}v_{1}-g_{1})f_{2}}-H_{\bar{\lambda}v_{1}-g_{1}}T_{f_{2}}+H_{v_{2}(g_{1}-\mu f_{1})}-\tilde{T}_{v_{2}}H_{g_{1}-\mu f_{1}}\\
=&H_{g_{1}f_{2}-v_{2}g_{1}+(\bar{\lambda}v_{1}-g_{1})f_{2}+v_{2}(g_{1}-\mu f_{1})}+H_{\alpha v_2}T_{f_{2}}+\tilde{T}_{v_{2}}H_{\alpha f_2}\\
=&H_{g_{1}f_{2}-v_{2}g_{1}+(\bar{\lambda}v_{1}-g_{1})f_{2}+v_{2}(g_{1}-\mu f_{1})+\alpha f_2 v_{2}}\\
=&H_{\bar{\lambda}v_1 f_2-\mu f_1 v_2 +\alpha v_2 f_2}.
\end{align*}

In case \textbf{(L1)+(R2)}. Substituting $\brackbinom{V(\bar{f_{1}})_{-}}{g_{1-}}=\lambda \brackbinom{V(\bar{u}_{1})_{-}}{v_{1-}}$
and $\brackbinom{V(\bar{u}_{2})_{-}}{v_{2-}}=\mu \brackbinom{V(\bar{f_{2}})_{-}}{g_{2-}}$ into \eqref{WW}
shows that there exists a non-zero constant $\alpha$ such that
\begin{align*}
\brackbinom{V(\bar{u}_{1})_{-}}{v_{1-}}=\alpha \brackbinom{V(\bar{f_{2}})_{-}}{g_{2-}},\ \
\brackbinom{Vf_{1-}}{(\bar{u}_{1})_{-}}-\bar{\mu}\brackbinom{Vg_{1-}}{(\bar{v}_{1})_{-}}
=\bar{\alpha}\bigg(\bar{\lambda}\brackbinom{Vf_{2-}}{(\bar{u}_{2})_{-}}-\brackbinom{Vg_{2-}}{(\bar{v}_{2})_{-}}\bigg).
\end{align*}
Substituting the above conditions into  \eqref{1.2} and \eqref{2.1}, and applying \eqref{eq:2m}, we obtain
\begin{align*}
&H_{\bar{u}_{2}}T_{\bar{f}_{1}}+\tilde{T}_{\bar{v}_{2}}H_{\bar{u}_{1}}-H_{\bar{u}_{1}}T_{\bar{f}_{2}}-\tilde{T}_{\bar{v}_{1}}H_{\bar{u}_{2}}\\
=&H_{\bar{u}_{2}\bar{f}_{1}}-\tilde{T}_{\bar{u}_{2}}H_{\bar{f}_{1}}+\tilde{T}_{\bar{v}_{2}}H_{\bar{u}_{1}}
-H_{\bar{u}_{1}\bar{f}_{2}}+\tilde{T}_{\bar{u}_{1}}H_{\bar{f}_{2}}-\tilde{T}_{\bar{v}_{1}}H_{\bar{\mu}\bar{f}_2}\\
=&H_{\bar{u}_{2}\bar{f}_{1}-\bar{u}_{1}\bar{f}_{2}}-\tilde{T}_{\bar{u}_{2}}H_{\bar{\lambda}\bar{u}_{1}}+\tilde{T}_{\bar{v}_{2}}H_{\bar{u}_{1}}
+\tilde{T}_{\bar{u}_{1}}H_{\bar{f}_{2}}-\tilde{T}_{\bar{v}_{1}}H_{\bar{\mu}\bar{f}_2}\\
=&H_{\bar{u}_{2}\bar{f}_{1}-\bar{u}_{1}\bar{f}_{2}}+\tilde{T}_{\bar{v}_{2}-\bar{\lambda}\bar{u}_{2}}H_{\bar{u}_{1}}
+\tilde{T}_{\bar{u}_{1}-\bar{\mu}\bar{v}_{1}}H_{\bar{f}_{2}}\\
=&H_{\bar{u}_{2}\bar{f}_{1}-\bar{u}_{1}\bar{f}_{2}}+\tilde{T}_{\bar{v}_{2}-\bar{\lambda}\bar{u}_{2}}H_{\bar{\alpha}\bar{f}_2}
+\tilde{T}_{\bar{u}_{1}-\bar{\mu}\bar{v}_{1}}H_{\bar{f}_{2}}\\
=&H_{\bar{u}_{2}\bar{f}_{1}-\bar{u}_{1}\bar{f}_{2}}+\tilde{T}_{\bar{\alpha}(\bar{v}_{2}-\bar{\lambda}\bar{u}_{2})+\bar{u}_{1}
-\bar{\mu}\bar{v}_{1}}H_{\bar{f}_2}\\
=&H_{\bar{u}_{2}\bar{f}_{1}-\bar{u}_{1}\bar{f}_{2}+\bar{f}_2(\bar{\alpha}(\bar{v}_{2}-\bar{\lambda}\bar{u}_{2})+\bar{u}_{1}-\bar{\mu}\bar{v}_{1})}\\
=&H_{\bar{u}_{2}\bar{f}_{1}+\bar{\alpha}\bar{f}_2\bar{v}_{2}-\bar{\lambda}\bar{\alpha}\bar{f}_2\bar{u}_{2}-\bar{\mu}\bar{v}_{1}\bar{f}_2},
\end{align*}
and
\begin{align*}
&H_{g_{1}}T_{f_{2}}+\tilde{T}_{v_{1}}H_{g_{2}}-H_{g_{2}}T_{f_{1}}-\tilde{T}_{v_{2}}H_{g_{1}}\\
=&H_{\lambda v_{1}}T_{f_{2}}+H_{v_1 g_2}-H_{v_1}T_{g_2}-H_{g_{2}}T_{f_{1}}-H_{v_2 g_1}+H_{\mu g_2}T_{g_1}\\
=&H_{v_{1}g_{2}-v_2 g_1}+H_{v_{1}}T_{\lambda f_{2}-g_2}+H_{g_2}T_{\mu g_1 -f_1}\\
=&H_{v_{1}g_{2}-v_2 g_1}+H_{\alpha g_{2}}T_{\lambda f_{2}-g_2}+H_{g_2}T_{\mu g_1 -f_1}\\
=&H_{v_{1}g_{2}-v_2 g_1}+H_{g_{2}}T_{\alpha(\lambda f_{2}-g_2)+\mu g_1 -f_1}\\
=&H_{v_{1}g_{2}-v_2 g_1+g_{2}(\alpha(\lambda f_{2}-g_2)+\mu g_1 -f_1)}\\
=&H_{v_{1}g_{2}-v_2 g_1+\alpha\lambda g_2 f_2-\alpha g^{2}_2+\mu g_1 g_2-f_1g_2}.
\end{align*}

In case \textbf{(L2)+(R2)}. Substituting $\brackbinom{V(\bar{u}_{1})_{-}}{v_{1-}}=\lambda \brackbinom{V(\bar{f_{1}})_{-}}{g_{1-}}$
and $\brackbinom{V(\bar{u}_{2})_{-}}{v_{2-}}=\mu \brackbinom{V(\bar{f_{2}})_{-}}{g_{2-}}$ into \eqref{WW}
shows that there exists a non-zero constant $\alpha$ such that
\begin{align*}
\brackbinom{V(\bar{f_{1}})_{-}}{g_{1-}}=\alpha \brackbinom{V(\bar{f_{2}})_{-}}{g_{2-}},\ \
\brackbinom{Vf_{1-}}{(\bar{u}_{1})_{-}}-\bar{\mu}\brackbinom{Vg_{1-}}{(\bar{v}_{1})_{-}}
=\bar{\alpha}\bigg(\brackbinom{Vf_{2-}}{(\bar{u}_{2})_{-}}-\bar{\lambda}\brackbinom{Vg_{2-}}{(\bar{v}_{2})_{-}}\bigg).	
\end{align*}
Substituting the above conditions into  \eqref{1.2} and \eqref{2.1}, and applying \eqref{eq:2m}, we obtain
\begin{align*}
&H_{\bar{u}_{2}}T_{\bar{f}_{1}}+\tilde{T}_{\bar{v}_{2}}H_{\bar{u}_{1}}-H_{\bar{u}_{1}}T_{\bar{f}_{2}}-\tilde{T}_{\bar{v}_{1}}H_{\bar{u}_{2}}\\
=&H_{\bar{u}_{2}}T_{\bar{f}_{1}}+\tilde{T}_{\bar{v}_{2}}H_{\bar{\lambda}\bar{f}_{1}}
-H_{\bar{u}_{1}}T_{\bar{f}_{2}}-\tilde{T}_{\bar{v}_{1}}H_{\bar{\mu}\bar{f}_{2}}\\
=&H_{\bar{u}_{2}\bar{f}_{1}}-\tilde{T}_{\bar{u}_2}H_{\bar{f}_1}+\tilde{T}_{\bar{v}_{2}}H_{\bar{\lambda}\bar{f}_{1}}
-H_{\bar{u}_{1}\bar{f}_{2}}+\tilde{T}_{\bar{u}_1}H_{\bar{f}_2}-\tilde{T}_{\bar{v}_{1}}H_{\bar{\mu}\bar{f}_{2}}\\
=&H_{\bar{u}_{2}\bar{f}_{1}-\bar{u}_{1}\bar{f}_{2}}+\tilde{T}_{\bar{\lambda}\bar{v}_2-\bar{u}_2}H_{\bar{f}_1}
+\tilde{T}_{\bar{u}_1-\bar{\mu}\bar{v}_1}H_{\bar{f}_2}\\
=&H_{\bar{u}_{2}\bar{f}_{1}-\bar{u}_{1}\bar{f}_{2}}+\tilde{T}_{\bar{\lambda}\bar{v}_2-\bar{u}_2}H_{\bar{\alpha}\bar{f}_2}
+\tilde{T}_{\bar{u}_1-\bar{\mu}\bar{v}_1}H_{\bar{f}_2}\\
=&H_{\bar{u}_{2}\bar{f}_{1}-\bar{u}_{1}\bar{f}_{2}}+\tilde{T}_{\bar{\alpha}(\bar{\lambda}\bar{v}_2-\bar{u}_2)+\bar{u}_1-\bar{\mu}\bar{v}_1}H_{\bar{f}_2}\\
=&H_{\bar{u}_{2}\bar{f}_{1}-\bar{u}_{1}\bar{f}_{2}+(\bar{\alpha}(\bar{\lambda}\bar{v}_2-\bar{u}_2)+\bar{u}_1-\bar{\mu}\bar{v}_1)\bar{f}_2}\\
=&H_{\bar{u}_{2}\bar{f}_{1}+\bar{\alpha}\bar{\lambda}\bar{v}_{2}\bar{f}_{2}-\bar{\alpha}\bar{u}_2\bar{f}_2-\bar{\mu}\bar{v}_1\bar{f}_2},
\end{align*}
and
\begin{align*}
&H_{g_{1}}T_{f_{2}}+\tilde{T}_{v_{1}}H_{g_{2}}-H_{g_{2}}T_{f_{1}}-\tilde{T}_{v_{2}}H_{g_{1}}\\
=&H_{g_{1}}T_{f_{2}}+H_{v_1 g_2}-H_{v_1}T_{g_2}-H_{g_{2}}T_{f_{1}}-H_{v_2 g_1}+H_{v_2}T_{g_1}\\
=&H_{v_{1}g_{2}-v_2 g_1}+H_{g_{1}}T_{f_{2}}-H_{\lambda g_1}T_{g_2}-H_{g_{2}}T_{f_{1}}+H_{\mu g_2}T_{g_1}\\
=&H_{v_{1}g_{2}-v_2 g_1}+H_{g_1}T_{f_2 -\lambda g_2}+H_{g_2}T_{\mu g_1-f_1}\\
=&H_{v_{1}g_{2}-v_2 g_1}+H_{\alpha g_2}T_{f_2 -\lambda g_2}+H_{g_2}T_{\mu g_1-f_1}\\
=&H_{v_{1}g_{2}-v_2 g_1}+H_{ g_2}T_{\alpha(f_2 -\lambda g_2)+\mu g_1-f_1}\\
=&H_{v_{1}g_{2}-v_2 g_1+g_2(\alpha(f_2 -\lambda g_2)+\mu g_1-f_1)}\\
=&H_{v_{1}g_{2}-v_2 g_1+\alpha f_2g_2-\lambda \alpha g^{2}_2+\mu g_1g_2-f_1g_2}.
\end{align*}

In case \textbf{(L3)+(R2)}. Substituting $\brackbinom{Vf_{2-}}{(\bar{u}_{2})_{-}}=\lambda \brackbinom{Vg_{2-}}{(\bar{v}_{2})_{-}}$
and $\brackbinom{V(\bar{u}_{2})_{-}}{v_{2-}}=\mu \brackbinom{V(\bar{f_{2}})_{-}}{g_{2-}}$ into \eqref{WW}
shows that there exists a non-zero constant $\alpha$ such that\\
\begin{align*}
\bar{\lambda}\brackbinom{V(\bar{f_{1}})_{-}}{g_{1-}}-\brackbinom{V(\bar{u}_{1})_{-}}{v_{1-}}
=\alpha \brackbinom{V(\bar{f_{2}})_{-}}{g_{2-}},\ \
\brackbinom{Vf_{1-}}{(\bar{u}_{1})_{-}}-\bar{\mu}\brackbinom{Vg_{1-}}{(\bar{v}_{1})_{-}}
=\bar{\alpha}\brackbinom{Vg_{2-}}{(\bar{v}_{2})_{-}}.
\end{align*}
Substituting the above conditions into  \eqref{1.2} and \eqref{2.1}, and applying \eqref{eq:2m}, we obtain
\begin{align*}
&H_{\bar{u}_{2}}T_{\bar{f}_{1}}+\tilde{T}_{\bar{v}_{2}}H_{\bar{u}_{1}}-H_{\bar{u}_{1}}T_{\bar{f}_{2}}-\tilde{T}_{\bar{v}_{1}}H_{\bar{u}_{2}}\\
=&H_{\lambda \bar{v}_{2}}T_{\bar{f}_{1}}+H_{\bar{v}_2\bar{u}_1}-H_{\bar{v}_2}T_{\bar{u}_1}-H_{\bar{u}_1\bar{f}_2}+\tilde{T}_{\bar{u}_1}H_{\bar{f}_2}-\tilde{T}_{\bar{v}_{1}}H_{\bar{\mu}\bar{f}_{2}}\\
=&H_{\bar{v}_2\bar{u}_1-\bar{u}_1\bar{f}_2}+H_{\bar{v}_{2}}T_{\lambda\bar{f}_{1}-\bar{u}_1}+\tilde{T}_{\bar{u}_1-\bar{\mu}\bar{v}_1}H_{\bar{f}_2}\\
=&H_{\bar{v}_2\bar{u}_1-\bar{u}_1\bar{f}_2}+H_{\bar{v}_{2}(\lambda\bar{f}_{1}-\bar{u}_1)}-\tilde{T}_{\bar{v}_{2}}H_{\lambda\bar{f}_{1}-\bar{u}_1}+H_{\bar{f}_2(\bar{u}_1-\bar{\mu}\bar{v}_1)}-H_{\bar{u}_1-\bar{\mu}\bar{v}_1}T_{\bar{f}_2}\\
=&H_{\bar{v}_2\bar{u}_1-\bar{u}_1\bar{f}_2+\bar{v}_{2}(\lambda\bar{f}_{1}-\bar{u}_1)+\bar{f}_2(\bar{u}_1-\bar{\mu}\bar{v}_1)}-\tilde{T}_{\bar{v}_{2}}H_{\bar{\alpha}\bar{f}_{2}}-H_{\bar{\alpha}\bar{v}_{2}}T_{\bar{f}_{2}}\\
=&H_{\bar{v}_2\bar{u}_1-\bar{u}_1\bar{f}_2+\bar{v}_{2}(\lambda\bar{f}_{1}-\bar{u}_1)+\bar{f}_2(\bar{u}_1-\bar{\mu}\bar{v}_1)-\bar{\alpha}\bar{f}_{2}\bar{v}_2}\\
=&H_{\lambda \bar{v}_2\bar{f}_1-\bar{\mu}\bar{v}_1\bar{f}_2-\bar{\alpha}\bar{v}_2\bar{f}_2},
\end{align*}
and
\begin{align*}
&H_{g_{1}}T_{f_{2}}+\tilde{T}_{v_{1}}H_{g_{2}}-H_{g_{2}}T_{f_{1}}-\tilde{T}_{v_{2}}H_{g_{1}}\\
=&H_{g_{1}f_{2}}-\tilde{T}_{g_{1}}H_{f_{2}}+\tilde{T}_{v_{1}}H_{g_{2}}-H_{g_{2}}T_{f_{1}}-H_{v_{2}g_{1}}+H_{v_{2}}T_{g_{1}}\\
=&H_{g_{1}f_{2}}-\tilde{T}_{g_{1}}H_{\bar{\lambda}g_{2}}+\tilde{T}_{v_{1}}H_{g_{2}}-H_{g_{2}}T_{f_{1}}-H_{v_{2}g_{1}}+H_{\mu g_{2}}T_{g_{1}}\\
=&H_{g_{1}f_{2}-v_{2}g_{1}}+\tilde{T}_{v_{1}-\bar{\lambda}g_{1}}H_{g_{2}}+H_{g_{2}}T_{\mu g_{1}- f_{1}}\\
=&H_{g_{1}f_{2}-v_{2}g_{1}}+H_{(v_{1}-\bar{\lambda}g_{1})g_{2}}-H_{v_{1}-\bar{\lambda}g_{1}}T_{g_{2}}+H_{g_{2}(\mu g_{1}- f_{1})}-\tilde{T}_{g_{2}}H_{\mu g_{1}-f_{1}}\\
=&H_{g_{1}f_{2}-v_{2}g_{1}+(v_{1}-\bar{\lambda}g_{1})g_{2}+g_{2}(\mu g_{1}- f_{1})}+H_{\alpha g_2}T_{g_{2}}+\tilde{T}_{g_{2}}H_{\alpha g_2}\\
=&H_{g_{1}f_{2}-v_{2}g_{1}+(v_{1}-\bar{\lambda}g_{1})g_{2}+g_{2}(\mu g_{1}- f_{1})+\alpha g^{2}_2}.
\end{align*}

In case \textbf{(L4)+(R2)}. Substituting $\brackbinom{Vg_{2-}}{(\bar{v}_{2})_{-}}=\lambda \brackbinom{Vf_{2-}}{(\bar{u}_{2})_{-}}$
and $\brackbinom{V(\bar{u}_{2})_{-}}{v_{2-}}=\mu \brackbinom{V(\bar{f_{2}})_{-}}{g_{2-}}$ into \eqref{WW}
shows that there exists a non-zero constant $\alpha$ such that
\begin{align*}
\brackbinom{V(\bar{f_{1}})_{-}}{g_{1-}}-\bar{\lambda}\brackbinom{V(\bar{u}_{1})_{-}}{v_{1-}}
=\alpha \brackbinom{V(\bar{f_{2}})_{-}}{g_{2-}},\ \
\brackbinom{Vf_{1-}}{(\bar{u}_{1})_{-}}-\bar{\mu}\brackbinom{Vg_{1-}}{(\bar{v}_{1})_{-}}
=\bar{\alpha}\brackbinom{Vf_{2-}}{(\bar{u}_{2})_{-}}.
\end{align*}
Substituting the above conditions into  \eqref{1.2} and \eqref{2.1}, and applying \eqref{eq:2m}, we obtain
\begin{align*}
&H_{\bar{u}_{2}}T_{\bar{f}_{1}}+\tilde{T}_{\bar{v}_{2}}H_{\bar{u}_{1}}-H_{\bar{u}_{1}}T_{\bar{f}_{2}}-\tilde{T}_{\bar{v}_{1}}H_{\bar{u}_{2}}\\
=&H_{\bar{u}_{2}}T_{\bar{f}_{1}}+H_{\bar{v}_2\bar{u}_1}-H_{\bar{v}_2}T_{\bar{u}_1}-H_{\bar{u}_1\bar{f}_2}+\tilde{T}_{\bar{u}_1}H_{\bar{f}_2}-\tilde{T}_{\bar{v}_{1}}H_{\bar{\mu}\bar{f}_{2}}\\
=&H_{\bar{u}_{2}}T_{\bar{f}_{1}}+H_{\bar{v}_2\bar{u}_1}-H_{\lambda \bar{u}_2}T_{\bar{u}_1}-H_{\bar{u}_1\bar{f}_2}+\tilde{T}_{\bar{u}_1}H_{\bar{f}_2}-\tilde{T}_{\bar{v}_{1}}H_{\bar{\mu}\bar{f}_{2}}\\
=&H_{\bar{v}_2\bar{u}_1-\bar{u}_1\bar{f}_2}+H_{\bar{u}_{2}}T_{\bar{f}_{1}-\lambda\bar{u}_1}+\tilde{T}_{\bar{u}_1-\bar{\mu}\bar{v}_1}H_{\bar{f}_2}\\
=&H_{\bar{v}_2\bar{u}_1-\bar{u}_1\bar{f}_2}+H_{\bar{u}_{2}(\bar{f}_{1}-\lambda\bar{u}_1)}-\tilde{T}_{\bar{u}_{2}}H_{\bar{f}_{1}-\lambda\bar{u}_1}+H_{\bar{f}_2(\bar{u}_1-\bar{\mu}\bar{v}_1)}-H_{\bar{u}_1-\bar{\mu}\bar{v}_1}T_{\bar{f}_2}\\
=&H_{\bar{v}_2\bar{u}_1-\bar{u}_1\bar{f}_2+\bar{u}_{2}(\bar{f}_{1}-\lambda\bar{u}_1)+\bar{f}_2(\bar{u}_1-\bar{\mu}\bar{v}_1)}-\tilde{T}_{\bar{u}_{2}}H_{\bar{\alpha}\bar{f}_{2}}-H_{\bar{\alpha}\bar{u}_{2}}T_{\bar{f}_{2}}\\
=&H_{\bar{v}_2\bar{u}_1-\bar{u}_1\bar{f}_2+\bar{u}_{2}(\bar{f}_{1}-\lambda\bar{u}_1)+\bar{f}_2(\bar{u}_1-\bar{\mu}\bar{v}_1)-\bar{\alpha}\bar{f}_{2}\bar{u}_2}\\
=&H_{\bar{v}_2\bar{u}_1+\bar{f}_{1}\bar{u}_{2}-\lambda\bar{u}_1\bar{u}_{2}-\bar{\mu}\bar{v}_1\bar{f}_2-\bar{\alpha}\bar{f}_{2}\bar{u}_2},
\end{align*}
and
\begin{align*}
&H_{g_{1}}T_{f_{2}}+\tilde{T}_{v_{1}}H_{g_{2}}-H_{g_{2}}T_{f_{1}}-\tilde{T}_{v_{2}}H_{g_{1}}\\
=&H_{g_{1}f_{2}}-\tilde{T}_{g_{1}}H_{f_{2}}+\tilde{T}_{v_{1}}H_{g_{2}}-H_{g_{2}}T_{f_{1}}-H_{v_{2}g_{1}}+H_{v_{2}}T_{g_{1}}\\
=&H_{g_{1}f_{2}}-\tilde{T}_{g_{1}}H_{f_{2}}+\tilde{T}_{v_{1}}H_{\bar{\lambda}f_{2}}-H_{g_{2}}T_{f_{1}}-H_{v_{2}g_{1}}+H_{\mu g_{2}}T_{g_{1}}\\
=&H_{g_{1}f_{2}-v_{2}g_{1}}+\tilde{T}_{\bar{\lambda}v_{1}-g_{1}}H_{f_{2}}+H_{g_{2}}T_{\mu g_{1}- f_{1}}\\
=&H_{g_{1}f_{2}-v_{2}g_{1}}+H_{(\bar{\lambda}v_{1}-g_{1})f_{2}}-H_{\bar{\lambda}v_{1}-g_{1}}T_{f_{2}}+H_{g_{2}(\mu g_{1}- f_{1})}-\tilde{T}_{g_{2}}H_{\mu g_{1}-f_{1}}\\
=&H_{g_{1}f_{2}-v_{2}g_{1}+(\bar{\lambda}v_{1}-g_{1})f_{2}+g_{2}(\mu g_{1}- f_{1})}+H_{\alpha g_2}T_{f_{2}}+\tilde{T}_{g_{2}}H_{\alpha f_2}\\
=&H_{g_{1}f_{2}-v_{2}g_{1}+(\bar{\lambda}v_{1}-g_{1})f_{2}+g_{2}(\mu g_{1}- f_{1})+\alpha f_2 g_2}\\
=&H_{-v_{2}g_{1}+\bar{\lambda}v_{1}f_{2}+g_{2}(\mu g_{1}- f_{1})+\alpha f_2 g_2}.
\end{align*}

In case \textbf{(L1)+(R3)}. Substituting $\brackbinom{V(\bar{f_{1}})_{-}}{g_{1-}}=\lambda \brackbinom{V(\bar{u}_{1})_{-}}{v_{1-}}$
and $\brackbinom{Vf_{1-}}{(\bar{u}_{1})_{-}}=\mu \brackbinom{Vg_{1-}}{(\bar{v}_{1})_{-}}$ into \eqref{WW}
shows that there exists a non-zero constant $\alpha$ such that
\begin{align*}
\brackbinom{V(\bar{u}_{1})_{-}}{v_{1-}}=\alpha \bigg(\bar{\mu}\brackbinom{V(\bar{f_{2}})_{-}}{g_{2-}}-\brackbinom{V(\bar{u}_{2})_{-}}{v_{2-}}\bigg),\ \
\brackbinom{Vg_{1-}}{(\bar{v}_{1})_{-}}=\bar{\alpha}\bigg(\bar{\lambda}\brackbinom{Vf_{2-}}{(\bar{u}_{2})_{-}}-\brackbinom{Vg_{2-}}{(\bar{v}_{2})_{-}}\bigg).
\end{align*}
Substituting the above conditions into  \eqref{1.2} and \eqref{2.1}, and applying \eqref{eq:2m}, we obtain
\begin{align*}
&H_{\bar{u}_{2}}T_{\bar{f}_{1}}+\tilde{T}_{\bar{v}_{2}}H_{\bar{u}_{1}}-H_{\bar{u}_{1}}T_{\bar{f}_{2}}-\tilde{T}_{\bar{v}_{1}}H_{\bar{u}_{2}}\\
=&H_{\bar{u}_{2}\bar{f}_{1}}-\tilde{T}_{\bar{u}_{2}}H_{\bar{f}_{1}}+\tilde{T}_{\bar{v}_{2}}H_{\bar{u}_{1}}-H_{\bar{u}_{1}}T_{\bar{f}_{2}}
-H_{\bar{v}_{1}\bar{u}_{2}}+H_{\bar{v}_{1}}T_{\bar{u}_{2}}\\
=&H_{\bar{u}_{2}\bar{f}_{1}}-\tilde{T}_{\bar{u}_{2}}H_{\bar{\lambda}\bar{u}_{1}}+\tilde{T}_{\bar{v}_{2}}H_{\bar{u}_{1}}-H_{\mu \bar{v}_{1}}T_{\bar{f}_{2}}
-H_{\bar{v}_{1}\bar{u}_{2}}+H_{\bar{v}_{1}}T_{\bar{u}_{2}}\\
=&H_{\bar{u}_{2}\bar{f}_{1}-\bar{v}_{1}\bar{u}_{2}}+\tilde{T}_{\bar{v}_2-\bar{\lambda}\bar{u}_{2}}H_{\bar{u}_{1}}+H_{\bar{v}_{1}}T_{\bar{u}_2-\mu \bar{f}_{2}}\\
=&H_{\bar{u}_{2}\bar{f}_{1}-\bar{v}_{1}\bar{u}_{2}}+\tilde{T}_{\bar{v}_2-\bar{\lambda}\bar{u}_{2}}H_{\bar{\alpha}(\mu \bar{f}_2-\bar{u}_2)}+H_{\bar{\alpha}(\bar{\lambda}\bar{u}_2-\bar{v}_2)}T_{\bar{u}_2-\mu \bar{f}_{2}}\\
=&H_{\bar{u}_{2}\bar{f}_{1}-\bar{v}_{1}\bar{u}_{2}+\bar{\alpha}(\mu \bar{f}_2-\bar{u}_2)(\bar{v}_2-\bar{\lambda}\bar{u}_{2})},
\end{align*}
and
\begin{align*}
&H_{g_{1}}T_{f_{2}}+\tilde{T}_{v_{1}}H_{g_{2}}-H_{g_{2}}T_{f_{1}}-\tilde{T}_{v_{2}}H_{g_{1}}\\
=&H_{\lambda v_{1}}T_{f_{2}}+H_{v_1 g_2}-H_{v_{1}}T_{g_{2}}-H_{g_{2}f_{1}}+\tilde{T}_{g_{2}}H_{f_{1}}-\tilde{T}_{v_{2}}H_{g_{1}}\\
=&H_{\lambda v_{1}}T_{f_{2}}+H_{v_1 g_2}-H_{v_{1}}T_{g_{2}}-H_{g_{2}f_{1}}+\tilde{T}_{g_{2}}H_{\bar{\mu}g_{1}}-\tilde{T}_{v_{2}}H_{g_{1}}\\
=&H_{v_1 g_2-g_{2}f_{1}}+H_{v_1}T_{\lambda f_2- g_2}+\tilde{T}_{\bar{\mu} g_2-v_{2}}H_{g_{1}}\\
=&H_{v_1 g_2-g_{2}f_{1}}+H_{\alpha(\bar{\mu} g_2-v_{2})}T_{\lambda f_2- g_2}+\tilde{T}_{\bar{\mu} g_2-v_{2}}H_{\alpha(\lambda f_2- g_2)}\\
=&H_{v_1 g_2-g_{2}f_{1}+\alpha(\bar{\mu} g_2-v_{2})(\lambda f_2- g_2)}.
\end{align*}

In case \textbf{(L2)+(R3)}. Substituting $\brackbinom{V(\bar{u}_{1})_{-}}{v_{1-}}=\lambda \brackbinom{V(\bar{f_{1}})_{-}}{g_{1-}}$
and $\brackbinom{Vf_{1-}}{(\bar{u}_{1})_{-}}=\mu \brackbinom{Vg_{1-}}{(\bar{v}_{1})_{-}}$ into \eqref{WW}
shows that there exists a non-zero constant $\alpha$ such that
\begin{align*}
\brackbinom{V(\bar{f_{1}})_{-}}{g_{1-}}=\alpha \bigg(\bar{\mu}\brackbinom{V(\bar{f_{2}})_{-}}{g_{2-}}-\brackbinom{V(\bar{u}_{2})_{-}}{v_{2-}}\bigg),\ \
\brackbinom{Vg_{1-}}{(\bar{v}_{1})_{-}}=\bar{\alpha}\bigg(\brackbinom{Vf_{2-}}{(\bar{u}_{2})_{-}}-\bar{\lambda}\brackbinom{Vg_{2-}}{(\bar{v}_{2})_{-}}\bigg).
\end{align*}
Substituting the above conditions into  \eqref{1.2} and \eqref{2.1}, and applying \eqref{eq:2m}, we obtain
\begin{align*}
&H_{\bar{u}_{2}}T_{\bar{f}_{1}}+\tilde{T}_{\bar{v}_{2}}H_{\bar{u}_{1}}-H_{\bar{u}_{1}}T_{\bar{f}_{2}}-\tilde{T}_{\bar{v}_{1}}H_{\bar{u}_{2}}\\
=&H_{\bar{u}_{2}\bar{f}_{1}}-\tilde{T}_{\bar{u}_{2}}H_{\bar{f}_{1}}+\tilde{T}_{\bar{v}_{2}}H_{\bar{u}_{1}}-H_{\bar{u}_{1}}T_{\bar{f}_{2}}
-H_{\bar{v}_{1}\bar{u}_{2}}+H_{\bar{v}_{1}}T_{\bar{u}_{2}}\\
=&H_{\bar{u}_{2}\bar{f}_{1}}-\tilde{T}_{\bar{u}_{2}}H_{\bar{f}_{1}}+\tilde{T}_{\bar{v}_{2}}H_{\bar{\lambda}\bar{f}_{1}}-H_{\mu \bar{v}_{1}}T_{\bar{f}_{2}}
-H_{\bar{v}_{1}\bar{u}_{2}}+H_{\bar{v}_{1}}T_{\bar{u}_{2}}\\
=&H_{\bar{u}_{2}\bar{f}_{1}-\bar{v}_{1}\bar{u}_{2}}+\tilde{T}_{\bar{\lambda}\bar{v}_2-\bar{u}_{2}}H_{\bar{f}_{1}}+H_{\bar{v}_{1}}T_{\bar{u}_2-\mu \bar{f}_{2}}\\
=&H_{\bar{u}_{2}\bar{f}_{1}-\bar{v}_{1}\bar{u}_{2}}+\tilde{T}_{\bar{\lambda}\bar{v}_2-\bar{u}_{2}}H_{\bar{\alpha}(\mu \bar{f}_2-\bar{u}_2)}+H_{\bar{\alpha}(\bar{u}_2-\bar{\lambda}\bar{v}_2)}T_{\bar{u}_2-\mu \bar{f}_{2}}\\
=&H_{\bar{u}_{2}\bar{f}_{1}-\bar{v}_{1}\bar{u}_{2}+\bar{\alpha}(\mu \bar{f}_2-\bar{u}_2)(\bar{\lambda}\bar{v}_2-\bar{u}_{2})},
\end{align*}
and
\begin{align*}
&H_{g_{1}}T_{f_{2}}+\tilde{T}_{v_{1}}H_{g_{2}}-H_{g_{2}}T_{f_{1}}-\tilde{T}_{v_{2}}H_{g_{1}}\\
=&H_{g_{1}}T_{f_{2}}+H_{v_1 g_2}-H_{v_{1}}T_{g_{2}}-H_{g_{2}f_{1}}+\tilde{T}_{g_{2}}H_{f_{1}}-\tilde{T}_{v_{2}}H_{g_{1}}\\
=&H_{g_{1}}T_{f_{2}}+H_{v_1 g_2}-H_{\lambda g_{1}}T_{g_{2}}-H_{g_{2}f_{1}}+\tilde{T}_{g_{2}}H_{\bar{\mu}g_{1}}-\tilde{T}_{v_{2}}H_{g_{1}}\\
=&H_{v_1 g_2-g_{2}f_{1}}+H_{g_1}T_{f_2-\lambda g_2}+\tilde{T}_{\bar{\mu}g_2-v_2}H_{g_1}\\
=&H_{v_1 g_2-g_{2}f_{1}}+H_{\alpha (\bar{\mu}g_2-v_2)}T_{f_2-\lambda g_2}+\tilde{T}_{\bar{\mu}g_2-v_2}H_{\alpha (f_2-\lambda g_2)}\\
=&H_{v_1 g_2-g_{2}f_{1}+\alpha (\bar{\mu}g_2-v_2)(f_2-\lambda g_2)}.
\end{align*}

In case \textbf{(L3)+(R3)}. Substituting $\brackbinom{Vf_{2-}}{(\bar{u}_{2})_{-}}=\lambda \brackbinom{Vg_{2-}}{(\bar{v}_{2})_{-}}$
and $\brackbinom{Vf_{1-}}{(\bar{u}_{1})_{-}}=\mu \brackbinom{Vg_{1-}}{(\bar{v}_{1})_{-}}$ into \eqref{WW}
shows that there exists a non-zero constant $\alpha$ such that
\begin{align*}
\bar{\lambda}\brackbinom{V(\bar{f_{1}})_{-}}{g_{1-}}-\brackbinom{V(\bar{u}_{1})_{-}}{v_{1-}}
=\alpha \bigg(\bar{\mu}\brackbinom{V(\bar{f_{2}})_{-}}{g_{2-}}-\brackbinom{V(\bar{u}_{2})_{-}}{v_{2-}}\bigg),\ \
\brackbinom{Vg_{1-}}{(\bar{v}_{1})_{-}}=\bar{\alpha}\brackbinom{Vg_{2-}}{(\bar{v}_{2})_{-}}.
\end{align*}
Substituting the above conditions into  \eqref{1.2} and \eqref{2.1}, and applying \eqref{eq:2m}, we obtain
\begin{align*}
&H_{\bar{u}_{2}}T_{\bar{f}_{1}}+\tilde{T}_{\bar{v}_{2}}H_{\bar{u}_{1}}-H_{\bar{u}_{1}}T_{\bar{f}_{2}}-\tilde{T}_{\bar{v}_{1}}H_{\bar{u}_{2}}\\
=&H_{\lambda \bar{v}_{2}}T_{\bar{f}_{1}}+H_{\bar{v}_{2}\bar{u}_{1}}-H_{\bar{v}_{2}}T_{\bar{u}_{1}}-H_{\mu \bar{v}_{1}}T_{\bar{f}_{2}}-H_{\bar{v}_1 \bar{u}_2}+H_{\bar{v}_{1}}T_{\bar{u}_{2}}\\
=&H_{\bar{v}_{2}\bar{u}_{1}-\bar{v}_1 \bar{u}_2}+H_{\bar{v}_2}T_{\lambda \bar{f}_{1}-\bar{u}_1}+H_{\bar{v}_{1}}T_{\bar{u}_2-\mu \bar{f}_{2}}\\
=&H_{\bar{v}_{2}\bar{u}_{1}-\bar{v}_1 \bar{u}_2}+H_{\bar{v}_2(\lambda \bar{f}_{1}-\bar{u}_1)}-\tilde{T}_{\bar{v}_2}H_{\lambda \bar{f}_{1}-\bar{u}_1}+H_{\bar{v}_{1}}T_{\bar{u}_2-\mu \bar{f}_{2}}\\
=&H_{\bar{v}_{2}\bar{u}_{1}-\bar{v}_1 \bar{u}_2}+H_{\bar{v}_2(\lambda \bar{f}_{1}-\bar{u}_1)}+\tilde{T}_{\bar{v}_2}H_{\bar{\alpha}(\bar{u}_2-\mu \bar{f}_{2})}+H_{\bar{\alpha}\bar{v}_{2}}T_{\bar{u}_2-\mu \bar{f}_{2}}\\
=&H_{\bar{v}_{2}\bar{u}_{1}-\bar{v}_1 \bar{u}_2+\bar{v}_2(\lambda \bar{f}_{1}-\bar{u}_1)+\bar{\alpha}\bar{v}_{2}(\bar{u}_2-\mu \bar{f}_{2})}\\
=&H_{-\bar{v}_1 \bar{u}_2+\lambda\bar{v}_2\bar{f}_{1}+\bar{\alpha}\bar{v}_{2}(\bar{u}_2-\mu \bar{f}_{2})},
\end{align*}
and
\begin{align*}
&H_{g_{1}}T_{f_{2}}+\tilde{T}_{v_{1}}H_{g_{2}}-H_{g_{2}}T_{f_{1}}-\tilde{T}_{v_{2}}H_{g_{1}}\\
=&H_{g_{1}f_{2}}-\tilde{T}_{g_{1}}H_{f_{2}}+\tilde{T}_{v_{1}}H_{g_{2}}-H_{g_{2}f_{1}}+\tilde{T}_{g_{2}}H_{f_{1}}-\tilde{T}_{v_{2}}H_{g_{1}}\\
=&H_{g_{1}f_{2}-g_{2}f_{1}}-\tilde{T}_{g_{1}}H_{\bar{\lambda}g_{2}}+\tilde{T}_{v_{1}}H_{g_{2}}+\tilde{T}_{g_{2}}H_{\bar{\mu}g_{1}}-\tilde{T}_{v_{2}}H_{g_{1}}\\
=&H_{g_{1}f_{2}-g_{2}f_{1}}+\tilde{T}_{v_{1}-\bar{\lambda}g_{1}}H_{g_{2}}+\tilde{T}_{\bar{\mu}g_{2}-v_2}H_{g_{1}}\\
=&H_{g_{1}f_{2}-g_{2}f_{1}}+H_{(v_{1}-\bar{\lambda}g_{1})g_{2}}-H_{v_{1}-\bar{\lambda}g_{1}}T_{g_{2}}+\tilde{T}_{\bar{\mu}g_{2}-v_2}H_{\alpha g_{2}}\\
=&H_{g_{1}f_{2}-g_{2}f_{1}}+H_{(v_{1}-\bar{\lambda}g_{1})g_{2}}+H_{\alpha(\bar{\mu}g_{2}-v_2)}T_{g_{2}}+\tilde{T}_{\bar{\mu}g_{2}-v_2}H_{\alpha g_{2}}\\
=&H_{g_{1}f_{2}-g_{2}f_{1}+(v_{1}-\bar{\lambda}g_{1})g_{2}+\alpha g_{2}(\bar{\mu}g_{2}-v_2)}.
\end{align*}

In case \textbf{(L4)+(R3)}. Substituting $\brackbinom{Vg_{2-}}{(\bar{v}_{2})_{-}}=\lambda \brackbinom{Vf_{2-}}{(\bar{u}_{2})_{-}}$
and $\brackbinom{Vf_{1-}}{(\bar{u}_{1})_{-}}=\mu \brackbinom{Vg_{1-}}{(\bar{v}_{1})_{-}}$ into \eqref{WW}
shows that there exists a non-zero constant $\alpha$ such that
\begin{align*}
\brackbinom{V(\bar{f_{1}})_{-}}{g_{1-}}-\bar{\lambda}\brackbinom{V(\bar{u}_{1})_{-}}{v_{1-}}
=\alpha \bigg(\bar{\mu}\brackbinom{V(\bar{f_{2}})_{-}}{g_{2-}}-\brackbinom{V(\bar{u}_{2})_{-}}{v_{2-}}\bigg),\ \
\brackbinom{Vg_{1-}}{(\bar{v}_{1})_{-}}=\bar{\alpha}\brackbinom{Vf_{2-}}{(\bar{u}_{2})_{-}}.
\end{align*}
Substituting the above conditions into  \eqref{1.2} and \eqref{2.1}, and applying \eqref{eq:2m}, we obtain
\begin{align*}
&H_{\bar{u}_{2}}T_{\bar{f}_{1}}+\tilde{T}_{\bar{v}_{2}}H_{\bar{u}_{1}}-H_{\bar{u}_{1}}T_{\bar{f}_{2}}-\tilde{T}_{\bar{v}_{1}}H_{\bar{u}_{2}}\\
=&H_{\bar{u}_{2}}T_{\bar{f}_{1}}+H_{\bar{v}_{2}\bar{u}_{1}}-H_{\bar{v}_{2}}T_{\bar{u}_{1}}-H_{\bar{u}_{1}}T_{\bar{f}_{2}}-H_{\bar{v}_{1}\bar{u}_{2}}+H_{\bar{v}_{1}}T_{\bar{u}_{2}}\\
=&H_{\bar{u}_{2}}T_{\bar{f}_{1}}+H_{\bar{v}_{2}\bar{u}_{1}}-H_{\lambda\bar{u}_{2}}T_{\bar{u}_{1}}-H_{\mu\bar{v}_{1}}T_{\bar{f}_{2}}-H_{\bar{v}_{1}\bar{u}_{2}}+H_{\bar{v}_{1}}T_{\bar{u}_{2}}\\
=&H_{\bar{v}_{2}\bar{u}_{1}-\bar{v}_{1}\bar{u}_{2}}+H_{\bar{u}_{2}}T_{\bar{f}_{1}-\lambda \bar{u}_1}+H_{\bar{v}_1}T_{\bar{u}_2-\mu \bar{f}_2}\\
=&H_{\bar{v}_{2}\bar{u}_{1}-\bar{v}_{1}\bar{u}_{2}}+H_{\bar{u}_{2}(\bar{f}_{1}-\lambda \bar{u}_1)}-\tilde{T}_{\bar{u}_{2}}H_{\bar{f}_{1}-\lambda \bar{u}_1}+H_{\bar{v}_1}T_{\bar{u}_2-\mu \bar{f}_2}\\
=&H_{\bar{v}_{2}\bar{u}_{1}-\bar{v}_{1}\bar{u}_{2}}+H_{\bar{u}_{2}(\bar{f}_{1}-\lambda \bar{u}_1)}+\tilde{T}_{\bar{u}_{2}}H_{\bar{\alpha} (\bar{u}_2-\mu \bar{f}_2)}+H_{\bar{\alpha}\bar{u}_2}T_{\bar{u}_2-\mu \bar{f}_2}\\
=&H_{\bar{v}_{2}\bar{u}_{1}-\bar{v}_{1}\bar{u}_{2}+\bar{u}_{2}(\bar{f}_{1}-\lambda \bar{u}_1)+\bar{\alpha} \bar{u}_{2}(\bar{u}_2-\mu \bar{f}_2)},
\end{align*}
and
\begin{align*}
&H_{g_{1}}T_{f_{2}}+\tilde{T}_{v_{1}}H_{g_{2}}-H_{g_{2}}T_{f_{1}}-\tilde{T}_{v_{2}}H_{g_{1}}\\
=&H_{g_{1}f_{2}}-\tilde{T}_{g_{1}}H_{f_{2}}+\tilde{T}_{v_{1}}H_{\bar{\lambda} f_{2}}-H_{g_{2}f_{1}}+\tilde{T}_{g_{2}}H_{f_{1}}-\tilde{T}_{v_{2}}H_{g_{1}}\\
=&H_{g_{1}f_{2}}-\tilde{T}_{g_{1}}H_{f_{2}}+\tilde{T}_{v_{1}}H_{\bar{\lambda} f_{2}}-H_{g_{2}f_{1}}+\tilde{T}_{g_{2}}H_{\bar{\mu}g_{1}}-\tilde{T}_{v_{2}}H_{g_{1}}\\
=&H_{g_{1}f_{2}-g_{2}f_{1}}+\tilde{T}_{\bar{\lambda}v_1- g_{1}}H_{f_{2}}+\tilde{T}_{\bar{\mu}g_{2}-v_2}H_{g_{1}}\\
=&H_{g_{1}f_{2}-g_{2}f_{1}}+H_{(\bar{\lambda}v_1- g_{1})f_2}-H_{\bar{\lambda}v_1- g_{1}}T_{f_{2}}+\tilde{T}_{\bar{\mu}g_{2}-v_2}H_{g_{1}}\\
=&H_{g_{1}f_{2}-g_{2}f_{1}}+H_{(\bar{\lambda}v_1- g_{1})f_2}+H_{\alpha(\bar{\mu}g_{2}-v_2)}T_{f_{2}}+\tilde{T}_{\bar{\mu}g_{2}-v_2}H_{\alpha f_{2}}\\
=&H_{g_{1}f_{2}-g_{2}f_{1}+(\bar{\lambda}v_1- g_{1})f_2+\alpha f_{2}(\bar{\mu}g_{2}-v_2)}\\
=&H_{-g_{2}f_{1}+\bar{\lambda}v_1f_2+\alpha f_{2}(\bar{\mu}g_{2}-v_2)}.
\end{align*}

In case \textbf{(L1)+(R4)}. Substituting $\brackbinom{V(\bar{f_{1}})_{-}}{g_{1-}}=\lambda \brackbinom{V(\bar{u}_{1})_{-}}{v_{1-}}$
and $\brackbinom{Vg_{1-}}{(\bar{v}_{1})_{-}}=\mu \brackbinom{Vf_{1-}}{(\bar{u}_{1})_{-}}$ into \eqref{WW}
shows that there exists a non-zero constant $\alpha$ such that
\begin{align*}
\brackbinom{V(\bar{u}_{1})_{-}}{v_{1-}}
=\alpha \bigg(\brackbinom{V(\bar{f_{2}})_{-}}{g_{2-}}-\bar{\mu}\brackbinom{V(\bar{u}_{2})_{-}}{v_{2-}}\bigg),\ \
\brackbinom{Vf_{1-}}{(\bar{u}_{1})_{-}}=\bar{\alpha}\bigg(\bar{\lambda}\brackbinom{Vf_{2-}}{(\bar{u}_{2})_{-}}-\brackbinom{Vg_{2-}}{(\bar{v}_{2})_{-}}\bigg).
\end{align*}
Substituting the above conditions into  \eqref{1.2} and \eqref{2.1}, and applying \eqref{eq:2m}, we obtain
\begin{align*}
&H_{\bar{u}_{2}}T_{\bar{f}_{1}}+\tilde{T}_{\bar{v}_{2}}H_{\bar{u}_{1}}-H_{\bar{u}_{1}}T_{\bar{f}_{2}}-\tilde{T}_{\bar{v}_{1}}H_{\bar{u}_{2}}\\
=&H_{\bar{u}_{2}\bar{f}_{1}}-\tilde{T}_{\bar{u}_{2}}H_{\bar{f}_{1}}+\tilde{T}_{\bar{v}_{2}}H_{\bar{u}_{1}}-H_{\bar{u}_{1}}T_{\bar{f}_{2}}-H_{\bar{v}_1\bar{u}_2}+H_{\bar{v}_{1}}T_{\bar{u}_{2}}\\
=&H_{\bar{u}_{2}\bar{f}_{1}}-\tilde{T}_{\bar{u}_{2}}H_{\bar{\lambda}\bar{u}_{1}}+\tilde{T}_{\bar{v}_{2}}H_{\bar{u}_{1}}-H_{\bar{u}_{1}}T_{\bar{f}_{2}}-H_{\bar{v}_1\bar{u}_2}+H_{\mu\bar{u}_{1}}T_{\bar{u}_{2}}\\
=&H_{\bar{u}_{2}\bar{f}_{1}-\bar{v}_1\bar{u}_2}+\tilde{T}_{\bar{v}_2-\bar{\lambda}\bar{u}_2}H_{\bar{u}_1}+H_{\bar{u}_1}T_{\mu \bar{u}_2-\bar{f}_2}\\
=&H_{\bar{u}_{2}\bar{f}_{1}-\bar{v}_1\bar{u}_2}+\tilde{T}_{\bar{v}_2-\bar{\lambda}\bar{u}_2}H_{\bar{\alpha}(\bar{f}_2-\mu\bar{u}_2)}+H_{\bar{\alpha}(\bar{\lambda}\bar{u}_2-\bar{v}_2)}T_{\mu \bar{u}_2-\bar{f}_2}\\
=&H_{\bar{u}_{2}\bar{f}_{1}-\bar{v}_1\bar{u}_2+\bar{\alpha}(\bar{v}_2-\bar{\lambda}\bar{u}_2)(\bar{f}_2-\mu\bar{u}_2)},
\end{align*}
and
\begin{align*}
&H_{g_{1}}T_{f_{2}}+\tilde{T}_{v_{1}}H_{g_{2}}-H_{g_{2}}T_{f_{1}}-\tilde{T}_{v_{2}}H_{g_{1}}\\
=&H_{\lambda v_{1}}T_{f_{2}}+H_{v_{1}g_{2}}-H_{v_1}T_{g_2}-H_{g_{2} f_{1}}+\tilde{T}_{g_{2}}H_{f_{1}}-\tilde{T}_{v_{2}}H_{\bar{\mu}f_{1}}\\
=&H_{v_{1}g_{2}-g_{2} f_{1}}+H_{v_1}T_{\lambda f_2-g_2}+\tilde{T}_{g_{2}-\bar{\mu}v_2}H_{f_{1}}\\
=&H_{v_{1}g_{2}-g_{2} f_{1}}+H_{\alpha(g_2-\bar{\mu}v_2)}T_{\lambda f_2-g_2}+\tilde{T}_{g_{2}-\bar{\mu}v_2}H_{\alpha (\lambda f_2-g_2)}\\
=&H_{v_{1}g_{2}-g_{2} f_{1}+\alpha(g_2-\bar{\mu}v_2)(\lambda f_2-g_2)}.
\end{align*}
In case \textbf{(L2)+(R4)}. Substituting $\brackbinom{V(\bar{u}_{1})_{-}}{v_{1-}}=\lambda \brackbinom{V(\bar{f_{1}})_{-}}{g_{1-}}$
and $\brackbinom{Vg_{1-}}{(\bar{v}_{1})_{-}}=\mu \brackbinom{Vf_{1-}}{(\bar{u}_{1})_{-}}$ into \eqref{WW}
shows that there exists a non-zero constant $\alpha$ such that
\begin{align*}
\brackbinom{V(\bar{f_{1}})_{-}}{g_{1-}}
=\alpha \bigg(\brackbinom{V(\bar{f_{2}})_{-}}{g_{2-}}-\bar{\mu}\brackbinom{V(\bar{u}_{2})_{-}}{v_{2-}}\bigg),\ \
\brackbinom{Vf_{1-}}{(\bar{u}_{1})_{-}}=\bar{\alpha}\bigg(\brackbinom{Vf_{2-}}{(\bar{u}_{2})_{-}}-\bar{\lambda}\brackbinom{Vg_{2-}}{(\bar{v}_{2})_{-}}\bigg).
\end{align*}
Substituting the above conditions into  \eqref{1.2} and \eqref{2.1}, and applying \eqref{eq:2m}, we obtain
\begin{align*}
&H_{\bar{u}_{2}}T_{\bar{f}_{1}}+\tilde{T}_{\bar{v}_{2}}H_{\bar{u}_{1}}-H_{\bar{u}_{1}}T_{\bar{f}_{2}}-\tilde{T}_{\bar{v}_{1}}H_{\bar{u}_{2}}\\
=&H_{\bar{u}_{2}\bar{f}_{1}}-\tilde{T}_{\bar{u}_{2}}H_{\bar{f}_{1}}+\tilde{T}_{\bar{v}_{2}}H_{\bar{u}_{1}}-H_{\bar{u}_{1}}T_{\bar{f}_{2}}-H_{\bar{v}_1\bar{u}_2}+H_{\bar{v}_{1}}T_{\bar{u}_{2}}\\
=&H_{\bar{u}_{2}\bar{f}_{1}}-\tilde{T}_{\bar{u}_{2}}H_{\bar{f}_{1}}+\tilde{T}_{\bar{v}_{2}}H_{\bar{\lambda}\bar{f}_{1}}-H_{\bar{u}_{1}}T_{\bar{f}_{2}}-H_{\bar{v}_1\bar{u}_2}+H_{\mu\bar{u}_{1}}T_{\bar{u}_{2}}\\
=&H_{\bar{u}_{2}\bar{f}_{1}-\bar{v}_1\bar{u}_2}+\tilde{T}_{\bar{\lambda}\bar{v}_2-\bar{u}_{2}}H_{\bar{f}_{1}}+H_{\bar{u}_{1}}T_{\mu \bar{u}_2-\bar{f}_{2}}\\
=&H_{\bar{u}_{2}\bar{f}_{1}-\bar{v}_1\bar{u}_2}+\tilde{T}_{\bar{\lambda}\bar{v}_2-\bar{u}_{2}}H_{\bar{\alpha}(\bar{f}_2-\mu \bar{u}_2)}+H_{\bar{\alpha}(\bar{u}_2-\bar{\lambda}\bar{v}_2)}T_{\mu \bar{u}_2-\bar{f}_{2}}\\
=&H_{\bar{u}_{2}\bar{f}_{1}-\bar{v}_1\bar{u}_2+\bar{\alpha}(\bar{f}_2-\mu \bar{u}_2)(\bar{\lambda}\bar{v}_2-\bar{u}_{2})},
\end{align*}
and
\begin{align*}
&H_{g_{1}}T_{f_{2}}+\tilde{T}_{v_{1}}H_{g_{2}}-H_{g_{2}}T_{f_{1}}-\tilde{T}_{v_{2}}H_{g_{1}}\\
=&H_{g_{1}}T_{f_{2}}+H_{v_{1}g_{2}}-H_{v_1}T_{g_2}-H_{g_{2} f_{1}}+\tilde{T}_{g_{2}}H_{f_{1}}-\tilde{T}_{v_{2}}H_{\bar{\mu}f_{1}}\\
=&H_{g_{1}}T_{f_{2}}+H_{v_{1}g_{2}}-H_{\lambda g_1}T_{g_2}-H_{g_{2} f_{1}}+\tilde{T}_{g_{2}}H_{f_{1}}-\tilde{T}_{v_{2}}H_{\bar{\mu}f_{1}}\\
=&H_{v_{1}g_{2}-g_{2} f_{1}}+H_{g_1}T_{ f_2-\lambda g_2}+\tilde{T}_{g_{2}-\bar{\mu}v_2}H_{f_{1}}\\
=&H_{v_{1}g_{2}-g_{2} f_{1}}+H_{\alpha(g_2-\bar{\mu}v_2)}T_{ f_2-\lambda g_2}+\tilde{T}_{g_{2}-\bar{\mu}v_2}H_{\alpha ( f_2-\lambda g_2)}\\
=&H_{v_{1}g_{2}-g_{2} f_{1}+\alpha(g_2-\bar{\mu}v_2)( f_2-\lambda g_2)}.
\end{align*}
In case \textbf{(L3)+(R4)}. Substituting $\brackbinom{Vf_{2-}}{(\bar{u}_{2})_{-}}=\lambda \brackbinom{Vg_{2-}}{(\bar{v}_{2})_{-}}$
and $\brackbinom{Vg_{1-}}{(\bar{v}_{1})_{-}}=\mu \brackbinom{Vf_{1-}}{(\bar{u}_{1})_{-}}$ into \eqref{WW}
shows that there exists a non-zero constant $\alpha$ such that
\begin{align*}
\bar{\lambda}\brackbinom{V(\bar{f_{1}})_{-}}{g_{1-}}-\brackbinom{V(\bar{u}_{1})_{-}}{v_{1-}}
=\alpha \bigg(\brackbinom{V(\bar{f_{2}})_{-}}{g_{2-}}-\bar{\mu}\brackbinom{V(\bar{u}_{2})_{-}}{v_{2-}}\bigg),\ \
\brackbinom{Vf_{1-}}{(\bar{u}_{1})_{-}}=\bar{\alpha}\brackbinom{Vg_{2-}}{(\bar{v}_{2})_{-}}.
\end{align*}
Substituting the above conditions into  \eqref{1.2} and \eqref{2.1}, and applying \eqref{eq:2m}, we obtain
\begin{align*}
&H_{\bar{u}_{2}}T_{\bar{f}_{1}}+\tilde{T}_{\bar{v}_{2}}H_{\bar{u}_{1}}-H_{\bar{u}_{1}}T_{\bar{f}_{2}}-\tilde{T}_{\bar{v}_{1}}H_{\bar{u}_{2}}\\
=&H_{\lambda \bar{v}_{2}}T_{\bar{f}_{1}}+H_{\bar{v}_{2}\bar{u}_{1}}-H_{\bar{v}_{2}}T_{\bar{u}_{1}}-H_{\bar{u}_{1}}T_{\bar{f}_{2}}-H_{\bar{v}_{1}\bar{u}_{2}}+H_{\bar{v}_{1}}T_{\bar{u}_{2}}\\
=&H_{\lambda \bar{v}_{2}}T_{\bar{f}_{1}}+H_{\bar{v}_{2}\bar{u}_{1}}-H_{\bar{v}_{2}}T_{\bar{u}_{1}}-H_{\bar{u}_{1}}T_{\bar{f}_{2}}-H_{\bar{v}_{1}\bar{u}_{2}}+H_{\mu\bar{u}_{1}}T_{\bar{u}_{2}}\\
=&H_{\bar{v}_{2}\bar{u}_{1}-\bar{v}_{1}\bar{u}_{2}}+H_{\bar{v}_{2}}T_{\lambda \bar{f}_1-\bar{u}_{1}}+H_{\bar{u}_1}T_{\mu \bar{u}_2 -\bar{f}_2}\\
=&H_{\bar{v}_{2}\bar{u}_{1}-\bar{v}_{1}\bar{u}_{2}}+H_{\bar{v}_{2}(\lambda \bar{f}_1-\bar{u}_{1})}+\tilde{T}_{\bar{v}_{2}}H_{\bar{\alpha} (\mu \bar{u}_2-\bar{f}_2)}+H_{\bar{\alpha} \bar{v}_2}T_{\mu \bar{u}_2 -\bar{f}_2}\\
=&H_{\bar{v}_{2}\bar{u}_{1}-\bar{v}_{1}\bar{u}_{2}+\bar{v}_{2}(\lambda \bar{f}_1-\bar{u}_{1})+\bar{\alpha}\bar{v}_{2} (\mu \bar{u}_2-\bar{f}_2)}\\
=&H_{-\bar{v}_{1}\bar{u}_{2}+\lambda \bar{v}_{2} \bar{f}_1+\bar{\alpha}\bar{v}_{2} (\mu \bar{u}_2-\bar{f}_2)},
\end{align*}
and
\begin{align*}
&H_{g_{1}}T_{f_{2}}+\tilde{T}_{v_{1}}H_{g_{2}}-H_{g_{2}}T_{f_{1}}-\tilde{T}_{v_{2}}H_{g_{1}}\\
=&H_{g_{1}f_2}-\tilde{T}_{g_{1}}H_{f_{2}}+\tilde{T}_{v_{1}}H_{g_{2}}-H_{g_2 f_1}+\tilde{T}_{g_{2}}H_{f_{1}}-\tilde{T}_{v_{2}}H_{\bar{\mu}f_{1}}\\
=&H_{g_{1}f_2}-\tilde{T}_{g_{1}}H_{\bar{\lambda} g_{2}}+\tilde{T}_{v_{1}}H_{g_{2}}-H_{g_2 f_1}+\tilde{T}_{g_{2}}H_{f_{1}}-\tilde{T}_{v_{2}}H_{\bar{\mu}f_{1}}\\
=&H_{g_{1}f_2-g_2 f_1}+\tilde{T}_{v_1-\bar{\lambda}g_{1}}H_{g_{2}}+\tilde{T}_{g_{2}-\bar{\mu}v_2}H_{f_{1}}\\
=&H_{g_{1}f_2-g_2 f_1}+H_{(v_1-\bar{\lambda}g_{1})g_2}-H_{v_1-\bar{\lambda}g_{1}}T_{g_{2}}+\tilde{T}_{g_{2}-\bar{\mu}v_2}H_{\alpha g_{2}}\\
=&H_{g_{1}f_2-g_2 f_1}+H_{(v_1-\bar{\lambda}g_{1})g_2}+H_{\alpha (g_2-\bar{\mu}v_2)}T_{g_{2}}+\tilde{T}_{g_{2}-\bar{\mu}v_2}H_{\alpha g_{2}}\\
=&H_{g_{1}f_2-g_2 f_1+(v_1-\bar{\lambda}g_{1})g_2+\alpha g_{2}(g_2-\bar{\mu}v_2)}.
\end{align*}
In case \textbf{(L4)+(R4)},
Substituting $\brackbinom{Vg_{2-}}{(\bar{v}_{2})_{-}}=\lambda \brackbinom{Vf_{2-}}{(\bar{u}_{2})_{-}}$
and $\brackbinom{Vg_{1-}}{(\bar{v}_{1})_{-}}=\mu \brackbinom{Vf_{1-}}{(\bar{u}_{1})_{-}}$ into \eqref{WW}
shows that there exists a non-zero constant $\alpha$ such that
\begin{align*}
\brackbinom{V(\bar{f_{1}})_{-}}{g_{1-}}-\bar{\lambda}\brackbinom{V(\bar{u}_{1})_{-}}{v_{1-}}
=\alpha \bigg(\brackbinom{V(\bar{f_{2}})_{-}}{g_{2-}}-\bar{\mu}\brackbinom{V(\bar{u}_{2})_{-}}{v_{2-}}\bigg),\ \
\brackbinom{Vf_{1-}}{(\bar{u}_{1})_{-}}=\bar{\alpha}\brackbinom{Vf_{2-}}{(\bar{u}_{2})_{-}}.	
\end{align*}
Substituting the above conditions into  \eqref{1.2} and \eqref{2.1}, and applying \eqref{eq:2m}, we obtain
\begin{align*}  
&H_{\bar{u}_{2}}T_{\bar{f}_{1}}+\tilde{T}_{\bar{v}_{2}}H_{\bar{u}_{1}}-H_{\bar{u}_{1}}T_{\bar{f}_{2}}-\tilde{T}_{\bar{v}_{1}}H_{\bar{u}_{2}}\\
=&H_{\bar{u}_{2}}T_{\bar{f}_{1}}+H_{\bar{v}_{2}\bar{u}_{1}}-H_{\bar{v}_{2}}T_{\bar{u}_{1}}-H_{\bar{u}_{1}}T_{\bar{f}_{2}}-H_{\bar{v}_{1}\bar{u}_{2}}+H_{\bar{v}_{1}}T_{\bar{u}_{2}}\\
=&H_{\bar{u}_{2}}T_{\bar{f}_{1}}+H_{\bar{v}_{2}\bar{u}_{1}}-H_{\lambda \bar{u}_{2}}T_{\bar{u}_{1}}-H_{\bar{u}_{1}}T_{\bar{f}_{2}}-H_{\bar{v}_{1}\bar{u}_{2}}+H_{\mu\bar{u}_{1}}T_{\bar{u}_{2}}\\
=&H_{\bar{v}_{2}\bar{u}_{1}-\bar{v}_{1}\bar{u}_{2}}+H_{\bar{u}_{2}}T_{\bar{f}_1-\lambda\bar{u}_{1}}+H_{\bar{u}_1}T_{\mu \bar{u}_2-\bar{f}_2}\\
=&H_{\bar{v}_{2}\bar{u}_{1}-\bar{v}_{1}\bar{u}_{2}}+H_{\bar{u}_2(\bar{f}_1-\lambda\bar{u}_{1})}-\tilde{T}_{\bar{u}_{2}}H_{\bar{f}_1-\lambda\bar{u}_{1}}+H_{\bar{\alpha}\bar{u}_2}T_{\mu \bar{u}_2-\bar{f}_2}\\
=&H_{\bar{v}_{2}\bar{u}_{1}-\bar{v}_{1}\bar{u}_{2}}+H_{\bar{u}_{2}(\bar{f}_1-\lambda\bar{u}_{1})}+\tilde{T}_{\bar{u}_{2}}H_{\bar{\alpha} (\mu \bar{u}_2-\bar{f}_2)}+H_{\bar{\alpha} \bar{u}_2}T_{\mu \bar{u}_2 -\bar{f}_2}\\
=&H_{\bar{v}_{2}\bar{u}_{1}-\bar{v}_{1}\bar{u}_{2}+\bar{u}_{2}(\bar{f}_1-\lambda \bar{u}_{1})+\bar{\alpha}\bar{u}_{2} (\mu \bar{u}_2-\bar{f}_2)},
\end{align*}   
and
\begin{align*}
&H_{g_{1}}T_{f_{2}}+\tilde{T}_{v_{1}}H_{g_{2}}-H_{g_{2}}T_{f_{1}}-\tilde{T}_{v_{2}}H_{g_{1}}\\
=&H_{g_{1}f_2}-\tilde{T}_{g_{1}}H_{f_{2}}+\tilde{T}_{v_{1}}H_{g_{2}}-H_{g_2 f_1}+\tilde{T}_{g_{2}}H_{f_{1}}-\tilde{T}_{v_{2}}H_{\bar{\mu}f_{1}}\\
=&H_{g_{1}f_2}-\tilde{T}_{g_{1}}H_{f_{2}}+\tilde{T}_{v_{1}}H_{\bar{\lambda} f_{2}}-H_{g_2 f_1}+\tilde{T}_{g_{2}}H_{f_{1}}-\tilde{T}_{v_{2}}H_{\bar{\mu}f_{1}}\\
=&H_{g_{1}f_2-g_2 f_1}+\tilde{T}_{\bar{\lambda} v_1-g_{1}}H_{f_{2}}+\tilde{T}_{g_{2}-\bar{\mu}v_2}H_{f_{1}}\\
=&H_{g_{1}f_2-g_2 f_1}+H_{(\bar{\lambda} v_1-g_{1})f_2}-H_{\bar{\lambda} v_1-g_{1}}T_{f_{2}}+\tilde{T}_{g_{2}-\bar{\mu}v_2}H_{\alpha f_{2}}\\
=&H_{g_{1}f_2-g_2 f_1}+H_{(\bar{\lambda}v_1-g_{1})f_2}+H_{\alpha (g_2-\bar{\mu}v_2)}T_{f_{2}}+\tilde{T}_{g_{2}-\bar{\mu}v_2}H_{\alpha f_{2}}\\
=&H_{g_{1}f_2-g_2 f_1+(\bar{\lambda}v_1-g_{1})f_2+\alpha f_{2}(g_2-\bar{\mu}v_2)}\\
=&H_{-g_2 f_1+\bar{\lambda}v_1f_2+\alpha f_{2}(g_2-\bar{\mu}v_2)}.
\end{align*}

\bigskip \noindent{\bf Acknowledgement}.
The author is grateful to  Hui Dan for his invaluable suggestions.

\textbf{Declarations}

Conficts of interest /Competing interests: The authors declare no confict of interests.\\
Availability of data and material: The data and material are transparent.\\
Code availability (software application or custom code): Not applicable.\\
Ethics approval (include appropriate approvals or waivers): Not applicable.\\
Consent to participate (include appropriate statements): The authors consent to participate.\\
Consent for publication (include appropriate statements): The authors consent for publication.\\
\end{document}